\documentclass[10pt,a4paper,final]{article}
\usepackage[latin1]{inputenc}
\usepackage{amsmath}
\usepackage{amsfonts}
\usepackage{amssymb}
\usepackage{makeidx}
\usepackage[latin1]{inputenc}
\usepackage{amsmath}
\usepackage{amsfonts}
\usepackage{amssymb}
\usepackage{amsthm}
\usepackage{makeidx}
\usepackage{graphicx}

\usepackage{latexsym}
\usepackage{cite}
\usepackage{mathrsfs}
\usepackage{bbm}
\usepackage{bbold}
\usepackage{subfigure}
\usepackage{float}
\usepackage{epsfig}
\usepackage{epstopdf}
\usepackage{mathtools}
\usepackage{listings}
\usepackage{graphicx}
\usepackage{textcomp }
\usepackage{todo}
\usepackage{pstricks}
\usepackage{amssymb}
\usepackage{amsmath}
\usepackage{latexsym}
\usepackage{geometry}
\usepackage{tikz}
\usepackage{amssymb}
\usepackage{amsthm}
\usepackage{cite}
\usepackage{mathrsfs}
\usepackage{bbm}
\usepackage{bbold}
\usepackage{subfigure}
\usepackage{float}
\usepackage{epsfig}
\usepackage[numbers]{natbib}
\usepackage{mathtools}
\usepackage{listings}
\usepackage{pgf}
\usepackage{listings}
\usepackage{epstopdf}

\theoremstyle{plain} 
\newtheorem{thm}{Theorem}[section] 
 
\newtheorem{prop}{Proposition}[section]

\newtheorem{lem}{Lemma}[section] 
\DeclareMathOperator{\Bin}{Bin}

\theoremstyle{remark} 
\newtheorem{obs}{Remark}

\newcommand{\N}{{\mathbb{N}}}   

\newcommand{\CC}{\mathcal{C}}

\title{Small maximal clusters are very unlikely in critical random graphs}

\author{Umberto De Ambroggio\thanks{Ludwig Maximilians Universit\"at M\"unchen, Department of Mathematical Sciences. \texttt{umbidea@gmail.com}}}

\begin{document}
	
	\maketitle

	\begin{abstract}
	We describe a probabilistic methodology, based on random walk estimates, to obtain \textit{exponential} upper bounds for the probability of observing unusually \textit{small} maximal components in two classical (near-)critical random graph models. More specifically, we analyse the near-critical Erd\H{o}s-R\'enyi model $\mathbb{G}(n,p)$ and the random graph $\mathbb{G}(n,d,p)$ obtained by performing near-critical $p$-bond percolation on a simple random $d$-regular graph and show that, for each one of these models, the probability that a largest component contains less than $n^{2/3}/A$ vertices is at most $\exp(-\Omega(A^{3/2}))$. The exponent $3/2$ is known to be \textit{optimal} for the near-critical $\mathbb{G}(n,p)$ random graph, whereas for the near-critical $\mathbb{G}(n,d,p)$ model the best known upper bound for the above probability was of (polynomial) order $O(A^{-3/5})$. As a secondary result we show, by means of an optimized version of the martingale method of Nachmias and Peres, that the above probability of observing an unusually small maximal component is at most $\exp(-\Omega(A^{3/5}))$ in other two critical models, namely a random intersection graph and the quantum random graph. These stretched-exponential bounds also improve upon the known (polynomial) upper bounds available for these other two critical models.
	\end{abstract}

\section{Introduction}\label{intro}
In recent years, the problem of deriving exponential bounds for the probability of observing unusually \textit{large} maximal components in critical random graphs has gained increasing interest; see e.g. \cite{pittel:largest_cpt_rg, hofstad_et_al:cluster_tails_power_law_RGs,roberts:component_ER,de_ambroggio_roberts:near_critical_ER,de_ambroggio_roberts:near_critical_RRG} as well as Section 7.1.1 in the PhD thesis of Dhara \cite{dharaphdthesis}.

In this work we consider the complementary and less investigated problem of determining \textit{exponential} upper bounds for the probability of observing unusually \textit{small} maximal components in critical random graphs. In other words, we are interested in providing exponential (upper) bounds for probabilities of the type
\begin{equation}\label{PROBLEM}
\mathbb{P}(|\mathcal{C}_{\max}(\mathbb{G})|<k),
\end{equation}
where $\mathcal{C}_{\max}(\mathbb{G})$ is a maximal cluster of some (near-)critical random graph $\mathbb{G}$ and $k\in \mathbb{N}$ (for us, $k=\Theta(n^{2/3}/A)$, with $n^{2/3}$ being the correct order of $|\mathcal{C}_{\max}(\mathbb{G})|$ in the regime we are interested in for the random graphs under examination). 

To the best of our knowledge, the only other work we are aware of that investigates this problem is that of Pittel \cite{pittel:largest_cpt_rg} in the context of the (near-critical) Erd\H{o}s-R\'enyi model $\mathbb{G}(n,p)$. Before looking at Pittel's result, let us recall the definition of the $\mathbb{G}(n,p)$ model. 

This random graph, which is obtained by performing $p$-bond percolation on the complete graph
on $n$ vertices (which means that it is constructed by independently retaining each edge with probability $p$ and deleting it with probability $1-p$), is known to undergo a phase transition as $np$ passes 1. Specifically, letting $p=p(n)\coloneqq \mu/n$, the following statements hold with probability tending to one as the number of vertices grows indefinitely: if $\mu < 1$ then $|\mathcal C_{\max}(\mathbb{G}(n,p))|$ is $O(\log n)$; if $\mu=1$ (the \textit{critical} case), then $|\mathcal C_{\max}(\mathbb{G}(n,p))|$ is $\Theta(n^{2/3})$; and if $\mu > 1$, then $|\mathcal C_{\max}(\mathbb{G}(n,p))|$ is $\Theta(n)$. See for instance the monographs \cite{bollobas_book}, \cite{remco:random_graphs} or \cite{janson_et_al:random_graphs} for more details. 

Actually we know more: if $\mu=\mu(n)=1+\lambda n^{-1/3}$ then $|\mathcal C_{\max}|$ is still $\Theta(n^{2/3})$ \cite{remco:random_graphs}; this is the so-called \textit{critical window}.
Pittel \cite{pittel:largest_cpt_rg} considered this regime and he showed (among many other things) that, in the near-critical Erd\H{o}s-R\'enyi random graph $\mathbb{G}(n,p)$ with $p=p(n)=(1+\lambda n^{-1/3})/n$ and $\lambda \in \mathbb{R}$, as $A\rightarrow \infty$
\begin{equation}\label{pittellower}
\lim_{n\to\infty}\mathbb{P}(|\mathcal C_{\max}(\mathbb{G}(n,p))|< n^{2/3}/A)=C(\lambda)\exp\Big(-c_1\lambda^2 A^{1/2}-c_2\lambda A-c_2A^{3/2}\Big),
\end{equation}
where $c_1,c_2$ and $c_3$ are positive constants (independent of $\lambda$) and $C(\lambda)$ is a positive constant which depends on $\lambda$. The asymptotic in (\ref{pittellower}) shows that, for large $A$ and $n$, 
\[-\log\Big(\mathbb{P}(|\mathcal C_{\max}(\mathbb{G}(n,p))|< n^{2/3}/A)\Big)=\Theta(A^{3/2}),\]
which we also write as
\[\mathbb{P}(|\mathcal C_{\max}(\mathbb{G}(n,p))|< n^{2/3}/A)=\exp(-\Theta(A^{3/2})).\]
We refer the reader to Corollary $2$ in Pittel's paper \cite{pittel:largest_cpt_rg} for the definition (and numerical estimates) of the constants appearing in (\ref{pittellower}).

To the best of our knowledge, with the exception of (\ref{pittellower}), all other upper bounds available in the literature for the probability of observing \textit{small} maximal clusters in (near-)critical random graphs are of much weaker order, namely the type $A^{-\gamma}$, where $\gamma \in (0,1)$. 

For instance, Nachmias and Peres \cite{nachmias:critical_perco_rand_regular} used martingale arguments to analyse the component structure in the random graph obtained by performing $p$-bond percolation on a random $d$-regular graph. This random graph, denoted by $\mathbb{G}(n,d,p)$, is the random graph on $n$ vertices obtained by first drawing uniformly at random a $d$-regular simple graph on $[n]$ and then performing independent $p$-bond percolation on it. 

Alon, Benjamini and Stacey \cite{alon_benj_stacey} showed that $\mathbb{G}(n,d,\mu/(d-1))$ has a phase transition similar to the one observed in the $\mathbb{G}(n,p)$ model. Indeed, $\mathbb{G}(n,d,\mu/(d-1))$ undergoes a phase transition as $\mu$ passes $1$: the size of the largest component $|\mathcal C_{\max}(\mathbb{G}(n,d,\mu/(d-1)))|$ is of order $\log n$ when $\mu < 1$, and of order $n$ when $\mu > 1$.

Nachmias and Peres \cite{nachmias:critical_perco_rand_regular} analysed (using probabilistic arguments) the critical window of this model and showed that, when $p=p(n,d)=(1+\lambda n^{-1/3})/(d-1)$ with $\lambda\in \mathbb{R}$ and $d\geq 3$ both fixed, there exists a positive constant $c=c(\lambda,d)$ which depends on $\lambda,d$ such that, for all large enough $A>0$ and $n$, 
\begin{equation}\label{boundperc}
\mathbb{P}\left(|\mathcal C_{\max}(\mathbb{G}_{n,d,p})|< n^{2/3}/A\right)\leq  cA^{-1/2}.
\end{equation}
More recently, Joos and Perarnau \cite{joospera} showed that $|\mathcal{C}_{\max}(\mathbb{G}(n,d,p))|$ is of order $n^{2/3}$ \textit{for all} $d=d(n)\in \{3,\dots,n-1\}$ when $p=(1+\lambda n^{-1/3})/(d-1)$ (with $\lambda\in \mathbb{R}$ fixed), extending the result of Nachmias and Peres \cite{nachmias:critical_perco_rand_regular} and confirming
a prediction of the two authors on a question of Benjamini. In particular, Joos and Perarnau \cite{joospera} showed that (in the above setting)
\[\mathbb{P}\left(n^{2/3}/A<|\mathcal C_{\max}(\mathbb{G}(n,p))|< An^{2/3}\right)\geq 1-cA^{-1/2}.\]
We remark that the (probabilistic) methodology used by Nachmias and Peres \cite{nachmias:critical_perco_rand_regular} was introduced by the same authors in \cite{nachmias_peres:CRG_mgs}, with the purpose of giving simple proofs that $|\mathcal{C}_{\max}(\mathbb{G}(n,p))|$ is of order $n^{2/3}$ when $p=(1+\lambda n^{-1/2})/n$. They showed (among other things) that when $p=1/n$ then
\begin{equation}\label{bounder}
\mathbb{P}\left(|\mathcal C_{\max}(\mathbb{G}(n,p))|< n^{2/3}/A\right)\leq cA^{-3/5}
\end{equation}
for some constant $c>0$ and for all sufficiently large $A$ and $n$. (In \cite{nachmias_peres:CRG_mgs} the authors also considered the probability of observing large maximal clusters and $p$ of of the form $p=(1+\lambda n^{-1/3})/n$ with $\lambda\in \mathbb{R}$, see Section \ref{relatedwork} below.).

The martingale method of Nachmias and Peres \cite{nachmias_peres:CRG_mgs,nachmias:critical_perco_rand_regular}, which lead to the bounds displayed in (\ref{boundperc}) and (\ref{bounder}), has been successfully employed by other authors to analyse the (near-)critical behaviour of several other random graphs, leading to polynomial (upper) bounds of the above.

For instance, Hatami and Molloy \cite{hatami_molloy_conf} adapted such a methodology to determine the critical window for a random graph on a given degree sequence, whereas Dembo, Levit and Vadlamani \cite{dembo_et_al:component_sizes_quantum_RG} used the method developed by Nachmias and Peres to study the near-critical behaviour of the \textit{quantum} random graph (see Section \ref{heurmg} for a brief introduction to this model). Furthermore, the same technique was also used by the author and Pachon \cite{de_ambroggio_pachon:upper_bounds_inhom_RGs} to analyse (a specific instance of) the critical Norros-Reittu model \cite{NRn}, whose near-critical behaviour was first established by van der Hofstad \cite{hofstad_critic}.

\subsection{Main result}
Pittel's proof of (\ref{pittellower}) heavily relies on combinatorial arguments and seems difficult to be adapted to other models. On the other hand, the probabilistic argument of Nachmias and Peres \cite{nachmias_peres:CRG_mgs,nachmias:critical_perco_rand_regular}, although very general and robust, does not seem sufficient to retrieve the optimal exponent $3/2$ which appears in the leading term of (\ref{pittellower}). Indeed, although we can boost (by means of a few modifications) the original martingale argument \cite{nachmias_peres:CRG_mgs,nachmias:critical_perco_rand_regular} to obtain upper bounds of the type $\exp(-\Omega(A^{3/5}))$ in several critical models, it seems that the method can't be pushed to achieve the (correct) decay $\exp(-\Omega(A^{3/2}))$, even for the critical $\mathbb{G}(n,p)$ model. We refer the reader to Section \ref{heurmg} for a detailed explanation concerning why we believe this to be the case.

The goal of this work is to introduce a robust probabilistic methodology, simply based on random walk estimates, to derive \textit{sharp} exponential upper bounds for the probability of observing unusually \textit{small} maximal clusters in (near-)critical random graphs. We do so by analyzing the near-critical models $\mathbb{G}(n,p)$ and $\mathbb{G}(n,d,p)$ and show that $|\mathcal{C}_{\max}(\mathbb{G})|< n^{2/3}/A$ with probability at most $\exp(-\Omega(A^{3/2}))$ when either $\mathbb{G}=\mathbb{G}(n,p)$ or $\mathbb{G}=\mathbb{G}(n,d,p)$. Here is our main result.

\begin{thm}\label{mainthm}
	Let $p=p(n)\coloneqq (1+\lambda n^{-1/3})/n$ with $\lambda\in \mathbb{R}$ fixed. There exist constants $A_0=A_0(\lambda)>0$ and $c>0$ such that, for any $A_0\leq A=A(n)=o(n^{2/3})$ and for all large enough $n$, we have
	\begin{equation}\label{mainres}
	\mathbb{P}(|\mathcal{C}_{\max}(\mathbb{G}(n,p))|<n^{2/3}/A)\leq \exp\big(-cA^{3/2}\big).
	\end{equation}
	Let $p=p(d,n)\coloneqq (1+\lambda n^{-1/3})/(d-1)$ with $\lambda\in \mathbb{R},3\leq d\in \mathbb{N}$ both fixed. There exist constants $A_0=A_0(\lambda,d)>0$ and $c=c(d)>0$ such that, for any $A_0\leq A=A(n)=O(n^{1/6})$ and for all large enough $n$ (with $dn$ even), we have
	\begin{equation}\label{mainres2}
	\mathbb{P}(|\mathcal{C}_{\max}(\mathbb{G}(n,d,p))|<n^{2/3}/A)\leq \exp\big(-cA^{3/2}\big).
	\end{equation}
\end{thm}

\begin{obs}
	We emphasizes once more that (\ref{mainres}) is not new (see (\ref{pittellower}) above), but our proof is purely \textit{probabilistic}. On the other hand, (\ref{mainres2}) is new and most likely captures the correct decay; moreover, it substantially improves upon the known bound displayed in (\ref{boundperc}). We remark that Theorem \ref{mainthm} only considers the near-critical models $\mathbb{G}(n,p)$ and $\mathbb{G}(n,d,p)$, but we believe that the argument used to establish (\ref{mainres}) and (\ref{mainres2}) is robust. One could try to adapt the methodology introduced in this work to prove similar tail bounds in other (near-)critical random graphs. We refer the reader to Section \ref{heurmg} for a detailed discussion (and related results) along these lines.
\end{obs}

\begin{obs}
	We emphasize that, in Theorem \ref{mainthm}, $A$ is allowed to depend on $n$. The condition $A=O(n^{1/6})$, which we require in our upper bound for the $\mathbb{G}(n,d,p)$ random graph could be changed to $A=o(n^{2/3})$ at the cost of adding a term of the type $\exp(-\Omega(A^{1/2}n^{1/6}))$ to (\ref{mainres2}). Moreover, the condition $A=o(n^{2/3})$ which we require in our upper bound for the $\mathbb{G}(n,p)$ model could be relaxed to $A\leq \varepsilon n^{2/3}$ for a sufficiently small constant $\varepsilon>0$. However, when $A=\Theta( n^{2/3})$, then $n^{2/3}/A$ is of constant order and so we would be bounding from above the probability that the size of a maximal cluster is at most some \textit{constant}, which is not very interesting. Therefore we can say that our upper bound in (\ref{mainres}) concerns \textit{all} the relevant values of $A$. Indeed, we believe the bounds in (\ref{mainres}) and (\ref{mainres2}) to be more interesting for `small' (possibly dependent of $n$) values of $A$, being $n^{2/3}$ the correct order for the size of a largest cluster in both (near-)critical models. We also remark that the parameter $\lambda$, which we consider as fixed, could be allowed to depend on $n$ at the expense of more careful computations. In this case, one would need to keep track of the dependence between $A$ and $\lambda$. Furthermore, if needed, one could give bounds on the `smallest values' of $A_0$ for which  (\ref{mainres}) and (\ref{mainres2}) hold; similarly, one might try to give bounds on the smallest $n$ for which (\ref{mainres}) and (\ref{mainres2}) are verified (as it is done e.g. in \cite{nachmias_peres:CRG_mgs}). However, we refrained to do so in order to keep the computations shorter. 
\end{obs}

\subsection{Related work}\label{relatedwork}
There are several works focusing on the problem of deriving sharp tail bounds for the probability of observing unusually \textit{large} clusters.

Concerning the Erd\H{o}s-R\'enyi random graph $\mathbb{G}(n,p)$, Pittel \cite{pittel:largest_cpt_rg} also showed that, in the near-critical regime $p=(1+\lambda n^{-1/3})/n$ with $\lambda\in\mathbb{R}$, then
\begin{equation}\label{pittelprob}
\lim_{n\to\infty}A^{3/2}\exp\Big(\frac{A^3}{8}-\frac{\lambda A^2}{2}+\frac{\lambda^2A}{2}\Big)\mathbb{P}(|\mathcal C_{\max}(\mathbb{G}(n,p))|> An^{2/3} )
\end{equation}
converges, as $A\to\infty$, to a specific constant, which is stated to be $(2\pi)^{-1/2}$ but should be $(8/9\pi)^{1/2}$, as remarked by Roberts \cite{roberts:component_ER}. 

Along the same lines, van der Hofstad, Kleim and Van
Leeuwaarden \cite{hofstad_et_al:cluster_tails_power_law_RGs} proved similar results in the context of inhomogeneous random graphs.

Both Pittel \cite{pittel:largest_cpt_rg} and Roberts \cite{roberts:component_ER} relied on a combinatorial formula for the expected number of components with $k$ vertices and $k+\ell$ edges, which is specific to the $\mathbb{G}_{n,p}$ model and appears difficult to adapt to other random graphs. 

A first, purely probabilistic approach to establish an exponential upper bound for the probability displayed in (\ref{pittelprob}) was introduced by Nachmias and Peres \cite{nachmias_peres:CRG_mgs}, who used martingale arguments to show that for all large enough $A$ and $n$, when $p=1/n$ then
\begin{equation}\label{boundnp1}
\mathbb{P}(|\mathcal{C}_{\max}(\mathbb{G}(n,p))|>An^{2/3})\leq \frac{4}{A}\exp\big(-A^2(A-4)/32\big).
\end{equation}
They also considered the near-critical model where $p = (1+\lambda n^{-1/3})/n$ for fixed $\lambda\in\mathbb{R}$ and established a similar exponential upper bound.

The same martingale-based method was subsequently used by the same authors in \cite{nachmias:critical_perco_rand_regular}, where they analysed the (near-critical) $\mathbb{G}(n,d,p)$ model.

Amongst other results (some of which were discussed earlier), Nachmias and Peres \cite{nachmias:critical_perco_rand_regular} proved that, in the near-critical regime where $p=(1+\lambda n^{-1/3})/(d-1)$ and $\lambda \in \mathbb{R},3\leq d\in \mathbb{N}$ are both \textit{fixed}, there are positive constants $c=c(\lambda ,d)$ and $C=C(\lambda,d)$ which depend on both $\lambda$ and $d$ such that, for any $A>0$ and for all large enough $n$ (with $dn$ even),
\begin{equation}\label{boundnp2}
\mathbb{P}\left(|\mathcal C_{\max}(\mathbb{G}(n,d,p))|>An^{2/3}\right)\leq \frac{C}{A}\exp\big(-cA^3\big).
\end{equation} 
With the purpose of introducing a robust \textit{probabilistic} methodology capable of recovering the correct asymptotic of Pittel displayed in (\ref{pittelprob}), more recently the author and Roberts \cite{de_ambroggio_roberts:near_critical_ER} introduced a novel method, based on (probabilistic) tools such as generalized ballot theorems \cite{addario_berry_reed:ballot_theorems} and strong embedding with Brownian motion \cite{chatterjee:strong_embeddings}, to derive \textit{matching} upper and lower bounds for the probability of observing unusually \textit{large} maximal clusters in critical random graphs.

In particular, for the near-critical $\mathbb{G}(n,p)$ model with $p=(1+\lambda n^{-1/3})/n$ and $\lambda \in \mathbb{R}$ (possibly dependent on $n$), the authors showed in \cite{de_ambroggio_roberts:near_critical_ER} that there exist constants $A_0>0,n_0\in \mathbb{N}$ such that, if $A_0\le A=A(n) =o(n^{1/30}),n\geq n_0$ and $\lambda = \lambda(n)$ satisfies $|\lambda|\leq A/3$, then
\begin{equation}\label{ourbound1}
 \mathbb{P}(|\mathcal{C}_{\max}(\mathbb{G}(n,p))|> An^{2/3} )=\Theta\Big( \frac{1}{A^{3/2}}\exp\big(-\frac{A^3}{8}+\frac{\lambda A^2}{2}-\frac{\lambda^2A}{2}\big)\Big),
\end{equation}
thus improving upon the Nachmias and Peres bound given in (\ref{boundnp1}) and retrieving the asymptotic result of Pittel displayed in (\ref{pittelprob}).

Subsequently, the methodology which lead to (\ref{ourbound1}) was adapted by the same authors to study the near-critical $\mathbb{G}(n,d,p)$ random graph where $p=(1+\lambda n^{-1/3})/(d-1)$ with $\lambda \in \mathbb{R}$ (possibly dependent on $n$) and $d\geq 3$ \textit{fixed}. For this model, it was shown in \cite{de_ambroggio_roberts:near_critical_RRG} that there exist constants $A_0>0,n_0\in \mathbb{N}$ such that, if $A_0\le A=A(n) =o(n^{1/30}),n\geq n_0$ and $\lambda=\lambda(n)$ satisfies $|\lambda|\leq A(1-2/d)[3(d-1)]^{-1}$, then
\begin{equation*}\label{ourbound2}
\mathbb{P}(|\mathcal{C}_{\max}(\mathbb{G}(n,d,p))|>An^{2/3})=\Theta\Big( \frac{1}{A^{3/2}}\exp\big(-\frac{A^3(d-1)(d-2)}{8d^2}+\frac{\lambda A^2(d-1)}{2d}-\frac{\lambda^2 A(d-1)}{2(d-2)}\big)\Big),
\end{equation*}
thus considerably improving upon the bound displayed in (\ref{boundnp2}). 

\begin{obs}
	Observe that, in the particular case where $\lambda=0$ (that is, in the purely critical regime), the constant $c(d)\coloneqq (d-1)(d-2)/d^2$ which multiplies the leading term $A^3/8$ in the exponential above tends to $1$ as $d\rightarrow \infty$. We suspect that this phenomenon is true more generally. Specifically, let $\mathbb{G}$ be a random graph that, when considered at criticality of its parameter(s), exhibits maximal components containing $\Theta(n^{2/3})$ vertices. We believe that, at criticality, there exist constants $c_1,c_2>0$ with $c_1<c_2$ such that, for all large enough $n$ and (not `too large') $A=A(n)$, then
	\[\mathbb{P}(|\mathcal{C}_{\max}(\mathbb{G})|>An^{2/3})=\Theta\big( A^{-3/2}e^{-C^*\frac{A^3}{8}}\big),\]
	where $C^*=1$ if $\mathbb{G}=\mathbb{G}(n,1/n)$, otherwise $C^*$ is a constant which depends on some parameter $\kappa$ related to the random graph $\mathbb{G}$ under consideration satisfying $C^*(\kappa)\rightarrow 1$ as $\kappa\rightarrow \infty$. Moreover, we further believe that (at criticality) there are constants $C,c>0$ such that, for all large enough $n$ and (not `too large') $A=A(n)$,
	\[\mathbb{P}(|\mathcal{C}_{\max}(\mathbb{G})|<n^{2/3}/A)\leq e^{-cA^{3/2}}.\]
\end{obs}

\subsection{Notation and structure of the paper}
Here we introduce the notation used throughout the article and describe the structure of the rest of the paper.

\paragraph{Notation.} We let $\mathbb{N}=\{1,2,\dots,\}$ and define $\N_0 \coloneqq \N\cup\{0\}$. For $n\in \mathbb{N}$ we set $[n]\coloneqq \{1,2,\ldots,n\}$. Given two sequences of positive real numbers $(x_n)_{n\geq 1}$ and $(y_n)_{n\geq 1}$ we write: $(1)$ $x_n=O(y_n)$ if there exist $N\in \mathbb{N}$ and $C\in (0,\infty)$ such that $x_n\leq C y_n$ for all $n\geq N$; (2) either $x_n=o(y_n)$ or $x_n\ll y_n$ if $x_n/y_n\rightarrow 0$ as $n\rightarrow \infty$; (3) $x_n=\Omega(y_n)$ if $y_n=O(x_n)$; and (4) $x_n=\Theta(y_n)$ if $x_n=O(y_n)$ and $x_n=\Omega(y_n)$ (in which case $y_n=\Theta(x_n)$ too). We often write $c,C$ to denote constants in $(0,\infty)$, and use these letters many times in a single proof even though their actual values may change from line to line. 

When talking about random variables, the abbreviation iid means \textit{independent and identically distributed}. Moreover, given random variables $X$ and $Y$, we write $X\leq_{sd}Y$ if $\mathbb{P}(X\geq z)\leq \mathbb{P}(Y\geq z)$ for every $z\in \mathbb{R}$, and say that $X$ is stochastically dominated by $Y$. We write $\text{Bin}(N,q)$ to denote the binomial distribution of parameters $N\in \mathbb{N}_0,q\in [0,1]$ and $\text{Unif}([0,1])$ to denote the uniform distribution on $[0,1]$. We use the symbol $=_d $ to mean equality in distribution. 

Let $\mathbb{G}=(V,E)$ be any (undirected, possibly random) multigraph. Given two vertices $u,v\in V$, we write $u\sim v$ if $\{u,v\}\in E$ and say that vertices $u$ and $v$ are \textit{neighbors}. We often write $uv$ as shorthand for the edge $\{u,v\}$. We write $u\leftrightarrow v$ if there exists a path of occupied edges connecting vertices $u$ and $v$, where we adopt the convention that $v\leftrightarrow v$ for every $v\in V$. We denote by $\CC(v)\coloneqq \{u\in V:u\leftrightarrow v\}$ the \textit{component} containing vertex $v\in V$. We define a \textit{largest component} $\CC_{\max}(\mathbb{G})$ of $\mathbb{G}$ to be some cluster $\CC(v)$ for which $|\CC(v)|$ is maximal, so that $|\CC_{\max}(\mathbb{G})|=\max_{v\in V}|\CC(v)|$. 

\paragraph{Structure of the paper.}The rest of the paper is organized as follows. In Section \ref{secmainthm} we prove Theorem \ref{mainthm}, whereas in Section \ref{heurmg} we describe how to modify the martingale argument of \cite{nachmias_peres:CRG_mgs,nachmias:critical_perco_rand_regular} with the purpose of obtaining stretched exponential estimates in other two critical models.

\subsection{Some useful tools}\label{knownfacts}
Here we recall some results which we will use later on. We start by recalling a Chernoff bound for the Binomial distribution. The following version comes from \cite{janson_et_al:random_graphs}.

\begin{lem}[{\cite[Theorem 2.1]{janson_et_al:random_graphs}}]\label{Bol2}
	Let $B_{N,P}$ be a binomial random variable of parameters $N$ and $p$. Then for every $x\geq 0$ we have
	\[\mathbb{P}(B_{N,P}\ge NP+x)\leq \exp\left(-\frac{x^2}{2(NP+x/3)}\right).\]
\end{lem}

Next we recall Doob's inequality and the Optional Stopping Theorem; proofs of both statements can be found in any advanced probability textbook.

\begin{thm}\label{doobineq}
	Let $(M_t)_{t\in \mathbb{N}_0}$ be a positive submartingale. Then, for every $x>0$, we have
	\[\mathbb{P}(\max_{t\in \{0\}\cup [n]}M_t\geq x)\leq \mathbb{E}[M_n]/x.\]
\end{thm}

\begin{thm}\label{OST}
	Let $(M_t)_{t\in \mathbb{N}_0}$ be a submartingale (resp. martingale, supermartingale) with respect to a filtration $(\mathcal{F}_t)_{t\in \mathbb{N}_0}$. Let $\sigma_1$ and $\sigma_2$ be bounded stopping times with respect to $(\mathcal{F}_t)_{t\in \mathbb{N}_0}$ such that $\sigma_1\leq \sigma_2$.  Then the random variables $M_{\sigma_1}$ and $M_{\sigma_2}$ are integrable and $\mathbb{E}[M_{\sigma_2}|\mathcal{F}_{\sigma_1}]\geq M_{\sigma_1}$ (resp. $\mathbb{E}[M_{\sigma_2}|\mathcal{F}_{\sigma_1}]= M_{\sigma_1}$, $\mathbb{E}[M_{\sigma_2}|\mathcal{F}_{\sigma_1}]\leq M_{\sigma_1}$).
\end{thm}

\section{Proof of Theorem \ref{mainthm}}\label{secmainthm}
Here the goal is to show that $|\mathcal{C}_{\max}(\mathbb{G})|\geq n^{2/3}/A$ with probability at least $1-\exp(-\Omega(A^{3/2}))$ when either $\mathbb{G}=\mathbb{G}(n,p)$ and $p=(1+\lambda n^{-1/3})/n$, or $\mathbb{G}=\mathbb{G}(n,d,p)$ and $p=(1+\lambda n^{-1/3})/(d-1)$, with $3\leq d\in \mathbb{N}$ fixed.
We begin by providing an heuristic derivation of these estimates and we do so in terms of the $\mathbb{G}(n,p)$ random graph; the proof for the $\mathbb{G}(n,d,p)$ model is similar.

\subsection{Heuristic derivation of the bounds stated in Theorem \ref{mainthm}}\label{mymethod}

In order to study the size of $\mathcal{C}_{\max}(\mathbb{G}(n,p))$ we use (as it is standard, see e.g. \cite{remco:random_graphs}) an \textit{exploration process}, which is an algorithmic procedure to sequentially reveal the connected components of a simple, undirected graph, that we apply to $\mathbb{G}(n,p)$ and which reduces the study of component sizes to the analysis of the trajectory of a (self-interacting) discrete-time, non-negative stochastic process.

Roughly speaking, the exploration process (in this setting) works as follows. We start by specifying an ordering of the vertices. Then, at time $t=0$, one of the nodes, say $u$, is declared \textit{active}, whereas all the remaining nodes are \textit{unseen} and there are no \textit{explored} vertices. At time $t=1$, we reveal the \textit{unseen} neighbours of $u$, which are declared active. The step terminates by declaring $u$ explored. Then we iterate the procedure: at time $t=2$, we select one of the active nodes (if any), say $v$, and reveal its unseen neighbours, which are declared active, whereas $v$ itself becomes explored. If at the start of some step $t$ the set of active vertices is \textit{empty}, then we pick the first (with respect to the ordering fixed at the very beginning) \textit{unseen} vertex (if any), say $w$, reveal its unseen neighbours, which are declared active and change the status of $w$ to explored. The procedure terminates when \textit{all} the vertices are in status explored.

We notice that the exploration of a connected component, started at some time $t$, continues as long as there are active vertices, say until time $t+k$, so that the size of the underlying cluster equals $k$, the number of steps of the procedure (starting from time $t$) at which the set of active vertices is non-empty. Following the above description, we use such a procedure to bound (from above) the probability that $\CC_{\max}(\mathbb{G}(n,p))$ contains less than $n^{2/3}/A$ vertices by the probability that the positive excursions of $Y_t$ never last for more than $n^{2/3}/A$ steps, where we write $Y_t$ for the number of \textit{active} vertices at the end of step $t$ in the exploration process.  

More precisely, setting $t_0\coloneqq 0$ and defining $t_i$ to be the first time $t>t_{i-1}$ at which $Y_t=0$ (prior to the end of the procedure) and denoting by $\mathcal{C}_i$ the $i$-th explored cluster, whose exploration starts at time $t_{i-1}+1$, we see that $|\mathcal{C}_i|=t_i-t_{i-1}$, we can bound from above the probability that $\CC_{\max}(\mathbb{G}(n,p))$ contains less than $n^{2/3}/A$ vertices by the probability that each excursion length $t_i-t_{i-1}$ is at most $n^{2/3}/A$:
\begin{equation}
\mathbb{P}(|\mathcal{C}_{\max}(\mathbb{G}(n,p))|<n^{2/3}/A)=\mathbb{P}(t_i-t_{i-1}< n^{2/3}/A\text{ }\forall i\in [N]),
\end{equation}
where we denote by $N$ the (random) number of clusters in $\mathbb{G}(n,p)$ (i.e. the number of excursions needed to reveal all the components in the graph). 

Let us start by making a simple, yet useful observation. On the event where $t_i-t_{i-1}<n^{2/3}/A$ for each $i$, we have
\[n=\sum_{i=1}^{N}|\mathcal{C}_i|=\sum_{i=1}^{N}(t_i-t_{i-1})<N\frac{n^{2/3}}{A},\]
whence $N>An^{1/3}\eqqcolon L$. Therefore, setting $T\coloneqq n^{2/3}/A$ for ease of notation, we can write
\begin{equation*}
\mathbb{P}(t_i-t_{i-1}<T \text{ }\forall i\in [N])
\leq \mathbb{P}(t_i-t_{i-1}<T \text{ }\forall L/2\leq i\leq L).
\end{equation*}
Moreover, 
\begin{equation}\label{prodd}
\mathbb{P}(t_i-t_{i-1}<T \text{ }\forall L/2\leq i\leq L)\leq \prod_{i=L/2+1}^{L}\Big(1-\mathbb{P}(t_i-t_{i-1}\geq T | t_j-t_{j-1}<T \text{ }\forall L/2\leq j\leq i-1)\Big).
\end{equation}
We claim that, for \textit{each} $L/2< i\leq L$,
\begin{equation}\label{mainclaim}
\mathbb{P}(t_i-t_{i-1}\geq T | t_j-t_{j-1}<T \text{ }\forall L/2\leq j\leq i-1)\gtrsim T^{-1/2}.
\end{equation}
Observe that, if (\ref{mainclaim}) were true, then (recalling the value of $T$ and using the classical inequality $1+x\leq e^x$, valid for every real $x$) the expression on the right-hand side of (\ref{prodd}) would satisfy
\begin{multline*}
\prod_{i=L/2+1}^{L}\Big(1-\mathbb{P}(t_i-t_{i-1}\geq T | t_j-t_{j-1}<T \text{ }\forall L/2\leq j\leq i-1)\Big)\\
\lesssim \Big(1-T^{-1/2}\Big)^{L/2}\leq \exp(-L/(2T^{1/2})) = \exp(-\Omega(A^{3/2})),
\end{multline*}
which is the desired upper bound. Hence we are left with the problem of showing (\ref{mainclaim}). To this end, recall the description of the exploration process given earlier. It is clear that, if at time $t-1$ there is at least one active node then, letting $u$ denoting one of them and writing $\eta_t$ for the number of \textit{unseen} neighbors of $u$ which become active at step $t$, we have $Y_t=Y_{t-1}+\eta_t-1$: indeed, the number of active vertices at time $t$ equals the number of active nodes at time $t-1$, plus the newly discovered (necessarily unseen) neighbors of $u$, minus one (since the status of $u$ passes from active to explored). Consequently, if the number of active vertices stays positive from time $s$ until a time $s+t$, then we can write $Y_{s+t}=1+\sum_{k=1}^{t}(\eta_{s+k}-1)$. Moreover, conditional on the history of the exploration process until time $s+k-1$, it is clear from the model definition and the description of the exploration process that the random variable $\eta_{s+k}$ has the $\text{Bin}(n-s-k+1-Y_{s+k-1},p)$ distribution (where $n-s-k+1-Y_{s+k-1}$ is the number of unseen vertices at the start of step $s+k$).

Keeping these considerations in mind and ignoring the term $Y_{s+k-1}$ from the first parameter of the binomial distribution (this approximation is justified by the fact that the number of active vertices is never `too large'), we write
\begin{equation*}
\mathbb{P}(t_i-t_{i-1}\geq T | t_j-t_{j-1}<T \text{ }\forall L/2\leq j\leq i-1)\approx \mathbb{P}\Big(\sum_{k=1}^{t}(\text{Bin}_k(n-t_{i-1}-k,p)-1)>0 \text{ }\forall t< T\Big).
\end{equation*} 
Let us \textit{assume} that the exploration of the $(L-1)$-th cluster finishes before time $LT^{1/2}$; that is, $t_{L-1}\lesssim LT^{1/2}$. Under this assumption, after noticing that each $k$ in the above sum of binomials satisfies $k\leq T=n^{2/3}/A\ll A^{1/2}n^{2/3} = LT^{1/2}$ (whence $LT^{1/2}+k\leq 2LT^{1/2}$ for all large enough $n$), we can couple the $\text{Bin}_k(n-t_{i-1}-k,p)$ random variables appearing within the above probability with independent $\text{Bin}_k(n-2LT^{1/2},p)$ random variables in such a way that
\[\text{Bin}_k(n-t_{i-1}-k,p)\geq  \text{Bin}_k(n-2LT^{1/2},p) \text{ for each }k<T.\]
Hence we can bound
\begin{equation}\label{star}
\mathbb{P}\Big(\sum_{k=1}^{t}(\text{Bin}_k(n-t_{i-1}-k,p)-1)>0 \text{ }\forall t<T\Big)\gtrsim \mathbb{P}\Big(\sum_{k=1}^{t}(\text{Bin}_k(n-2LT^{1/2},p)-1)>0 \text{ }\forall t< T\Big).
\end{equation}
Now the process considered in the last probability is a random walk with a \textit{small} negative drift, as $\mathbb{E}[\text{Bin}_k(n-2LT^{1/2},p)]\approx 1-A^{1/2}n^{-1/3}$. Such a small drift does not alter too much the behavior of the process, which from a practical point of view can be regarded as a $\mathbb{Z}$-valued, mean-zero random walk with finite variance. It is well-known that a random walk with these properties stays above zero for $T$ steps with probability $\Theta(T^{-1/2})$, thus establishing the claim (\ref{mainclaim}).

\subsection{Proof of Theorem \ref{mainthm}: the $\mathbb{G}(n,p)$ model}
We start by providing a formal description of the exploration process introduced in Section \ref{mymethod}. Fix an ordering of the $n$ vertices with $v$ first. Let us denote by $\mathcal{A}_t$,  $\mathcal{U}_t$ and $\mathcal{E}_t$ the (random) sets of active, unseen and explored vertices at the end of step $t\in \mathbb{N}_0$, respectively. Then, for any given $t\in \mathbb{N}_0$, we can partition the vertex set as $[n]=\mathcal{A}_t\cup \mathcal{U}_t\cup \mathcal{E}_t$ (a disjoint union), so that in particular $\mathcal{U}_t=[n]\setminus (\mathcal{A}_t\cup \mathcal{E}_t)$ at each step $t$.\\ 

\textbf{Algorithm 1}. At time $t=0$, vertex $v$ is declared active whereas all other vertices are declared unseen, so that $\mathcal{A}_0=\{v\}$, $\mathcal{U}_0=[n]\setminus\{v\}$ and $\mathcal{E}_0=\emptyset$. For every $t\in \mathbb{N}$, the algorithm proceeds as follows.
\begin{itemize}
	\item [(a)] If $|\mathcal{A}_{t-1}|\geq 1$, we let $u_t$ be the first active vertex (here and in what follows, the term \textit{first} refers to the ordering that we have fixed at the beginning of the procedure).
	\item [(b)] If $|\mathcal{A}_{t-1}|=0$ and $|\mathcal{U}_{t-1}|\geq 1$, we let $u_t$ be the first unseen vertex.
	\item [(c)] If $|\mathcal{A}_{t-1}|=0=|\mathcal{U}_{t-1}|$ (so that $\mathcal{E}_{t-1}=[n]$), we halt the procedure.
\end{itemize}
Denote by $\mathcal{D}_t$ the (random) set of \textit{unseen} neighbours of $u_t$, i.e. we set $\mathcal{D}_t\coloneqq \left\{x\in \mathcal{U}_{t-1}\setminus\{u_t\}:u_t\sim x \right\}$ and note that $\mathcal{U}_{t-1}\setminus\{u_t\}=\mathcal{U}_{t-1}$ if $\mathcal{A}_{t-1}\neq \emptyset$, since $u_t\in \mathcal{A}_{t-1}$ in this case. Then we update
$\mathcal{U}_t\coloneqq \mathcal{U}_{t-1}\setminus (\mathcal{D}_t\cup \{u_t\})$, $\mathcal{A}_t\coloneqq (\mathcal{A}_{t-1}\setminus\{u_t\})\cup \mathcal{D}_t$ and $\mathcal{E}_t\coloneqq \mathcal{E}_{t-1}\cup \{u_t\}$.\\

\begin{obs}
	Note that, since in the procedure \textbf{Algorithm 1} we explore \textit{one vertex} at each step, we have $\mathcal{A}_t\cup \mathcal{U}_t\neq \emptyset$ for every $t\leq n-1$ and $\mathcal{A}_n\cup \mathcal{U}_n=\emptyset$ (as $\mathcal{E}_n=[n]$). Thus the algorithm runs for $n$ steps.
\end{obs}
Denoting by $\eta_t$ the (random) number of unseen vertices that we add to the set of active nodes at time $t$, since at the end of each step $i$ in which $|\mathcal{A}_{i-1}|\geq 1$ we remove the (active) vertex $u_i$ from $\mathcal{A}_{i-1}$ (after having revealed its unseen neighbours), we have the recursion
\begin{itemize}
	\item $|\mathcal{A}_t|=|\mathcal{A}_{t-1}|+\eta_t-1$, if $|\mathcal{A}_{t-1}|>0$;
	\item $|\mathcal{A}_{t}|=\eta_t$, if $|\mathcal{A}_{t-1}|=0$.
\end{itemize}

We set $t_0\coloneqq 0$ and recursively we define $t_i$ to be the first time after $t_{i-1}$ at which the number of active vertices hits zero; that is,
\begin{equation}
t_i\coloneqq \min\{t\geq t_{i-1}+1:Y_t=0\}, \text{ for } i\in [N],
\end{equation}
where we denote by $N$ the (random) number of excursions of the process $Y_t$ (which is the same as the number of clusters in the random graph). Setting $T\coloneqq \lceil n^{2/3}/A\rceil$, we write  
\[\mathbb{P}(|\mathcal{C}_{\max}(\mathbb{G}(n,p))|<n^{2/3}/A)\leq \mathbb{P}(|\mathcal{C}_{\max}(\mathbb{G}(n,p))|<T)=\mathbb{P}(t_i-t_{i-1}<T \text{ }\forall i\in [N]).\]
As remarked in Section \ref{mymethod}, if $t_i-t_{i-1}<T$ for all $i\in [N]$ then $n=\sum_{i=1}^{N}(t_i-t_{i-1})<NT$ and hence 
$N>n/T$. Setting $L\coloneqq \lfloor n/(2T)\rfloor$, we see that $N>2L$. Hence we can bound
\begin{equation}\label{tree}
\mathbb{P}(t_i-t_{i-1}<T \text{ }\forall i\in [N])\leq \mathbb{P}(t_i-t_{i-1}<T \text{ }\forall L+1\leq i\leq 2L).
\end{equation}
Let us introduce the following (good) event:
\begin{equation*}
\mathcal{G}_j=\mathcal{G}_j(B)\coloneqq \{t_j\leq BLT^{1/2}\},
\end{equation*}
where $L\leq j\leq 2L-1$ and $B$ is some large positive constant (independent of $A,n,\lambda$). The event $\mathcal{G}_{j}$ says that we start exploring the $(j+1)$-th cluster before time $BLT^{1/2}$, which also means that when starting exploring the $(j+1)$-th cluster, the number of explored vertices is at most $BLT^{1/2}$ (recall that at each step of the exploration exactly \textit{one} vertex becomes explored). A good control on the $t_j$'s is needed when bounding from \textit{below} the probability that $|\mathcal{C}_i|=t_i-t_{i-1}$ is larger than $T$, which is the same as the probability that the number of active vertices stays above zero for $T$ steps starting from time $t_{i-1}$. Indeed, if the number of explored vertices were `too large' when starting the exploration of $\mathcal{C}_i$, then the latter probability (that is, the probability that $Y_t>0$ for $T$ steps) would be `too small', because the negative drift of the increments $\eta_{t_{i-1}+k}-1$ would be `too large' (in absolute value) to keep the process above zero for $T$ steps. 

Going back to (\ref{tree}), we continue by writing
\begin{equation*}
\mathbb{P}(t_i-t_{i-1}<T \text{ }\forall L+1\leq i\leq 2L)\leq \mathbb{P}(\{t_i-t_{i-1}<T \text{ }\forall L+1\leq i\leq 2L\}\cap\mathcal{G}_{2L-1})+\mathbb{P}(\mathcal{G}^c_{2L-1}).
\end{equation*}
The next result controls the probability of the (bad) event $\mathcal{G}^c_{2L-1}$ and shows that it is at most $\exp(-\Omega(A^{3/2}))$.
\begin{prop}\label{rub}
	There exist positive constants $B_0,A_0=A_0(\lambda)>0$ such that, for any $A\geq A_0,B\geq B_0$ and for all large enough $n$, we have 
	\[\mathbb{P}(\mathcal{G}^c_{2L-1})\leq \exp(-cA^{3/2}),\]
	for some positive constant $c=c(B)>0$.
\end{prop}
Therefore, in what follows, we can focus on the probability
\begin{equation}\label{focus}
\mathbb{P}(\{t_i-t_{i-1}<T \text{ }\forall L+1\leq i\leq 2L\}\cap\mathcal{G}_{2L-1}).
\end{equation}
Let us denote by $\mathcal{F}_{j}$ the $\sigma$-algebra consisting of all the information collected by the exploration process by time $j$; then $\mathcal{F}_{t_j}$ represents the information collected by the time at which we finish the exploration of the $j$-th component. 

With the next result we show that, on the event $\mathcal{G}_{i-1}$, we have $t_i-t_{i-1}\geq T$ with probability at least of order $T^{-1/2}$ for each $L+1\leq i\leq 2L-1$.
\begin{prop}\label{mainlemma}
	There exist positive constants $B_0,A_0=A_0(\lambda)>0$ such that, for any $A\geq A_0,B\geq B_0$ and for all large enough $n$, on the event $\mathcal{G}_{i-1}$ we have
	\[\mathbb{P}(t_{i}-t_{i-1}\geq T|\mathcal{F}_{t_{i-1}})\geq \frac{c_0}{T^{1/2}}\]
	for each $L+1\leq i\leq 2L$, for some constant $c_0=c_0(B)>0$. 
\end{prop}
It is now very easy to show that the expression in (\ref{focus}) is at most $\exp(-cA^{3/2})$ (for some $c>0$). Indeed, using the tower property of conditional expectation, we can write
\begin{multline}\label{mainexpr}
\mathbb{P}(\{t_i-t_{i-1}<T \text{ }\forall L+1\leq i\leq 2L\}\cap\mathcal{G}_{2L-1})\\
= \mathbb{E}\Big[\mathbbm{1}_{\{t_i-t_{i-1}<T \text{ }\forall L+1\leq i\leq 2L-1\}\cap \mathcal{G}_{2L-1}}\big(1-\mathbb{P}(t_{2L}-t_{2L-1}\geq T|\mathcal{F}_{t_{2L-1}})\big)\Big].
\end{multline}
Thanks to Lemma \ref{mainlemma}, the expression on the right-hand side of (\ref{mainexpr}) is at most
\[\big(1-c_0T^{-1/2}\big)\mathbb{P}(\{t_i-t_{i-1}<T \text{ }\forall L+1\leq i\leq 2L-1\}\cap\mathcal{G}_{2L-1}).\]
Iterating and using the classical inequality $1+x\leq e^x$ (which is valid for all $x\in \mathbb{R}$) together with the fact that  $\mathcal{G}_{j}\subset \mathcal{G}_{j-1}$ for each $j$, we conclude that for all large enough $A$ and $n$
\begin{equation}\label{iterate}
\mathbb{P}(\{t_i-t_{i-1}<T \text{ }\forall L+1\leq i\leq 2L\}\cap\mathcal{G}_{2L})\leq \big(1-c_0T^{-1/2}\big)^L\leq \exp(-c_0LT^{-1/2}).
\end{equation}
Recalling the definition of $L$ and $T$, it is easy to see that $LT^{-1/2}\geq A^{3/2}/4$ for all large enough $n$ (the constant $4$ is clearly not best possible but sufficient for our purposes) and hence, putting all pieces together, we conclude that there is a constant $c>0$ such that, for all large enough $n$,
\[\mathbb{P}(|\mathcal{C}_{\max}(\mathbb{G}(n,p))|<n^{2/3}/A)\leq \exp(-cA^{3/2}).\]
This concludes the proof of Theorem \ref{mainthm} for the $\mathbb{G}(n,p)$ random graph. What is left to do is to prove Propositions \ref{rub} and \ref{mainlemma}; this is done in the following section.

\subsubsection{Proofs of Propositions \ref{rub} and \ref{mainlemma}}
In this section we establish the two auxiliary facts which were stated, without proof, while proving the main theorem for the $\mathbb{G}(n,p)$ random graph.

\paragraph{Proof of Proposition \ref{rub}.}
Recall that $\mathcal{G}_{2L-1}\coloneqq \{t_{2L-1}\leq BLT^{1/2}\}$. To bound from above the probability of the complementary event, we proceed as follows. Let $S_0\coloneqq 1$ and define recursively $S_t\coloneqq S_{t-1}+\eta_{t}-1$, where the random variable $\eta_i$, conditional on the history of the exploration process until time $i-1$ (which we recall is denoted by $\mathcal{F}_{i-1}$), has the $\text{Bin}(n-(i-1)-Y_{i-1},p)$ distribution. 

The $\mathbb{Z}$-valued random process $(S_t)_{t\in \mathbb{N}_0}$ is closely related to the process of active vertices $(Y_t)_{t\in \mathbb{N}_0}$; in particular, they coincide up to time $t_1$ (the time at which we completes the exploration of the first cluster). However, at time $t_1+1$, we have $Y_{t_1+1}=\eta_{t_1+1}$ whereas $S_{t_1+1}=\eta_{t_1+1}-1$. The two processes $Y_{t_1+t}$ and $S_{t_1+t}$ have the same increments for $t_1+1<t\leq t_2$, but then again $Y_{t_2+1}=\eta_{t_2+1}$ whereas $S_{t_2+1}=\eta_{t_2+1}-2$, and so on.
Then $(-\min_{j\in [t]}S_j)+1$ denotes the number of disjoint clusters that have been fully explored
by time $t\in [n]$ (see e.g. Chapter 5.2.4 in \cite{remco:random_graphs}).  
When $Y_t=0$, we have fully explored a cluster, and this occurs precisely when $S_t=\min_{j\in [t-1]}S_j-1$, i.e., the process $(S_t)_{t\in [n]}$ reaches a new all-time minimum. 

Now observe that, if $t_{2L-1}>BLT^{1/2}$, then the number of components that are fully explored up to time $\lfloor BLT^{1/2}\rfloor $ is smaller than $2L-1$ (precisely because $t_{2L-1}$ is the time at which we finish exploring the $(2L-1)$-th cluster). Therefore we can bound 
\begin{equation}\label{starstar}
\mathbb{P}(\mathcal{G}^c_{2L-1})\leq \mathbb{P}\Big(-\min_{j\leq \lfloor BLT^{1/2}\rfloor}S_j<2L-1\Big)= \mathbb{P}\Big(\max_{j\leq \lfloor BLT^{1/2}\rfloor}(-S_j)<2L-1\Big)\leq \mathbb{P}\big(S_{\lfloor BLT^{1/2}\rfloor}>-2L+1\big).
\end{equation} 
Now, by definition of $S_i$, we have
\[S_{\lfloor BLT^{1/2}\rfloor}=1+\sum_{i=1}^{\lfloor BLT^{1/2}\rfloor}(\eta_i-1) \text{ and } (\eta_i|\mathcal{F}_{i-1})=_d\text{Bin}(n-(i-1)-Y_{i-1},p)\leq _{sd}\text{Bin}(n-(i-1),p).\] 
By coupling the $\eta_i$ with independent $\text{Bin}(n-(i-1),p)$ random variables $\delta_i$ in such a way that $\eta_i\leq \delta_i$ for each $i$, we arrive at
\[\mathbb{P}(S_{\lfloor BLT^{1/2}\rfloor}>-2L+1)=\mathbb{P}\Big(\sum_{i=1}^{\lfloor BLT^{1/2}\rfloor}(\eta_i-1)>-2L\Big)\leq \mathbb{P}\Big(\sum_{i=1}^{\lfloor BLT^{1/2}\rfloor}(\delta_i-1)>-2L\Big).\]
Now let $\delta'_i$ be independent random variables with the $\text{Bin}(i-1,p)$ distribution, also independent of the $\delta_i$. Then, setting $\xi_i\coloneqq \delta_i+\delta'_i=_d\text{Bin}(n,p)$ for each $i$, we see that 
\begin{equation}\label{couplingeq}
\mathbb{P}\Big(\sum_{i=1}^{\lfloor BLT^{1/2}\rfloor}(\delta_i-1)>-2L\Big)=\mathbb{P}\Big(\sum_{i=1}^{\lfloor BLT^{1/2}\rfloor}(\xi_i-1)>\sum_{i=1}^{\lfloor BLT^{1/2}\rfloor}\delta'_i-2L\Big).
\end{equation}
Since 
\[\sum_{i=1}^{\lfloor BLT^{1/2}\rfloor}\delta'_i=_d\text{Bin}\Big(\frac{M}{2},p \Big) \text{ where } M\coloneqq (\lfloor BLT^{1/2}\rfloor-1)(\lfloor BLT^{1/2}\rfloor-2)\sim B^2L^2T,\]
we can use Lemma \ref{Bol2} to bound 
\begin{equation*}
\mathbb{P}\Big(\sum_{i=1}^{\lfloor BLT^{1/2}\rfloor}\delta'_i\leq (Mp)/2-L\Big)\leq \exp\big(-c\frac{L^2}{Mp}\big)\leq \exp\big(-\frac{c}{Tp}\big) 
\end{equation*}
where $c=c(B)>0$ is a constant which depends on $B$. Since $Tp=O(A^{-1}n^{-1/3})$, we see that
\[\exp\big(-\frac{c}{Tp}\big) \leq \exp(-cAn^{1/3}).\]
Therefore, going back to (\ref{couplingeq}) and setting $\ell\coloneqq \lfloor BLT^{1/2}\rfloor(1-np)+(Mp)/2-3L$, 
we can bound
\begin{align}\label{pp}
\nonumber\mathbb{P}\Big(\sum_{i=1}^{\lfloor BLT^{1/2}\rfloor}(\xi_i-1)&>\sum_{i=1}^{\lfloor BLT^{1/2}\rfloor}\delta'_i-2L\Big)\\
\nonumber&\leq  \mathbb{P}\Big(\sum_{i=1}^{\lfloor BLT^{1/2}\rfloor}(\xi_i-1)>(Mp)/2-3L\Big)+\exp(-cAn^{1/3})\\
&= \mathbb{P}\Big(\text{Bin}(n\lfloor BLT^{1/2}\rfloor,p)>np\lfloor BLT^{1/2}\rfloor +\ell\Big)+\exp(-cAn^{1/3}).
\end{align}
Now note that, by taking a large enough $A$ (in particular, $A>B^2\lambda^2$ suffices) we obtain
\[\ell\geq (Mp)/2-\Big(3L+\frac{|\lambda|BLT^{1/2}}{n^{1/3}}\Big)\geq (Mp)/2-4L.\]
Therefore the last probability in (\ref{pp}) is at most
\begin{equation}\label{hh}
\mathbb{P}\Big(\text{Bin}(n\lfloor BLT^{1/2}\rfloor,p)>np\lfloor BLT^{1/2}\rfloor + (Mp)/2-4L\Big).
\end{equation}
A simple computation shows that, for all large enough $A$ and $n$, since $LT\geq (1-o(1))n2^{-1}$,
\[(Mp)/2-4L\geq \frac{B^2L^2T}{4n} -4L=L\Big(\frac{B^2LT}{4n}-4\Big)\geq L(B^2/12-4),\]
which is positive provided that $B$ is large enough. Hence, using once again Lemma \ref{Bol2}, we can bound from above the probability in (\ref{hh}) by
\begin{equation}\label{gb}
\mathbb{P}\Big(\text{Bin}(n\lfloor BLT^{1/2}\rfloor,p)>np\lfloor BLT^{1/2}\rfloor + L(B^2/12-4)\Big)\leq \exp\big(-c\frac{L}{T^{1/2}}\big)\leq \exp\big(-cA^{3/2}\big)
\end{equation}
for some constant $c=c(B)>0$. Since $An^{1/3}\gg A^{3/2}$ (because of our assumption $A\ll n^{2/3}$), we see that the term $\exp(-cAn^{1/3})$ in (\ref{pp}) is much smaller than $ \exp(-cA^{3/2})$ which appears on the right-hand side of (\ref{gb}), whence the desired result follows.

\paragraph{Proof of Proposition \ref{mainlemma}.}
Here the goal is to show that there is a constant $c_0=c_0(B)>0$ such that, for each $L+1\leq i\leq 2L$, on the event $\mathcal{G}_{i-1}$ it holds that
\[\mathbb{P}(t_{i}-t_{i-1}\geq T|\mathcal{F}_{t_{i-1}})\geq \frac{c_0}{T^{1/2}}\]
for all large enough $n$, where we recall that $\mathcal{F}_{t_{i-1}}$ is the $\sigma$-algebra consisting of all the information gathered by the exploration process until time $t_{i-1}$ (i.e. until the moment at which the procedure finishes revealing the $(i-1)$-th cluster). Setting $Y^i_t\coloneqq Y_{t_{i-1}+t}$ and $\eta^i_t\coloneqq \eta_{t_{i-1}+t}$ to simplify notation, we can write
\begin{equation}\label{prrr}
\mathbb{P}(t_{i}-t_{i-1}\geq T|\mathcal{F}_{t_{i-1}})=\mathbb{P}(Y^i_t> 0 \text{ }\forall t < T|\mathcal{F}_{t_{i-1}})\geq \mathbb{P}\Big(1+\sum_{k=1}^{t}(\eta^i_k-1)> 0 \text{ }\forall t \leq T|\mathcal{F}_{t_{i-1}}\Big),
\end{equation}
where we recall that, setting $\mathcal{F}_{i-1}(k-1)\coloneqq \mathcal{F}_{t_{i-1}+k-1}$, 
\begin{equation}\label{condlawgnp}
(\eta^i_k|\mathcal{F}_{i-1}(k-1))=_d\text{Bin}_k(n-t_{i-1}-k+1-Y^i_{k-1},p).
\end{equation}
Before proceeding with the actual proof, let us make a few observations and give some hints on how we intend to proceed. (A similar heuristic was given in Section \ref{mymethod}).

Note that the $\eta^i_k$ are not independent, which makes our analysis non-trivial. The plan is to replace these random variables with a sequence of \textit{independent} variables which are easier to analyze. To this end we first observe that, on the event $\mathcal{G}_{i-1}$, we have $t_{i-1}< BLT^{1/2}$; moreover, recalling that $L\sim An^{1/3}/2,T\sim n^{2/3}/A$ and $A\ll n^{2/3}$ by assumption, we have $k\leq t<T\ll LT^{1/2}(\ll n)$. If we can show that $Y^i_{k-1}$ is at most $LT^{1/2}$ too, then we would obtain $t_{i-1}+k+Y^i_{k-1}\leq 3BLT^{1/2}$ and we could couple the $\text{Bin}_k(n-t_{i-1}-k+1-Y^i_{k-1},p)$ random variables with independent $\text{Bin}_k(n-3BLT^{1/2},p)$ random variables in such a way that
\[\text{Bin}_k(n-t_{i-1}-k+1-Y^i_{k-1},p)\geq  \text{Bin}_k(n-3BLT^{1/2},p),\]
thus removing the two sources of randomness from the first parameter of the (conditional) distribution in (\ref{condlawgnp}). At this stage, we would have a random walk with independent and identically distributed increments $\text{Bin}_k(n-3BLT^{1/2},p)-1$ having a \textit{small} negative drift; the drift is so small that it does not impact much the behavior of the process, which can be regarded as a random walk with iid, mean zero increments having finite second moment. Putting all pieces together, we would bound from below the expression in (\ref{prrr}) by the probability that a mean-zero $\mathbb{Z}$-valued random walk with finite variance stays above zero for $T$ steps, which is known to be $\Theta(T^{-1/2})$ (as desired).

We now make rigorous the above simple argument, and start by replacing the $\eta^i_k$ with iid random variables having the $\text{Bin}(n-3L\lceil BT^{1/2}\rceil,p)$ law. 

The next result is basically just a rephrasing of Lemma $4.1$ in \cite{de_ambroggio_roberts:near_critical_ER} adapted to our current need; we include a proof for completeness.

\begin{lem}\label{couplinglem}
	There exists a sequence $(\delta_k)_{k\leq T}$ of independent random variables, with $\delta_k=_d\text{Bin}(n-3L\lceil BT^{1/2}\rceil,p)$ for each $k$, such that on the event $\mathcal{G}_{i-1}$ it holds that 
	\[\mathbb{P}\Big(1+\sum_{k=1}^{t}(\eta^i_k-1)> 0 \text{ }\forall t \leq T|\mathcal{F}_{t_{i-1}}\Big)\geq \mathbb{P}\Big(1+\sum_{k=1}^{t}(\delta_k-1)> 0 \text{ }\forall t \leq T\Big)-\exp(-cLT^{1/2})\]
	for all large enough $n$, for some constant $c=c(B)>0$.
\end{lem}
\begin{proof}
Recall that $u_k$ is the vertex that is explored at step $k$ in the exploration of $\mathbb{G}(n,p)$. Denoting by $X^i_k(v)$ the indicator that vertex $u_{t_{i-1}+k}$ is a neighbor of $v\in [n]$, we note that the random variable $\eta^i_k$ can be expressed as
\[\eta^i_k=\sum_{v\in [n]\setminus (\mathcal{A}^i_{k-1}\cup \mathcal{E}^i_{k-1})}^{}X^i_k(v),\]
with $\mathcal{E}^i_{j}\coloneqq \mathcal{E}_{t_{i-1}+j}$ and $\mathcal{A}^i_{j}\coloneqq \mathcal{A}_{t_{i-1}+j}$ being the sets of explored and active vertices at time $t_{i-1}+j$, respectively. Let us denote by $\mathcal{R}^i_{j}\coloneqq \mathcal{A}^i_{j}\cup \mathcal{E}^i_{j}$ the set of vertices which are either active or explored at time $t_{i-1}+j$ (we use the letter $\mathcal{R}$ for `revealed'). We also give a name to the set of \textit{unseen} vertices which become active at step $j$ of the exploration of $\mathcal{C}_{i}$ (i.e. at time $t_{i-1}+j$); denote such a set by $\mathcal{A}^{i,*}_j$, so that then $\mathcal{R}^i_{k-1}=\bigcup_{j=0}^{k-1}\mathcal{A}^{i,*}_j$ (with $\mathcal{A}^{i,*}_0=\{u_{t_{i-1}+1}\}$, the vertex from which we start the exploration of $\mathcal{C}_i$). 

For each $1\leq k\leq T$, if $|\mathcal R^i_{k-1}| < 3L \lceil BT^{1/2}\rceil$ (which occurs when $Y^i_{k-1}<BLT^{1/2}$, because $|\mathcal R^i_{k-1}|=Y^i_{k-1}+t_{i-1}+k-1<Y^i_{k-1}+2BLT^{1/2}$ on the event $\mathcal{G}_{i-1}$) we let $\mathcal B^{i,*}_{k-1}$ be any subset of $[n]$ such that
	\[\mathcal A^{i,*}_{k-1} \subset \mathcal B^{i,*}_{k-1}; \hspace{1cm}\mathcal B^{i,*}_{k-1} \cap \mathcal R^i_{k-2} = \emptyset; \hspace{1cm} \big|\mathcal B^{i,*}_{k-1} \cup \mathcal R^i_{k-2}\big| = 3L\lceil BT^{1/2}\rceil\ll n.\]
If $|\mathcal R^i_{k-1}| \geq 3L\lceil BT^{1/2}\rceil$ then we simply let $\mathcal B^{i,*}_{k-1} = \mathcal A^{i,*}_{k-1}$. Next we define
\[h_k \coloneqq  \big|\mathcal B^{i,*}_{k-1} \cup \mathcal R^i_{k-2}\big| - 3L\lceil BT^{1/2}\rceil \ge 0.\]
Take a (doubly-infinite) sequence $(I^k_j:j,k\in \mathbb{N})$ of independent Bernoulli random variables of parameter $p$, also independent of everything else. Then define
\[\delta_k \coloneqq  \sum_{v\in [n]\setminus (\mathcal B^{i,*}_{k-1} \cup \mathcal R^i_{k-2})}^{}X^i_k(v) + \sum_{i=1}^{h_k} I^k_j.\]
Note that
\[\big|[n]\setminus \big(\mathcal B^{i,*}_{k-1}\cup \mathcal R^i_{k-2}\big)\big| + h_k = n - 3L\lceil BT^{1/2}\rceil.\]
Moreover, the random variables $\delta_k$ are independent because the $\{X^i_k(v) : v\in (\mathcal{R}^i_{k-2})^c\}$ are independent and independent of $\{X^i_j(v) : v\in (\mathcal{R}^i_{j-2})^c\}$ for any $j\neq k$; they are also independent of $\mathcal{F}_{t_{i-1}}$. 

We also observe that, if $|\mathcal R^i_{k-1}| < 3L\lceil BT^{1/2}\rceil$, then $|\mathcal B^{i,*}_{k-1} \cup \mathcal R^i_{k-2}| = 3L\lceil BT^{1/2}\rceil$ and so $h_k=0$. Since we also have $\mathcal A^{i,*}_{k-1} \subset \mathcal B^{i,*}_{k-1} $ by construction, we see that if $|\mathcal R^i_{k-1}| < 3L\lceil BT^{1/2}\rceil$ then $\eta^i_k \ge \delta_k$. 
Thus, recalling that $t_i$ denotes the first time at which the set of active vertices becomes empty after time $t_{i-1}$, we see that the probability which appears in the statement of the lemma is at least
\begin{equation}\label{last}
\mathbb{P}\Big(1+\sum_{k=1}^{t}(\delta_k-1)> 0 \text{ }\forall t \leq T, Y^i_{k}<BLT^{1/2} \text{ }\forall k< (t_{i}-t_{i-1})\wedge T|\mathcal{F}_{t_{i-1}}\Big).
\end{equation}
Indeed, on the event $\{Y^i_{k}<BLT^{1/2} \text{ }\forall k< T\}$, we have $|\mathcal{R}^i_{k-1}|<3L\lceil BT^{1/2}\rceil$ for each $k$ as discussed earlier, and so $\eta^i_k\geq \delta_k$ for $k<T$ (and $(t_i-t_{i-1})\wedge T=T$ on the event $\{1+\sum_{k=1}^{t}(\eta^i_k-1)> 0 \text{ }\forall t \leq T-1\}$). To conclude, we observe that the probability on the right-hand side of (\ref{last}) is at least 
\begin{equation*}
\mathbb{P}\Big(1+\sum_{k=1}^{t}(\delta_k-1)> 0 \text{ }\forall t \leq T\Big)-\mathbb{P}\Big(\exists k< (t_{i}-t_{i-1})\wedge T: Y^i_{k}\geq BLT^{1/2}|\mathcal{F}_{t_{i-1}}\Big),
\end{equation*}
where we have used the above mentioned independence of the $\delta_k$ and $\mathcal{F}_{t_{i-1}}$. There remains to bound from above the second probability on the right-hand side of the last expression, which is easily done. Indeed, using a union bound together with the fact that $Y^i_k=1+\sum_{j=1}^k(\eta^i_j-1)$ for $k\leq t_i-t_{j-1}$, we can write
\begin{equation*}
\mathbb{P}\Big(\exists k< (t_{i}-t_{i-1})\wedge T: Y^i_{k}\geq BLT^{1/2}|\mathcal{F}_{t_{i-1}}\Big)\leq \sum_{k=1}^{T-1}\mathbb{P}\Big(1+\sum_{j=1}^k(\eta^i_j-1)\geq BLT^{1/2}|\mathcal{F}_{t_{i-1}}\Big).
\end{equation*}
By coupling the $\eta^i_j$ with independent $\text{Bin}(n,p)$ random variables $\xi_j$ in such a way that $\eta^i_j\leq \xi_j$ and using Lemma \ref{Bol2} we see that, if $A$ is large enough, then the probability on the right-hand side of the last expression is at most
\begin{equation*}
\sum_{k=1}^{T-1}\exp\big(-c\frac{B^2L^2T}{knp+BLT^{1/2}}\big)\leq  \exp\big(-cLT^{1/2}[1-\log(T)/(cLT^{1/2})]\big)\leq \exp\big(-cLT^{1/2})
\end{equation*}  
for some $c=c(B)>0$, where we have used that $\log(T)\leq T\ll LT^{1/2}$.
\end{proof}
In order to bound from below the probability that the random walk formed by the $\delta_k$ stays above zero for $T$ steps (see Lemma \ref{couplinglem}), we use the following result, which is taken from \cite{de_ambroggio_roberts:near_critical_ER}.

\begin{lem}\label{turnmeanzero}[Lemma 4.6 in \cite{de_ambroggio_roberts:near_critical_ER}]
	Take $n\in\N$, $h_n\ge 0$, $a_n\in(-1,\infty)$ satisfying $na_n\in\mathbb Z$, $b_n\in(-1,n-1)$ and $t_n\in\N$. Suppose that $M_t = 1 + \sum_{i=1}^t (W_i-1)$ where the $W_i$ are independent $\Bin(n(1+a_n),(1+b_n)/n)$ random variables. Let $\mu_n = (1+a_n)(1+b_n)$. Then
	\begin{multline*}
	\mathbb{P}\big(M_t>0\,\,\,\,\forall t\in [t_n], \, M_{t_n}\in[h_n,2h_n]\big)\\
	\ge (\mu_n\wedge 1)^{2h_n} \mu_n^{t_n-1} e^{(1-\mu_n)t_n}\mathbb{P}\big(\hat M_t > 0 \,\,\,\,\forall t\in[t_n],\, \hat M_{t_n}\in[h_n,2h_n]\big) - \frac{t_n}{n}(1+a_n)(1+b_n)^2
	\end{multline*}
	where $\hat M_t= 1+\sum_{i=1}^{t}(\hat W_i-1)$, and $(\hat W_i)_{i=1}^{t_n}$ is a sequence of independent $\text{Poi}(1)$ random variables.
\end{lem}

We apply this lemma with 
\[a_n\coloneqq -3L\lceil BT^{1/2}\rceil/n \text{ and }b_n\coloneqq \lambda/n^{1/3}. \]
Note that clearly $na_n\in \mathbb{Z}$ and, since $LT^{1/2}/n\sim A^{1/2}/2n^{1/3}\ll 1$ (as we assumed that $A\ll n^{2/3}$), we also have $a_n>-1$ for all sufficiently large $n$. Moreover, it is clear that $b_n\in (-1,n-1)$ provided $n$ is large enough. 

Now, setting $\mu_n\coloneqq (1+a_n)(1+b_n)$ as in the statement of Lemma \ref{turnmeanzero}, it is not difficult to see that $\mu_n$ is negative provided $A\geq A_0$ for a large enough $A_0=A_0(\lambda)>0$ which depends on $\lambda$ and moreover
\[|1-\mu_n|=O\big(\frac{LT^{1/2}}{n}\big)=O\big( \frac{A^{1/2}}{n^{1/3}}\big)\ll 1.\]
Hence, as $n\rightarrow \infty$, we can write
\begin{equation}\label{taylorlog}
\log(\mu_n)=\log(1-(1-\mu_n))=-(1-\mu_n)+O((1-\mu_n)^2).
\end{equation}
Now using the above Taylor expansion we see that, for all large enough $n$, 
\begin{equation*}
(\mu_n\wedge 1)^{2\lceil T^{1/2}\rceil}\mu_n^{T-1}e^{(1-\mu_n)T}\geq c\mu_n^{2\lceil T^{1/2}\rceil}\exp(-cT(1-\mu_n)^2)\geq c\mu_n^{2\lceil T^{1/2}\rceil},
\end{equation*}
where the last inequality follows from the fact that $T(1-\mu_n)^2=O(1)$. Moreover, by expressing $\mu_n^{2\lceil T^{1/2}\rceil}=e^{2\lceil T^{1/2}\rceil \log(\mu_n)}$ and using once more the expansion in (\ref{taylorlog}), we conclude that  
\[(\mu_n\wedge 1)^{2\lceil T^{1/2}\rceil}\mu_n^{T-1}e^{(1-\mu_n)T}\geq c\]
for some positive constant $c=c(B)>0$ which depends solely on $B$. But then we can bound, using Lemma \ref{turnmeanzero} (with $t_n=T$ and $h_n=\lceil T^{1/2}\rceil$)
\begin{align}\label{concl}
\nonumber\mathbb{P}\Big(1+\sum_{k=1}^{t}(\delta_k-1)> 0 \text{ }\forall t \leq T\Big)&\geq \mathbb{P}\Big(1+\sum_{k=1}^{t}(\delta_k-1)> 0 \text{ }\forall t \leq T,1+\sum_{k=1}^{T}(\delta_k-1)\in [\lceil T^{1/2}\rceil,2\lceil T^{1/2}\rceil]\Big)\\
&\geq c(B)\mathbb{P}\big( M_t > 0 \text{ }\forall t\leq T, M_{T}\in[\lceil T^{1/2}\rceil,2\lceil T^{1/2}\rceil]\big)-O\big(frac{T}{2n}\big)
\end{align}
for all large enough $n$, where $ M_t= 1+\sum_{i=1}^{t}( W_i-1)$ and $(W_i)_{i=1}^{T}$ is a sequence of independent Poisson random variables with mean one. To finish the proof, we use the following generalised ballot theorem:
\begin{thm}[Addario-Berry and Reed \cite{addario_berry_reed:ballot_theorems}]\label{genballot}
	Suppose $X$ is a random variable satisfying $\mathbb{E}[X]=0$, $\text{Var}(X)>0$, $\mathbb{E}[X^{2+\alpha}]<\infty$ for some $\alpha >0$, and $X$ is a lattice random variable with period $d$ (meaning that $dX$ is an integer random variable and $d$ is the smallest positive real number for which this holds). Then given independent random variables $X_1,X_2,\dots$ distributed as $X$ with associated partial sums $S_t = \sum_{i=1}^{t}X_i$, for all $j$ such that $0\leq j =O\left(\sqrt{n}\right)$ and such that $j$ is a multiple of $1/d$ we have 
	\begin{equation}\label{ballotestim}
	\mathbb{P}\left(S_t > 0 \hspace{0.15cm}\forall t\in [n],S_n=j\right)=\Theta\left(\frac{j+1}{n^{3/2}}\right).
	\end{equation}
\end{thm}
Going back to (\ref{concl}) and using (\ref{ballotestim}) we see that
\begin{equation}
\mathbb{P}\big( M_t > 0 \text{ }\forall t\leq T, M_{T}\in[\lceil T^{1/2}\rceil,2\lceil T^{1/2}\rceil]\big)\geq c\sum_{j=\lceil T^{1/2}\rceil-1}^{2\lceil T^{1/2}\rceil-1}\frac{j}{T^{3/2}}\geq cT^{-1/2}.
\end{equation}
Finally, observe that the term $T/n$ on the right-hand side of (\ref{concl}) is $O((An^{1/3})^{-1})$, which is much smaller than $A^{1/2}/n^{1/3}=\Theta(T^{-1/2})$. 
Therefore we conclude that there exists a constant $c_0=c_0(B)>0$ which depends solely on $B$ such that, for all large enough $A,n$,
\[\mathbb{P}\Big(1+\sum_{k=1}^{t}(\delta_k-1)> 0 \text{ }\forall t \leq T\Big)\geq c_0T^{-1/2},\]
for each $L+1\leq i\leq 2L$. Combining this last estimate together with (\ref{prrr}) and Lemma \ref{couplinglem} we obtain the desired result.

\subsection{Proof of Theorem \ref{mainthm}: the $\mathbb{G}(n,d,p)$ model}
We start by recalling a method for exploring the components of the graph $\mathbb{G}(n,d,p)$, which we recall is the random graph obtained by performing bond percolation with parameter $p$ on a realization $\mathbb G(n,d)$ of a $d$-regular graph sampled uniformly at random from the set of all $d$-regular (simple) graphs on $[n]$. The exploration process we use here is taken verbatim from \cite{de_ambroggio_roberts:near_critical_RRG} (see also \cite{remco:random_graphs} and \cite{nachmias:critical_perco_rand_regular}). 

In order to describe such procedure, we first recall the \textit{configuration model}; this is an algorithm, which is due to Bollob\'as \cite{bollobas_config}, that gives us a way of choosing a graph $\mathbb{G}(n,d)$ uniformly at random from the set of all $d$-regular graphs on $n$ vertices, provided that $dn$ is even.

\paragraph{The configuration model.} Start with $dn$ stubs, labelled $(v,i)$ for $v\in[n]$ and $i\in[d]$. Choose a stub $(V_0,I_0)$ in some way (the manner of choosing may be deterministic or random) and pair it uniformly at random with another stub $(W_0,J_0)$. Say that these two stubs are \emph{matched} and put $\{V_0,W_0\}\in E$. Then at each subsequent step $k\in\{1,\ldots, nd/2-1\}$, choose a stub $(V_k,I_k)$ in some way from the set of unmatched stubs, and pair it uniformly at random with another unmatched stub $(W_k,J_k)$. Say that these two stubs are matched and put $\{V_k,W_k\}\in E$.

At the end of this process, the resulting object $G=([n],E)$ is uniformly chosen amongst all $d$-regular \emph{multigraphs} on $[n]$, i.e.~it may have multiple edges or self-loops. However, with probability converging to $\exp((1-d^2)/4)$ it is a simple graph, and conditioning on this event, it is uniformly chosen amongst all $d$-regular (simple) graphs on $[n]$.\\

The exploration process we employ will use the configuration model to generate components of $\mathbb{G}'(n,d,p)$, the $p$-percolated version of a uniformly random $d$-regular multigraph $\mathbb{G}'(n,d)$. When we talk about whether an edge of $\mathbb{G}'(n,d)$ is \emph{retained}, we mean whether it is present in $\mathbb{G}'(n,d,p)$. 


During our exploration process, each stub (or half-edge) of $\mathbb{G}'(n,d)$ is either \textit{active}, \textit{unseen} or \textit{explored}, and its status changes during the course of the process. We write $\mathcal{A}_{t}$, $\mathcal{U}_{t}$ and $\mathcal{E}_{t}$ for the sets of active, unseen and explored stubs at the end of the $t$-th step of the exploration process, respectively.

Given a stub $h$ of $\mathbb{G}(n,d)$, we denote by $v(h)$ the vertex incident to $h$ (in other words, if $h=(u,i)$ for some $i$ then $v(h) = u$) and we write $\mathcal{S}(h)$ for the set of \textit{all} stubs incident to $v(h)$ in $\mathbb{G}'(n,d)$ (that is, $\mathcal{S}(h) = \{(v(h),i) : i\in[d]\}$; note in particular that $h\in \mathcal{S}(h)$).\\

\textbf{Algorithm 2}. Let $V_n$ be a vertex selected uniformly at random from $[n]$. At step $t=0$ we declare \textit{active} all stubs incident to $V_n$, while all the other $d(n-1)$ stubs are declared \textit{unseen}. Therefore we have that $|\mathcal{A}_0|=d$, $|\mathcal{U}_0|=d(n-1)$ and $|\mathcal{E}_0|=0$. For every $t\geq 1$, we proceed as follows.
\begin{itemize}
	\item [(a)] If $|\mathcal{A}_{t-1}|\geq 1$, we choose (in an arbitrary way) one of the active stubs, say $e_t$, and we pair it with a stub $h_t$ picked uniformly at random from $[dn]\setminus \left(\mathcal{E}_{t-1}\cup \{e_t\}\right)$, i.e. from the set of all unexplored stubs after having removed $e_t$. 
	\begin{itemize}
		\item [(a.1)] If $h_t\in \mathcal{U}_{t-1}$ \textbf{and} the edge $e_t h_t$ is retained in the percolation (the latter event occurs with probability $p$, independently of everything else), then all the unseen stubs in $\mathcal{S}(h_t)\setminus \{h_t\}$ are declared active, while $e_t$ and $h_t$ are declared explored. Formally we update
		\begin{itemize}
			\item $\mathcal{A}_t\coloneqq \left(\mathcal{A}_{t-1}\setminus \{e_t\}\right)\cup\left(\mathcal{U}_{t-1} \cap \mathcal{S}(h_t)\setminus \{h_t\}\right)$;
			\item $\mathcal{U}_t\coloneqq \mathcal{U}_{t-1}\setminus \mathcal{S}(h_t)$;
			\item $\mathcal{E}_t\coloneqq \mathcal{E}_{t-1}\cup\{e_t, h_t\}$.
		\end{itemize}
		\item [(a.2)] If $h_t\in \mathcal{U}_{t-1}$ \textbf{but} the edge $e_th_t$ is not retained in the percolation, then we simply declare $e_t$ and $h_t$ explored while the status of all other stubs remain unchanged. Formally we update 
		\begin{itemize}
			\item $\mathcal{A}_t\coloneqq \mathcal{A}_{t-1}\setminus \{e_t\}$;
			\item $\mathcal{U}_t\coloneqq \mathcal{U}_{t-1}\setminus \{h_t\}$;
			\item $\mathcal{E}_t\coloneqq \mathcal{E}_{t-1}\cup\{e_t,h_t\}$.
		\end{itemize}
		\item [(a.3)] If $h_t\in \mathcal{A}_{t-1}$, then we simply declare $e_t$ and $h_t$ explored while the status of all other stubs remain unchanged. Formally we update 
		\begin{itemize}
			\item $\mathcal{A}_t\coloneqq \mathcal{A}_{t-1}\setminus \{e_t,h_t\}$;
			\item $\mathcal{U}_t\coloneqq \mathcal{U}_{t-1}$;
			\item $\mathcal{E}_t\coloneqq \mathcal{E}_{t-1}\cup\{e_t,h_t\}$.
		\end{itemize}
	\end{itemize}
	\item [(b)] If $|\mathcal{A}_{t-1}|=0$ \textbf{and} $|\mathcal{U}_{t-1}|\geq 1$, we pick (in an arbitrary way) an unseen stub $e_t$, we declare active \textit{all} the unseen stubs in $\mathcal{S}(e_t)$ (thus $e_t$ at least is declared active), so that the number of active stubs is non-zero, and then we proceed as in step (a).
	\item [(c)] Finally, if $|\mathcal{A}_{t-1}|=0$ \textbf{and} $|\mathcal{U}_{t-1}|=0$, then all the stubs have been paired and we terminate the procedure.\\
\end{itemize}

\begin{obs}
	Note that, since in the above procedure we explore \textit{two half-edges} at each step, we have $\mathcal{A}_t\cup \mathcal{U}_t\neq \emptyset$ for $t\leq (dn/2)-1$ (recall that $dn$ is even) and $\mathcal{A}_{dn/2}\cup \mathcal{U}_{dn/2}=\emptyset$ (as $\mathcal{E}_{dn/2}=\{(u,i):u\in [n],i\in[d]\}$, the set of all stubs). Thus the algorithm runs for $dn/2$ steps.
\end{obs}

For $t\geq 1$ we define the event $R_t\coloneqq \{e_t h_t\in \mathbb{G}'(n,d,p)\}$ that the edge $e_t h_t$ revealed during the $t$-th step of the exploration process is retained in the percolation.

Observe that, if $|\mathcal{A}_{t-1}|\geq 1$, denoting by $\eta_t$ the number of unseen half-edges that become active at step $t$, we can write
\begin{equation}\label{mainquantity}
\eta_t=|\mathcal{A}_t|-|\mathcal{A}_{t-1}|=\mathbbm{1}_{\{h_t\in \mathcal{U}_{t-1}\}} \mathbbm{1}_{R_t}\left|\mathcal{S}(h_t)\cap \mathcal{U}_{t-1}\setminus \{h_t\}\right|-\mathbbm{1}_{\{h_t\in \mathcal{A}_{t-1}\}}-1.
\end{equation}
In words, assuming $|\mathcal{A}_{t-1}|\geq 1$, the number of active stubs at the end of step $t$ decreases by two if $h_t$ is an active stub; it decreases by one if $h_t$ is unseen and the edge $e_t h_t$ is not retained in the percolation, or if $h_t$ is the unique unseen stub incident to $v(h_t)$ and the edge $e_t h_t$ is retained in the percolation; and it increases by $m-2\in \{0,1,\dots,d-2\}$ if $v(h_t)$ has $m\in\{2,\ldots,d\}$ unseen stubs at the end of step $t-1$ and the edge $e_t h_t$ is retained in the percolation. 

Moreover, note that if $|\mathcal{A}_{t-1}|=0$, then
\begin{equation}\label{etawhenzero}
\eta_t=\left|\mathcal{S}(e_t)\cap \mathcal{U}_{t-1}\setminus \{e_t\}\right|+\mathbb{1}_{\{h_t\in \mathcal{U}_{t-1}\setminus \mathcal{S}(e_t) \}} \mathbbm{1}_{R_t}\left|\mathcal{S}(h_t)\cap \mathcal{U}_{t-1}\setminus \{h_t\}\right|-\mathbb{1}_{\{h_t\in \mathcal{S}(e_t)\}}.
\end{equation}

Set $Y_t\coloneqq |\mathcal{A}_t|$ for $t\in [dn/2]\cup\{0\}$. We define $t_0\coloneqq 0$ and recursively we set $t_i$ to be the first time after $t_{i-1}$ at which the number of active vertices hits zero; that is,
\begin{equation}
t_i\coloneqq \min\{t\geq t_{i-1}+1:Y_t=0\},
\end{equation}
for $i\geq 1$ (prior to the end of the procedure). The next result, which corresponds to Lemma $10$ in \cite{nachmias:critical_perco_rand_regular}, gives us a way to obtain an upper bound for the probability that $\CC_{\max}(\mathbb{G}'(n,d,p))$ contains less than $n^{2/3}/A$ vertices in terms of the (random) lengths of the positive excursions of the process $Y_t$. (We refer the reader to \cite{de_ambroggio:upper_component_sizes_crit_RGs} for an \textit{intuitive} explanation.)

\begin{lem}[Lemma 10 in \cite{nachmias:critical_perco_rand_regular}]\label{lemmarelcomp}
	For $i\in \mathbb{N}$, denote by $\mathcal{C}_i$ the $i$-th explored component in $\mathbb{G}'(n,d,p)$. Then $t_i-t_{i-1}\leq (d-1)|\mathcal{C}_i|$ for every $i$ (until the end of the exploration process).
\end{lem}

We now proceed in the exact same way as we did for the $\mathbb{G}(n,p)$ model. Setting $T'\coloneqq \lceil n^{2/3}/A\rceil$ and using the above lemma we bound 
\[\mathbb{P}(|\mathcal{C}_{\max}(\mathbb{G}'(n,d,p))|<n^{2/3}/A)\leq \mathbb{P}(t_i-t_{i-1}<(d-1)T' \text{ }\forall i\in [N]),\]
where we denote by $N$ the number of positive excursions of the process $Y_t$ until the end of \textbf{Algorithm 2}. Define $T\coloneqq (d-1)T$. If $t_i-t_{i-1}<T$ for all $i\in [N]$ then $(dn)/2=\sum_{i=1}^{N}(t_i-t_{i-1})<NT$ and hence 
$N>(dn)/(2T)$. Setting $L\coloneqq \lfloor (dn)/(4T)\rfloor$, we see that $N>2L$. Hence we can bound
\begin{equation*}
\mathbb{P}(t_i-t_{i-1}<T \text{ }\forall i\in [N])\leq \mathbb{P}(t_i-t_{i-1}<T \text{ }\forall L+1\leq i\leq 2L).
\end{equation*}
As we did when analyzing the $\mathbb{G}(n,d,p)$ model, we define the (good) event
\begin{equation*}
\mathcal{G}_{j}=\mathcal{G}_{j}(B)\coloneqq \{t_j\leq BLT^{1/2}\},
\end{equation*}
where $L\leq j\leq 2L-1$ and $B$ is some large positive constant (independent of $A,n,\lambda$ but possibly dependent on $d$). The events $\mathcal{G}_j$ are needed to bound from above the number of explored half-edges while revealing the cluster of a vertex in $\mathbb{G}'(n,d,p)$; such an upper bound on the number of explored stubs is needed in order to derive a meaningful lower bound for the probability that the number of active half-edges stays above zero for the required number of steps.

We continue by writing
\begin{equation*}
\mathbb{P}(t_i-t_{i-1}<T \text{ }\forall L+1\leq i\leq 2L)\leq \mathbb{P}(\{t_i-t_{i-1}<T \text{ }\forall L+1\leq i\leq 2L\}\cap\mathcal{G}_{2L-1})+\mathbb{P}(\mathcal{G}^c_{2L-1}).
\end{equation*}
The next proposition controls the probability of the (bad) event $\mathcal{G}^c_{2L-1}$ and shows that it is at most $\exp(-\Omega(A^{3/2}))$.
\begin{prop}\label{rubreg}
	There exist positive constants $B_0=B_0(d),A_0=A_0(\lambda,d)>0$ such that, for any $A\geq A_0, B\geq B_0$ and for all large enough $n$, we have
	\[\mathbb{P}(\mathcal{G}^c_{2L-1})\leq \exp(-cA^{3/2}),\]
	for some positive constant $c=c(B)>0$. 
\end{prop}
Therefore, in what follows we can focus on the probability
\begin{equation}\label{focusreg}
\mathbb{P}(\{t_i-t_{i-1}<T \text{ }\forall L+1\leq i\leq 2L\}\cap\mathcal{G}_{2L-1})
\end{equation}
Let us write $\mathcal{F}_{j}$ for the $\sigma$-algebra consisting of all the information collected by the exploration process by time $j$; then $\mathcal{F}_{t_j}$ is the information gathered by the time at which we finish the exploration of the $j$-th component. Wit the next result we show that, on the event $\mathcal{G}_{i-1}$, we have $t_i-t_{i-1}\geq T$ with probability at least of order $T^{-1/2}$.
\begin{prop}\label{mainlemmareg}
	There exist positive constants $B_0=B_0(d),A_0=A_0(\lambda,d)>0$ such that, for any $A\geq A_0,B\geq B_0$ and for all large enough $n$, on the event $\mathcal{G}_{i-1}$ we have  
	\[\mathbb{P}(t_{i}-t_{i-1}\geq T|\mathcal{F}_{t_{i-1}})\geq \frac{c_0}{T^{1/2}}\]
	for each $L+1\leq i\leq 2L$, for some constant $c_0=c_0(B,d)>0$. 
\end{prop}
It is now very easy to show that the expression in (\ref{focusreg}) is at most $\exp(-cA^{3/2})$ (for some $c>0$). Indeed, proceeding as we did for the $\mathbb{G}(n,p)$ model, we can write
\begin{multline}\label{mainexprreg}
\mathbb{P}(\{t_i-t_{i-1}<T \text{ }\forall L+1\leq i\leq 2L\}\cap\mathcal{G}_{2L-1})\\
= \mathbb{E}\Big[\mathbbm{1}_{\{t_i-t_{i-1}<T \text{ }\forall L+1\leq i\leq 2L-1\}\cap \mathcal{G}_{2L-1}}\big(1-\mathbb{P}(t_{2L}-t_{2L-1}\geq T|\mathcal{F}_{t_{2L-1}})\big)\Big].
\end{multline}
Thanks to Proposition \ref{mainlemmareg} we know that the expression on the right-hand side of (\ref{mainexprreg}) is at most
\[\big(1-c_0T^{-1/2}\big)\mathbb{P}(\{t_i-t_{i-1}<T \text{ }\forall L+1\leq i\leq 2L-1\}\cap\mathcal{G}_{2L-1}).\]
Iterating and using the classical inequality $1+x\leq e^x$ (which is valid for all $x\in \mathbb{R}$) together with the fact that $\mathcal{G}_{j}\subset \mathcal{G}_{j-1}$ for each $j$, we conclude that 
\begin{equation}\label{iteratereg}
\mathbb{P}(\{t_i-t_{i-1}<T \text{ }\forall L+1\leq i\leq 2L\}\cap\mathcal{G}_{2L})\leq \big(1-c_0T^{-1/2}\big)^L\leq \exp(-(c_0L)T^{-1/2}).
\end{equation}
Recalling the definition of $L$ and $T$, it is easy to see that $LT^{-1/2}\geq cA^{3/2}$ for all large enough $n$, for some constant $c=c(d)>0$; hence, putting all pieces together, we conclude that there is a constant $c>0$ such that, for all large enough $n$,
\begin{equation}\label{simpleprob}
\mathbb{P}(|\mathcal{C}_{\max}(\mathbb{G}'(n,d,p))|<n^{2/3}/A)\leq \exp(-cA^{3/2}).
\end{equation}
Finally, writing $\mathbb S_n$ for the event that the multigraph $\mathbb{G}'(n,d)$ underlying $\mathbb{G}'(n,d,p)$ is simple, we have
\[\mathbb{P}(|\mathcal{C}_{\max}(\mathbb{G}'(n,d,p))|<  n^{2/3}/A | \mathbb S_n )  \le \frac{\mathbb{P}(|\mathcal{C}_{\max}(\mathbb{G}'(n,d,p))|<  n^{2/3}/A)}{\mathbb{P}(\mathbb S_n)}.\]
Since $\mathbb{P}(\mathbb S_n)\sim e^{(1-d^2)/4}$ as $n\rightarrow \infty$, using (\ref{simpleprob}) we arrive at 
\[\mathbb{P}(|\mathcal{C}_{\max}(\mathbb{G}(n,d,p))|<  n^{2/3}/A)=\mathbb{P}(|\mathcal{C}_{\max}(\mathbb{G}'(n,d,p))|<  n^{2/3}/A | \mathbb S_n )\leq \exp(-cA^{3/2})\]
for some (new) constant $c=c(d)>0$. This concludes the proof of Theorem \ref{mainthm}. What is left to do is giving proofs of Lemmas \ref{rubreg} and \ref{mainlemmareg}; this is done in the following section.

\subsubsection{Proofs of Propositions \ref{rubreg} and \ref{mainlemmareg}}
In this section we establish the two auxiliary facts which were stated, without proof, while proving the main theorem for the $\mathbb{G}(n,d,p)$ random graph.

\paragraph{Proof of Proposition \ref{rubreg}.}
Recall that $\mathcal{G}_{2L-1}\coloneqq \{t_{2L-1}\leq BLT^{1/2}\}$. To bound from above the probability of the complementary event, we employ the same strategy used for the $\mathbb{G}(n,p)$ model. Let $S_0\coloneqq d$ and define recursively $S_i\coloneqq S_{i-1}+\hat{\eta}_{i}$, where the random variable $\hat{\eta}_i$ is defined (for $i\leq (dn)/2$) by
\begin{equation}\label{etahat}
\hat{\eta}_i\coloneqq \left|\mathcal{S}(e_i)\cap \mathcal{U}_{i-1}\setminus \{e_i\}\right|+\mathbb{1}_{\{h_i\in \mathcal{U}_{i-1}\setminus \mathcal{S}(e_i) \}} \mathbbm{1}_{R_i}\left|\mathcal{S}(h_i)\cap \mathcal{U}_{i-1}\setminus \{h_i\}\right|-\mathbb{1}_{\{h_i\in \mathcal{A}_{i-1}\cup \mathcal{S}(e_i)\}}-1. 
\end{equation}
Recalling (\ref{mainquantity}) and (\ref{etawhenzero}) we immediately see that $\hat{\eta}_i=\eta_i$ when either $1\leq i\leq t_1$ or $i\in (t_{j-1}+1,t_j]$ for some $1<j\leq N$, whereas $\hat{\eta}_{t_{j-1}+1}=\eta_{t_{j-1}+1}-1$ for $1<j\leq N$. Indeed, if $1\leq i\leq t_1$, then $e_i$ is taken from the set of active half-edge,  so that $\mathcal{S}(e_i)\cap \mathcal{U}_{i-1}\setminus \{e_i\}=\mathcal{S}(e_i)\cap \mathcal{U}_{i-1}=\emptyset$ (whence $|\mathcal{S}(e_i)\cap \mathcal{U}_{i-1}\setminus \{e_i\}|=0$) and also $\mathcal{A}_{i-1}\cup \mathcal{S}(e_i)=\mathcal{A}_{i-1}, \mathcal{U}_{i-1}\setminus \mathcal{S}(e_i)=\mathcal{U}_{i-1}$; thus $\hat{\eta}_i=\eta_i$. The other cases are treated similarly.
Then $(-\min_{j\in [t]}S_j)+1$ denotes the number of disjoint clusters that have been fully explored
by time $t\in [n]$. When $Y_t=0$, we have fully explored a
cluster, and this occurs precisely when $S_t=\min_{j\in [t-1]}S_j-1$, i.e., the process $(S_t)_{t\in [n]}$ reaches a new all-time minimum. 

Proceeding in the exact same way as we did in the proof of Lemma \ref{rub}, we arrive at
\begin{equation}\label{fromherediff}
\mathbb{P}(\mathcal{G}_{2L-1}^c)\leq \mathbb{P}\Big(\max_{j\leq \lfloor BLT^{1/2}\rfloor}(-S_j)<2L-1\Big). 
\end{equation}
To continue, we split $\hat{\eta}_t$ into its contributions by introducing
\begin{equation*}\label{xi1}
X_{i,1}\coloneqq \mathbb{1}_{\{h_i\in \mathcal{U}_{i-1}\setminus \mathcal{S}(e_i) \}} \mathbbm{1}_{R_i}\left|\mathcal{S}(h_i)\cap \mathcal{U}_{i-1}\setminus \{h_i\}\right|-1
\end{equation*} 
and
\begin{equation*}\label{xi2}
X_{i,2}\coloneqq \left|\mathcal{S}(e_i)\cap \mathcal{U}_{i-1}\setminus \{e_i\}\right|-\mathbb{1}_{\{h_i\in \mathcal{A}_{i-1}\cup \mathcal{S}(e_i)\}}.
\end{equation*}
Then clearly $\hat{\eta}_i=X_{i,1}+X_{i,2}$ and hence we can write 
\[S_j=d+\sum_{i=1}^{j}\hat{\eta}_i=d+\sum_{i=1}^{j}X_{i,1}+\sum_{i=1}^{j}X_{i,2}.\]
Observe that, on the event within the probability on the right-hand side of (\ref{fromherediff}), we have
\[\sum_{i=1}^{ \lfloor BLT^{1/2}\rfloor}\left|\mathcal{S}(e_i)\cap \mathcal{U}_{i-1}\setminus \{e_i\}\right|\leq d\sum_{i=1}^{ \lfloor BLT^{1/2}\rfloor}\mathbbm{1}_{\{|\mathcal{A}_{i-1}|=0\}}=d \max_{j\in \lfloor BLT^{1/2}\rfloor}(-S_j)<2dL, \]
where the first inequality follows from the observation that $\left|\mathcal{S}(e_i)\cap \mathcal{U}_{i-1}\setminus \{e_i\}\right|>0$ if, and only if, $\mathcal{A}_{i-1}=\emptyset$ (otherwise $e_i$ would be selected from the set of active stubs and hence we would have $\mathcal{S}(e_i)\cap \mathcal{U}_{i-1}=\emptyset$), in which case it is at most $d$.  

Therefore, the probability on the right-hand side of (\ref{fromherediff}) equals
\begin{equation}\label{onev}
\mathbb{P}\Big(d+\sum_{i=1}^{j}X_{i,1}>-2L+1-\sum_{i=1}^{j}X_{i,2} \text{ }\forall j\leq \lfloor BLT^{1/2}\rfloor,\sum_{i=1}^{ \lfloor BLT^{1/2}\rfloor}\left|\mathcal{S}(e_i)\cap \mathcal{U}_{i-1}\setminus \{e_i\}\right|< 2dL\Big).
\end{equation}
Since $X_{i,2}\leq \left|\mathcal{S}(e_i)\cap \mathcal{U}_{i-1}\setminus \{e_i\}\right|\in \mathbb{N}_0$, on the (second) event in (\ref{onev}) we have
\[\sum_{i=1}^{j}X_{i,2}\leq \sum_{i=1}^{j}\left|\mathcal{S}(e_i)\cap \mathcal{U}_{i-1}\setminus \{e_i\}\right|\leq \sum_{i=1}^{\lfloor BLT^{1/2}\rfloor}\left|\mathcal{S}(e_i)\cap \mathcal{U}_{i-1}\setminus \{e_i\}\right|<2dL \text{ for each }j\leq \lfloor BLT^{1/2}\rfloor.\]
Whence the probability in (\ref{onev}) is at most
\[\mathbb{P}\Big(d+\sum_{i=1}^{j}X_{i,1}>-4dL\text{ }\forall j\leq \lfloor BLT^{1/2}\rfloor\Big)\leq \mathbb{P}\Big(d+\sum_{i=1}^{\lfloor BLT^{1/2}\rfloor}X_{i,1}>-4dL\Big).\]
In order to bound the probability on the right-hand side of the last inequality, we follow \cite{de_ambroggio_roberts:near_critical_RRG}. 

Denote by $\mathcal{V}^{(d)}_k$ the set of vertices having $d$ unseen stubs at step $k$ of the exploration process; we call these vertices \textit{fresh}. We also introduce the event $F_k\coloneqq \{v(h_k)\in \mathcal{V}^{(d)}_{k-1}\}$ that $v(h_k)$ is fresh. Then, defining
\[\eta'_i\coloneqq \mathbbm{1}_{R_i}(d-2)+\mathbbm{1}_{R_i}\mathbbm{1}_{F_i}-1,\]
we have $X_{i,1}\leq \eta'_i$. Indeed, the inequality is actually an identity whenever $v(h_i)$ has $m\in \{{d-1,d}\}$ unseen stubs attached to it and $R_i$ occurs, whereas $X_{i,1}<\eta'_i$ in all the other cases. 

Then we can bound
\begin{equation}\label{jkj}
\mathbb{P}\Big(d+\sum_{i=1}^{\lfloor BLT^{1/2}\rfloor}X_{i,1}>-4dL\Big)\leq \mathbb{P}\Big(d+\sum_{i=1}^{\lfloor BLT^{1/2}\rfloor}\eta'_i>-4dL\Big).
\end{equation}
In order to bound from above the last probability we substitute (as in\cite{de_ambroggio_roberts:near_critical_RRG}) the dependent indicators $\mathbbm{1}_{F_i}$ that appear in the definition of $\eta'_i$ with independent $\{0,1\}$-valued random variables. To this end notice that, conditional on everything that occurred up to the end of step $i-1$ in the exploration process, vertex $v(h_i)$ is fresh with probability 
\begin{align*}
\frac{d \big|\mathcal{V}^{(d)}_{i-1}\big|}{dn-2(i-1)-1}.
\end{align*}
Thus, in order to substitute the $\mathbbm{1}_{F_i}$ with (larger) independent indicator random variables, we need an upper bound for the number of fresh vertices that we expect to observe at each step $i\leq \lfloor BLT^{1/2}\rfloor$ of the exploration process.

The next result, which is basically the same as Lemma $3.1$ in \cite{de_ambroggio_roberts:near_critical_RRG}, states that it is very unlikely to have more than $n-1-i+i^2/2n$ fresh vertices at the $i$-th step of the exploration process.
\begin{lem}\label{numberfresh}
	Let $a_n(i)\coloneqq n-1-i+i^2/2n$. Then, for every $m\geq 1$ and all large enough $n$, we have that
	\begin{equation*}\label{boundonfresh}
	\mathbb{P}\left(\exists i\leq \lfloor BLT^{1/2}\rfloor: |\mathcal{V}^{(d)}_{i}|> a_n(i)+ m\right)\leq BLT^{1/2}\exp\big(-c\frac{mn^{1/2}}{LT^{1/2}}\big)
	\end{equation*}
	where $c=c(d)$ is some finite constant that depends only on $d$.
\end{lem}
Arguing as in \cite{de_ambroggio_roberts:near_critical_RRG} we observe that, if $|\mathcal{V}^{(d)}_{i}|\leq a_n(i)+m$ for all $i\leq \lfloor BLT^{1/2}\rfloor$, then we can write
\begin{equation}\label{theetaprime}
\eta'_i=\mathbbm{1}_{R_i}(d-2)+\mathbbm{1}_{R_i}\mathbbm{1}_{F_i}\mathbbm{1}_{\left\{|\mathcal{V}^{(d)}_{i-1}|\leq a_n(i-1)+m \right\}} -1\eqqcolon \eta''_i, \text{ for }i\leq \lfloor BLT^{1/2}\rfloor.
\end{equation}
Then we can bound from above the probability on the right-hand side of (\ref{jkj}) by
\begin{align*}
\mathbb{P}\Big(d+\sum_{i=1}^{\lfloor BLT^{1/2}\rfloor}\eta'_i>-4dL, \text{ }|\mathcal{V}^{(d)}_{i}|&\leq  a_n(i)+ m \text{ }\forall i\leq \lfloor BLT^{1/2}\rfloor \Big)+BLT^{1/2}\exp\big(-c\frac{mn^{1/2}}{LT^{1/2}}\big)\\
&\leq \mathbb{P}\Big(d+\sum_{i=1}^{\lfloor BLT^{1/2}\rfloor}\eta''_i>-4dL\Big)+BLT^{1/2}\exp\big(-c\frac{mn^{1/2}}{LT^{1/2}}\big).
\end{align*}
Conditional on everything that has occurred up to the end of step $i-1$ in the exploration process, the random variable $\mathbbm{1}_{R_i}\mathbbm{1}_{F_i}\mathbbm{1}_{\left\{|\mathcal{V}^{(d)}_{i-1}|\leq a_n(i-1)+m \right\}}$ which appears in (\ref{theetaprime}) equals $1$ with probability 
\begin{align*}
p\mathbbm{1}_{\left\{|\mathcal{V}^{(d)}_{i-1}|\leq a_n(i-1)+m \right\}}\frac{d \left|\mathcal{V}^{(d)}_{i-1}\right|}{dn-2(i-1)-1}\leq p \frac{d (a_n(i-1)+m)}{dn-2(i-1)-1}.
\end{align*}
Therefore, arguing precisely as in \cite{de_ambroggio_roberts:near_critical_RRG}, if $(U_i)_{i\geq 1}$ is an iid sequence of $U([0,1])$ random variables (also independent from all other random quantities involved), then
\begin{equation}\label{zi}
\mu_{i}\coloneqq \mathbbm{1}_{R_i}(d-2)+\mathbbm{1}_{R_i}\mathbbm{1}_{\left\{U_i\leq  \frac{d(a_n(i-1)+m)}{dn-2(i-1)-1}\right\}}-1
\end{equation}
defines a sequence of \textit{independent} random variables that should be larger than the $\eta''_i$. The next result can be proved in the exact same way as Proposition $3.3$ in \cite{de_ambroggio_roberts:near_critical_RRG}.
\begin{lem}\label{mainpropforupper}
	Let $(U_i)_i$ be a sequence of iid random variables, also independent from all other random variables involved, with $U_1=_d U([0,1])$. For each $i\leq \lfloor BLT^{1/2}\rfloor$ let $\mu_i$ be as in (\ref{zi}) above.
	Then
	\begin{equation*}
	\mathbb{P}\Big(d+\sum_{i=1}^{t}\eta''_i>-4dL\Big)
	\leq \mathbb{P}\Big(d+\sum_{i=1}^{t}\mu_i>-4dL\Big).
	\end{equation*}
\end{lem}
In order to provide an upper bound for the above quantity we would like to turn the independent (but not identically distributed) $\mu_i$ into iid~random variables $\xi_i$. To this end, keeping in mind the definition of $\mu_i$ given in \eqref{zi}, following \cite{de_ambroggio_roberts:near_critical_RRG} we define
\begin{equation}\label{bi}
\mu'_i\coloneqq \mathbbm{1}_{R_i}\mathbbm{1}_{\left\{U_i> \frac{d(a_n(i-1)+m)}{dn-2(i-1)-1}\right\}}
\end{equation}
and set, for $i\leq \lfloor BLT^{1/2}\rfloor$,
\begin{equation}\label{xixi}
\xi_i\coloneqq \mu_i+\mu'_i=\mathbbm{1}_{R_i}(d-2)+\mathbbm{1}_{R_i} - 1 = \mathbbm{1}_{R_i}(d-1) - 1.
\end{equation}
By adding the (random) sums $\sum_{i=1}^{\lfloor BLT^{1/2}\rfloor}\mu'_i$ to the $\sum_{i=1}^{\lfloor BLT^{1/2}\rfloor}\mu_i$ we can rewrite the probability which appears in Lemma \ref{mainpropforupper} as 
\begin{equation}\label{eqq}
\mathbb{P}\Big(d+\sum_{i=1}^{\lfloor BLT^{1/2}\rfloor}\xi_i>\sum_{i=1}^{\lfloor BLT^{1/2}\rfloor}\mu'_i-4dL\Big).
\end{equation}
In order to control the (random) sum $\sum_{i=1}^{\lfloor BLT^{1/2}\rfloor}\mu'_i$ we use 
Lemma $3.4$ in \cite{de_ambroggio_roberts:near_critical_RRG}, which we recall here for the reader's convenience. Note that the $m$ which appears in the statement comes from the definition of $\mu'_i$.

\begin{lem}\label{arub}
	Define 
	\begin{equation}\label{qt}
	q(t)=q_{n,d}(t)\coloneqq p\Big(1-\frac{2}{d}\Big)\frac{t(t-1)}{2n}, \text{ } t\leq \lfloor BLT^{1/2}\rfloor, t\in \mathbb{N}.
	\end{equation}
	Then, for all large enough $n$, if $m=O(n^{1/2})$ we have  
	\begin{equation}\label{mbm}
	\mathbb{P}\Big(\sum_{i=1}^{\lfloor BLT^{1/2}\rfloor}\mu'_i\leq q(\lfloor BLT^{1/2}\rfloor)-h\Big) \leq C e^{-\frac{hn^{1/2}}{\lfloor BLT^{1/2}\rfloor}},
	\end{equation} 
	where $C=C(B,d)$ is a constant that depends on $d,B$. 
\end{lem}
(We remark that the factor $T$ which multiplies the exponential term in Lemma $3.4$ of \cite{de_ambroggio_roberts:near_critical_RRG} should not be there.) Thanks to the last lemma applied with $h\coloneqq L$ and recalling the definition of $\xi_i$ given in (\ref{xixi}) above, if $n$ is large enough we can easily upper bound the probability in (\ref{eqq}) by
\begin{align*}
\mathbb{P}\Big(\sum_{i=1}^{\lfloor BLT^{1/2}\rfloor}\xi_i>q(\lfloor BLT^{1/2}\rfloor)-L-4dL-d\Big)&+\mathbb{P}\Big(\sum_{i=1}^{\lfloor BLT^{1/2}\rfloor}\mu'_i\leq q(\lfloor BLT^{1/2}\rfloor)-L\Big)\\
&\leq \mathbb{P}\Big(\sum_{i=1}^{\lfloor BLT^{1/2}\rfloor}\mathbbm{1}_{R_i}>\lfloor BLT^{1/2}\rfloor p+y\Big)+ e^{-c\frac{n^{1/2}}{T^{1/2}}},
\end{align*}
where we set
\[y\coloneqq \Big(q(\lfloor BLT^{1/2}\rfloor)-5dL\Big)(d-1)^{-1}+\lfloor BLT^{1/2}\rfloor\big((d-1)^{-1}-p\big).\]
Recalling that $L\sim dAn^{1/3}/4(d-1),T\sim (d-1)n^{2/3}/A$ and $p=(1+\lambda n^{-1/3})/(d-1)$ we obtain (for all large enough $n$)
\[q(\lfloor BLT^{1/2}\rfloor)-5dL\geq L\Big(\frac{B^2LT(d-2)}{4nd(d-1)}-5d\Big)\geq L\Big(\frac{B^2(d-2)}{18(d-1)}-5d\Big).\]
Moreover, 
\[\lfloor BLT^{1/2}\rfloor\big((d-1)^{-1}-p\big)\geq -|\lambda| BLT^{1/2}(d-1)^{-1}n^{-1/3}\geq -2|\lambda| BL(d-1)^{-1/2}A^{-1/2}\]  
so that (if $n$ is large enough)
\[y\geq (d-1)^{-1}L\Big(\frac{B^2(d-2)}{18(d-1)}-5d-2(d-1)^{1/2}|\lambda| BA^{-1/2}\Big)\geq c'L\]
for some constant $c'=c'(B,d)>0$, provided $B=B(d)$ is sufficiently large and $A\geq A_0$ for some large enough $A_0=A_0(\lambda,d)>0$. Hence, since $\sum_{i=1}^{\lfloor BLT^{1/2}\rfloor}\mathbbm{1}_{R_i}$ has the $\text{Bin}\big(\lfloor BLT^{1/2}\rfloor,p\big)$ law, using Lemma \ref{Bol2} we arrive at
\begin{align*}
\mathbb{P}\Big(\sum_{i=1}^{\lfloor BLT^{1/2}\rfloor}\mathbbm{1}_{R_i}>\lfloor BLT^{1/2}\rfloor p+y\Big)&\leq \mathbb{P}\Big(\sum_{i=1}^{\lfloor BLT^{1/2}\rfloor}\mathbbm{1}_{R_i}>\lfloor BLT^{1/2}\rfloor p+c'L\Big)\\
&\leq \exp\big(-c\frac{L}{T^{1/2}}\big)\leq \exp(-cA^{3/2}),
\end{align*}
for some constant $c=c(d)>0$. Putting all pieces together, we have shown that, if $m=O(n^{1/2})$, then for all large enough $A,n$ we have
\begin{equation*}
\mathbb{P}(\mathcal{G}^c_{2L})\leq \exp\big(-cA^{3/2}\big)+\exp\big(-c\frac{n^{1/2}}{T^{1/2}}\big)+BLT^{1/2}\exp\big(-c\frac{mn^{1/2}}{LT^{1/2}}\big).
\end{equation*}
Since, by assumption, $A=O(n^{1/6})$, taking $m=L$ (which is indeed $O(n^{1/2})$ by our assumption on $A$) we see that the second and third terms on the right-hand side of the last inequality are both at most $\exp(-\Omega(A^{3/2}))$, thus completing the proof of the proposition.

\paragraph{Proof of Proposition \ref{mainlemmareg}.} Here the goal is to show that there is a constant $c_0=c_0(B,d)>0$ such that, for each $L+1\leq i\leq 2L$, on the event $\mathcal{G}_{i-1}$ it holds that
\[\mathbb{P}(t_{i}-t_{i-1}\geq T|\mathcal{F}_{t_{i-1}})\geq \frac{c_0}{T^{1/2}}\]
for all large enough $n$, where we recall that $\mathcal{F}_{t_{i-1}}$ is the $\sigma$-algebra consisting of all the information gathered by the exploration process until time $t_{i-1}$ (i.e. until the moment at which we finish revealing the $(i-1)$-th cluster). Then, setting $Y^i_t\coloneqq Y_{t_{i-1}+t}$ and $\eta^i_t\coloneqq \eta_{t_{i-1}+t}$ to simplify notation, we can write
\begin{equation}\label{prrrreg}
\mathbb{P}(t_{i}-t_{i-1}\geq T|\mathcal{F}_{t_{i-1}})=\mathbb{P}(Y^i_t> 0 \text{ }\forall t \leq T-1|\mathcal{F}_{t_{i-1}})\geq \mathbb{P}\Big(d_i+\sum_{k=1}^{t}\eta^i_k> 0 \text{ }\forall t \leq T|\mathcal{F}_{t_{i-1}}\Big),
\end{equation} 
where $1\leq d_i\geq d$ represents the number of unseen stubs attached to the node from which we start exploring (at time $t_{i-1}+1$) the $i$-th cluster. Recall that, if $Y^i_{k-1}\geq 1$, then
\begin{equation}
\eta^i_k=\mathbbm{1}_{\{h^i_k\in \mathcal{U}^i_{k-1}\}} \mathbbm{1}_{R^i_k}\left|\mathcal{S}(h^i_k)\cap \mathcal{U}^i_{k-1}\setminus \{h^i_k\}\right|-\mathbbm{1}_{\{h^i_k\in \mathcal{A}^i_{k-1}\}}-1,
\end{equation}
where we denote by $h^i_k$ the half-edge selected at time $t_{i-1}+k$, whereas $\mathcal{U}^i_{k-1},\mathcal{A}^i_{k-1}$ are the sets of unseen stubs and active half-edges at time $t_{i-1}+k-1$, respectively. We further denote by $\mathcal{V}^{(d)}_{k,i}$ the set of (fresh) vertices with $d$ unseen stubs at time $t_{i-1}+k$ and by $F^i_k \coloneqq \{v(h^i_k)\in \mathcal{V}^{(d)}_{k-1,i}\}$ the event that $v(h^i_k)$ is fresh.

Since we want to bound the probability in (\ref{prrrreg}) from below, we need to approximate the $\eta^i_k$ with \textit{smaller} random variables, sufficiently close to the former but easier to be dealt with. To this end, following \cite{de_ambroggio_roberts:near_critical_RRG}, we define
\begin{equation}\label{delta'def}
\delta^i_k \coloneqq \mathbbm{1}_{R^i_k}\mathbbm{1}_{F^i_k}(d-1) - \mathbbm{1}_{\{h^i_k\in \mathcal{A}^i_{k-1}\}} - 1
\end{equation}
and we observe that, on the event where $d_i+\sum_{k=1}^{t}\eta^i_k> 0$ for all $t \leq T$, we have
\begin{equation}\label{nudelta}
\eta^i_k \ge \delta^i_k.
\end{equation}
Indeed, if at step $t_{i-1}+k$ we pick $h^i_k$ from the set of active vertices, then $\eta^i_k=\delta^i_k$; similarly, if $h^i_k$ is incident to a fresh vertex (whence in particular $h^i_k$ is unseen), then again $\eta^i_k=\delta^i_k$. In all other cases (i.e. when $h^i_k$ is unseen but the vertex to which is incident has less than $d$ unseen half-edges incident to it), then $\eta^i_k\geq \delta^i_k$. 

Then, using (\ref{nudelta}), we bound from below the probability in (\ref{prrrreg}) by 
\begin{equation}\label{withdelta}
\mathbb{P}\Big(d_i+\sum_{k=1}^{t}\delta^i_k> 0 \text{ }\forall t \leq T|\mathcal{F}_{t_{i-1}}\Big).
\end{equation}
To establish the required lower bound for this probability, we proceed exactly as in Section $4.2$ of \cite{de_ambroggio_roberts:near_critical_RRG}. Actually, in our case the proof is even simpler, because we are \textit{not} conditioning on the event $\mathbb{S}_n$ that the multigraph $\mathbb{G}'(n,d)$ underlying $\mathbb{G}'(n,d,p)$ is simple (removing the conditioning on $\mathbb{S}_n$ was the content of Section $4.2.1$ in \cite{de_ambroggio_roberts:near_critical_RRG}). We do not give all the details here but nevertheless, for the sake of clarity, we carefully explain the required adaptations needed from \cite{de_ambroggio_roberts:near_critical_RRG} to the current setting. 

First of all note that the probability in (\ref{withdelta}) is at least 
\[\mathbb{P}\Big(d_i+\sum_{k=1}^{t}\delta^i_k> t^{\gamma} \text{ }\forall t \leq T|\mathcal{F}_{t_{i-1}}\Big),\]
where $0<\gamma<1/2$ (for $\gamma\geq 1/2$, the lower bound that we obtain is not of the required order).
Then the idea is to replace the $\delta^i_k$ with the simpler random variables
\begin{equation}\label{chichirichi}
\chi^i_k \coloneqq \mathbbm{1}_{R^i_k}\mathbbm{1}_{F^i_k}(d-1) - 1
\end{equation}
and then to show that, for any $\gamma\in(0,1/2)$, we can find a constant $c_0$ (which depends on $B,d$) such that, for all sufficiently large $A,n$, on the event $\mathcal{G}_{i-1}$ we can bound
\begin{equation}\label{delprimeabovegamma}
\mathbb{P}\Big(d_i+\sum_{i=1}^t \chi^i_k  \ge t^\gamma \text{ }\forall t\leq T\mid \mathcal{F}_{t_{i-1}}\Big) \ge c'_0T^{-1/2}
\end{equation}
for some constant $c'_0=c'_0(B,d)>0$. Subsequently we show that for $\gamma\in(0,1/2)$ (and for all large enough $A,n$), on the event $\mathcal{G}_{i-1}$ we have
\begin{equation}\label{activesumbelowgamma}
\mathbb{P}\Big(\exists t\leq (t_i-t_{i-1})\wedge T : \sum_{i=1}^t \mathbbm{1}_{\{h^i_k\in \mathcal{A}^i_{k-1}\}} \ge t^\gamma\mid \mathcal{F}_{i-1}\Big) \ll T^{-1/2}.
\end{equation}
Combining \eqref{delprimeabovegamma} and \eqref{activesumbelowgamma} we obtain the desired lower bound on (\ref{withdelta}); this is the content of Lemma \ref{stepuno} below. (As in \cite{de_ambroggio_roberts:near_critical_RRG}, we note that the choice of $\gamma$ is not important above; one may choose, for example, $\gamma=1/4$ in both \eqref{delprimeabovegamma} and \eqref{activesumbelowgamma}. We retain the general $\gamma$ in the proofs since this is no extra work.)
 
\begin{lem}\label{stepuno}
	Let $\chi^i_k$ be as in (\ref{chichirichi}). Suppose that \eqref{delprimeabovegamma} and \eqref{activesumbelowgamma} hold. Then, there exists a constant $c_0=c_0(B,d)>0$ such that, for all sufficiently large $A,n$, on the event $\mathcal{G}_{i-1}$ we have
	\[\mathbb{P}\Big(d_i+\sum_{k=1}^{t}\eta^i_k> 0 \text{ }\forall t \leq T|\mathcal{F}_{t_{i-1}}\Big)\geq c_0T^{-1/2},\]
	for each $L+1\leq i\leq 2L$. 
\end{lem}
\begin{proof}
	We closely follow the argument given in Section $4.2.6$ in \cite{de_ambroggio_roberts:near_critical_RRG}. Note that
	\begin{equation}\label{kk}
	\mathbb{P}\Big(d_i+\sum_{k=1}^{t}\delta^i_k> 0 \text{ }\forall t \leq T|\mathcal{F}_{t_{i-1}}\Big)=\mathbb{P}\Big(d_i+\sum_{k=1}^{t}\delta^i_k> 0 \text{ }\forall t \leq (t_i-t_{i-1})\wedge T|\mathcal{F}_{t_{i-1}}\Big)
	\end{equation}
	because, if $d_i+\sum_{k=1}^{t}\delta^i_k> 0$ for all $t\leq T$, then (thanks to (\ref{nudelta})) also $d_i+\sum_{k=1}^{t}\eta^i_k> 0$ for all $t\leq T$ and hence we must have that $t_i-t_{i-1}> T$ (so that $(t_i-t_{i-1})\wedge T=T$). 
	
	Now the probability on the right-hand side of (\ref{kk}) is clearly at least
	\[\mathbb{P}\Big(d_i+\sum_{k=1}^{t}\chi^i_k> t^{\gamma} \text{ }\forall t \leq (t_i-t_{i-1})\wedge T, \sum_{i=1}^t \mathbbm{1}_{\{h^i_k\in \mathcal{A}^i_{k-1}\}} < t^\gamma \text{ }\forall t \leq (t_i-t_{i-1})\wedge T \mid \mathcal{F}_{t_{i-1}}\Big),\] 
	and the last probability is at least the difference between (\ref{delprimeabovegamma}) and (\ref{activesumbelowgamma}), which is at least (for all large enough $n$) $c_0T^{-1/2}$ with $c_0\coloneqq c'_0/2$, as desired.
\end{proof}
There remains to establish (\ref{delprimeabovegamma}) and (\ref{activesumbelowgamma}). To this end, we continue with the following Lemma, whose proof follows the exact same steps as Proposition $4.11$ in \cite{de_ambroggio_roberts:near_critical_RRG} and is therefore omitted.
\begin{lem}\label{stepdue}
	Define $\Delta^i_k\coloneqq \mathbbm{1}_{R^i_k}\mathbbm{1}_{\{U_k\leq 1-(4BLT^{1/2})/n\}}(d-1)-1$ for $k\leq T$, where $(U_k)_{k\in \mathbb{N}}$ is a sequence of iid $\text{Unif}([0,1])$ random variables, independent of everything else. Then
    \[\mathbb{P}\Big(\sum_{k=1}^t \chi^i_k  \ge t^\gamma \text{ }\forall t\leq T\mid \mathcal{F}_{t_{i-1}}\Big)\geq \mathbb{P}\Big(\sum_{i=1}^t \Delta^i_k  \ge t^\gamma \text{ }\forall t\leq T\Big).\] 
\end{lem}
The intuition here is that, since $|\mathcal{V}^{(d)}_{k,i}|\ge n-1-2(t_{i-1}+k)$ for each $k\ge 1$ (as at most two vertices can be removed from the set of fresh nodes at each step), the (conditional) probability that $h^i_k$ is fresh when $k\leq T$ is, on the event $\mathcal{G}_{i-1}$, equal to
\[\frac{d|\mathcal{V}^{(d)}_{i,k-1}|}{dn-2(t_{i-1}-k-1)-1} \ge \frac{d(n-2t_{i-1}-2k)}{dn} = 1 - \frac{2(t_{i-1}+k)}{n} \ge  1-\frac{4BLT^{1/2}}{n},\]
provided that $n$ is large enough. Thus, setting $\mathcal{F}_{i-1}(k-1)\coloneqq \mathcal{F}_{t_{i-1}+k-1}$, for all large enough $n$ we obtain
\[\mathbb{P}( F^i_k |\mathcal{F}_{i-1}(k-1)) \ge 1-\frac{4BLT^{1/2}}{n} = \mathbb{P}(U_k\leq 1-(4BLT^{1/2})/n).\]
which then (recalling the definitions of $\chi^i_k$ and $\Delta^i_k$) shows that $\chi^i_k$ stochastically dominates $\Delta^i_k$ for each $k$. 
(We remark that the $1-(4BLT^{1/2})/n$ in the definition of $\Delta^i_k$ given here translates into $1-T'/n$ in Proposition $4.11$ of \cite{de_ambroggio_roberts:near_critical_RRG}.) 

The next step is the following lemma, which is almost the same as Lemma $4.13$ in \cite{de_ambroggio_roberts:near_critical_RRG}.
\begin{lem}\label{steptre}
	Let $\Delta^i_k$ be as in Lemma \ref{stepdue}. Let $(D_k)_{k\leq T}$ be a sequence of iid random variables with $\mathbb{P}(D_k=d-2)=1/(d-1)=1-\mathbb{P}(D_k=-1)$. There exist constants $A_0=A_0(\lambda,d)>0$ and $c=c(d)>0$ such that, for any $A\geq A_0, \varepsilon\in (0,1)$ and for all large enough $n$, we have
	\[\mathbb{P}\Big(\sum_{k=1}^t \Delta^i_k  \ge t^\gamma \text{ }\forall t\leq T\Big)\geq c\mathbb{P}\Big(\sum_{k=1}^t D_k  \ge t^\gamma \text{ }\forall t\leq T,\sum_{k=1}^{T} D_k\in [\varepsilon T^{1/2}, T^{1/2}/\varepsilon]\Big).\]
\end{lem}
\begin{proof}
	This is basically Lemma $4.13$ in \cite{de_ambroggio_roberts:near_critical_RRG} (whence we won't provide all the details), the proof of which is based on an exponential change of measure. We set
	\[\frac{\text{d}\hat{\mathbb{P}}}{\text{d}\mathbb{P}}\Big|_{\mathcal{F}_t}\coloneqq \frac{\exp\big(\nu \sum_{k=1}^{t}\Delta^i_k\big)}{\mathbb{E}\big[\exp\big(\nu \sum_{k=1}^{t}\Delta^i_k\big)\big]}, \text{ }t\in \mathbb{N}_0,\]
	where $\nu$ chosen in such a way to obtain the desired distribution for the $\Delta^i_k$. In particular, in our case we set 
	\begin{equation}\label{newnu}
	\nu\coloneqq \frac{1}{d-1}\log\Big(1-p\big(1-\frac{4BLT^{1/2}}{n}\big)\Big)-\frac{1}{d-1}\log\Big((d-2)p\big(1-\frac{4BLT^{1/2}}{n}\big)\Big).
	\end{equation}
	(The term $4BLT^{1/2}$ here corresponds to the term $T'$ in Lemma $4.13$ of \cite{de_ambroggio_roberts:near_critical_RRG}.) With this choice of $\nu$, the sequence $(\Delta^i_k)_{k\in [T]}$ is iid with
	\begin{equation}\label{lawwewant}
	\hat{\mathbb{P}}(\Delta^i_k=d-2)=\frac{1}{d-1}=1-\hat{\mathbb{P}}(\Delta^i_k=-1),
	\end{equation}
	which is the required distribution. Proceeding in the exact same way as in the proof of Lemma $4.13$ in \cite{de_ambroggio_roberts:near_critical_RRG}, we arrive at 
	\[\mathbb{P}\big(\sum_{k=1}^t \Delta^i_k  \ge t^\gamma \text{ }\forall t\leq T\big)\geq e^{-|\nu|T^{1/2}\varepsilon^{-1}}\mathbb{E}[e^{\nu \Delta^i_1}]^T\hat{\mathbb{P}}\big(\sum_{k=1}^t \Delta^i_k  \ge t^\gamma \text{ }\forall t\leq T, \sum_{k=1}^T \Delta^i_k\in [\varepsilon T^{1/2},T^{1/2}/\varepsilon]\big).\]
	There remains to check that 
	\begin{equation}\label{whatisleft}
	e^{-|\nu|T^{1/2}\varepsilon^{-1}}\mathbb{E}[e^{\nu \Delta^i_1}]^T\geq c
	\end{equation}
	for some constant $c=c(d)>0$. To this end, we start by noticing that (by definition of $\nu$) 
	\[\mathbb{E}[e^{\nu \Delta^i_1}]=\exp\Big(O\Big(\frac{\lambda^2}{n^{2/3}}+\frac{\lambda BLT^{1/2}}{n^{4/3}}+\frac{(BLT^{1/2})^2}{n^2}\Big)\Big);\]
	whence, recalling that $T\sim (d-1)n^{2/3}/A$ and $L=\Theta(An^{1/3})$, we see that
	\[T\Big(\frac{\lambda^2}{n^{2/3}}+\frac{\lambda BLT^{1/2}}{n^{4/3}}+\frac{(BLT^{1/2})^2}{n^2}\Big)=\Theta\big( \frac{\lambda^2}{A}+\frac{\lambda}{A^{1/2}}+1\big).\]
	Therefore $\mathbb{E}[e^{\nu \Delta^i_1}]^T\geq c$ provided that $A\geq A_0$ for a large enough $A_0=A_0(\lambda,d)>0$. Moreover, 
	\[|\nu|T^{1/2}\leq C\Big(1+\frac{|\lambda|}{A^{1/2}}+\frac{A^{1/2}}{n^{1/3}}\Big)\]
	and so (since $A\ll n^{2/3}$) we see that $e^{-|\nu|T^{1/2}\varepsilon^{-1}}\geq c>0$ too (provided that $A\geq A_0$ for a large enough $A_0=A_0(\lambda,d)>0$), as required.
\end{proof}
Finally, using Lemma $4.14$ in \cite{de_ambroggio_roberts:near_critical_RRG} (with $T$ in place of $T'$), we obtain 
\begin{equation}\label{iii}
\mathbb{P}\big(\sum_{k=1}^t D_k  \ge t^\gamma \text{ }\forall t\leq T,\sum_{k=1}^{T} D_k\in [\varepsilon T^{1/2}, T^{1/2}/\varepsilon]\big)\geq cT^{-1/2}.
\end{equation}
Combining Lemmas \ref{stepdue} and \ref{steptre} together with (\ref{iii}) we conclude that indeed (\ref{delprimeabovegamma}) holds. 
There remains to establish (\ref{activesumbelowgamma}); this is the content of the next lemma, which is established following the steps carried out in Section $4.2.5$ of \cite{de_ambroggio_roberts:near_critical_RRG}. 

\begin{lem}\label{stepquttro}
	On the event $\mathcal{G}_{i-1}$, for all large enough $n$ we have
\begin{equation}\label{toprovenow}
\mathbb{P}\Big(\exists t\leq (t_i-t_{i-1})\wedge T : \sum_{i=1}^t \mathbbm{1}_{\{h^i_k\in \mathcal{A}^i_{k-1}\}} \ge t^\gamma\mid \mathcal{F}_{t_{i-1}}\Big) \ll T^{-1/2}.
\end{equation}
\end{lem}
\begin{proof}
	As we said prior to the statement of the lemma, the proof follows the same steps carried out in Section $4.2.5$ of \cite{de_ambroggio_roberts:near_critical_RRG}. The strategy there was to use a union bound to show that, for small values of $t$ (i.e. during the first phase of the exploration process), it is very unlikely that $h^i_k$ is selected from the set of active stubs. On the other hand, for large values of $t$, say $t\geq b$ with $b=b(A,n)$ to be selected, we show that it is unlikely to observe a large time at which the sum of indicators in (\ref{toprovenow}) is larger than $b^{\gamma}$ (provided $b$ is large enough). 
	
	Keeping the above explanation in mind, we start by observing that, since $|\mathcal{A}^i_k|\leq dk$ (which follows after noticing that, at each step of the exploration process, at most $d$ stubs become active), we have 
	\begin{equation*}
	\mathbb{P}\Big(\exists t\leq (t_i-t_{i-1})\wedge b: h^i_k\in \mathcal{A}^i_{k-1}\mid \mathcal{F}_{t_{i-1}}\Big)\leq \sum_{k=1}^{b}\frac{d(k-1)}{dn-2(t_{i-1}+k-1)-1}\leq \frac{2b^2}{n}
	\end{equation*}
	for all large enough $n$, where for the last inequality we have used that, on the event $\mathcal{G}_{i-1}$, we have $t_{i-1}+k\leq BLT^{1/2}+T\leq 2BLT^{1/2}$ and $BLT^{1/2}=\Theta(A^{1/2}n^{2/3})\ll n$ (as $A\ll n^{2/3}$). 
	
	Next we consider the case of \textit{large} $t$, i.e. we take $t> b$. In this case, we set 
	\[\mathcal{B}\coloneqq \{|\mathcal{A}^i_{k}|\leq \ell \text{ }\forall t\leq (t_i-t_{i-1})\wedge T\}, \text{ }\ell \in \mathbb{N},\]
	and observe that
	\begin{multline}\label{ignoresecond}
	\mathbb{P}\Big(\sum_{i=1}^{(t_i-t_{i-1})\wedge T} \mathbbm{1}_{\{h^i_k\in \mathcal{A}^i_{k-1}\}} \ge b^\gamma\mid \mathcal{F}_{t_{i-1}}\Big)\\
	\leq \mathbb{P}\Big(\big\{\sum_{i=1}^{(t_i-t_{i-1})\wedge T} \mathbbm{1}_{\{h^i_k\in \mathcal{A}^i_{k-1}\}} \ge b^\gamma\big\}\cap \mathcal{B}\mid \mathcal{F}_{i-1}\Big)+\mathbb{P}(\mathcal{B}^c\mid \mathcal{F}_{t_{i-1}}).
	\end{multline}
	Let us ignore for the the moment the second probability on the right-hand side of (\ref{ignoresecond}). Arguing exactly as in Section 4.2.5 of \cite{de_ambroggio_roberts:near_critical_RRG}, it is not difficult to see that
	\[\mathbb{P}\Big(\big\{\sum_{i=1}^{(t_i-t_{i-1})\wedge T} \mathbbm{1}_{\{h^i_k\in \mathcal{A}^i_{k-1}\}} \ge b^\gamma\big\}\cap \mathcal{B}\mid \mathcal{F}_{t_{i-1}}\Big)\leq \exp\big(-b^{\gamma}+c T \ell n^{-1}\big)\]
	for some constant $c>0$. But then combining the last estimates we can bound from above the probability which appears in the statement of the lemma by
	\begin{align}\label{thexpterm}
	\nonumber\mathbb{P}\big(\exists t\leq (t_i-t_{i-1})\wedge b: h^i_k\in \mathcal{A}^i_{k-1}\mid \mathcal{F}_{t_{i-1}}\big)&+\mathbb{P}\big(\exists b<t\leq (t_i-t_{i-1})\wedge T: \sum_{i=1}^t \mathbbm{1}_{\{h^i_k\in \mathcal{A}^i_{k-1}\}} \ge t^\gamma\mid \mathcal{F}_{t_{i-1}}\big)\\
	\nonumber&\leq \frac{2b^2}{n}+	\mathbb{P}\big(\sum_{i=1}^{(t_i-t_{i-1})\wedge T} \mathbbm{1}_{\{h^i_k\in \mathcal{A}^i_{k-1}\}} \ge b^\gamma\mid \mathcal{F}_{t_{i-1}}\big)\\
	&\leq \frac{2b^2}{n}+\exp\big(-b^{\gamma}+c T \ell n^{-1}\big)+\mathbb{P}(\mathcal{B}^c\mid \mathcal{F}_{t_{i-1}}).
	\end{align}
	There remains to bound the (conditional) probability of $\mathcal{B}^c$, given $\mathcal{F}_{t_{i-1}}$. This is easily done (here we follows closely Lemma $4.6$ in \cite{de_ambroggio_roberts:near_critical_RRG}). Note that
	\begin{equation}\label{befdoob}
	\mathbb{P}(\mathcal{B}^c\mid \mathcal{F}_{t_{i-1}})\leq \mathbb{P}\big(\exists t\leq (t_i-t_{i-1})\wedge T: d_i+\sum_{k=1}^{t}D^i_k> \ell\big),
	\end{equation}
	where $D^i_k\coloneqq \mathbbm{1}_{R^i_k}(d-1)-1$ (the upper bound is justified by the fact that, for $t\leq t_i-t_{i-1}$, we can write $|\mathcal{A}^i_t|=d_i+\sum_{k=1}^{t}\eta^i_k$ and each $\eta^i_k$ is at most $D^i_k$; moreover, the events $R^i_k$ do not depend on $\mathcal{F}_{t_{i-1}}$, as they only pertain to the percolation process on $\mathbb{G}'(n,d)$). Recalling that $R^i_k$ occurs with probability $p=(1+\lambda n^{-1/3})/(d-1)$ (being the event that the edge revealed at step $t_{i-1}+k$ is retained), we see that $\sum_{k=1}^{t}D^i_k$ is a sub-martingale for $\lambda>0$ (an assumption which can be made without loss of generality); an application of lemma \ref{doobineq} (i.e. Doob's inequality) as in Lemma $4.6$ of \cite{de_ambroggio_roberts:near_critical_RRG} yields that the probability on the right-hand side of (\ref{befdoob}) is at most
	\begin{equation}\label{noideahowtocall}
	C\exp\Big(-\frac{\ell^2}{p(d-1)T}+\frac{\ell}{2}\big(1-1/[p(d-1)]\big)\Big)\leq C\exp\Big(-\frac{\ell^2}{p(d-1)T}+C'\frac{\ell|\lambda|}{n^{1/3}}\Big)
	\end{equation}
	for positive constants $C,C'>0$. Setting for instance $b\coloneqq \lceil A^{1/5}n^{1/3}\rceil$ and $\ell\coloneqq n^{1/3+\gamma/3}A$ we see that $T\ell n^{-1}\ll b^{\gamma}$ as well as $\ell^2/T\gg \ell/n^{1/3}$. Thus, the exponential term in (\ref{thexpterm}) is at most $e^{-cb^{\gamma}}$, whereas the expression on the right-hand side of (\ref{noideahowtocall}) is at most $e^{-c\ell^2/T}$. Both these terms are, for all large enough $n$, at most $b^2/n=\Theta(A^{2/5}/n^{1/3})$, which in turn is $\ll A^{1/2}/n^{1/3}=\Theta(T^{-1/2})$, concluding the proof of the lemma.
\end{proof}

\section{Boosting the martingale method of Nachmias and Peres}\label{heurmg}
The proof of Theorem \ref{mainthm} rely on random walk estimates which are easily applicable provided that the (conditional) distribution of the number of unseen vertices/stubs becoming active at any step of the underlying exploration process is `simple enough'. When the above-mentioned conditional law is more complicated, an analysis like the one we provided in Section \ref{secmainthm} becomes a bit more involved. To this end, the goal of this section is to show how to optimize the (very robust) martingale argument of Nachmias and Peres \cite{nachmias_peres:CRG_mgs,nachmias:critical_perco_rand_regular} in order to obtain \textit{stretched}-exponential upper bounds for the probability of observing an unusually small maximal cluster in critical random graphs. With our adjustments, we can show that the above probability is at most $\exp(-A^{3/5})$ in other two (critical) models (which are introduced below), thus considerably improving upon the existing polynomial bounds available in the literature for such random graphs; this is the content of Theorem \ref{mainthm2} below. We remark that the same methodology used to establish the bounds of Theorem \ref{mainthm2} applies to the models treated in Theorem \ref{mainthm}; however, since we already derived \textit{better} bounds for those random graphs in Section \ref{secmainthm}, we won't treat them again here.

For the random graphs treated in Theorem \ref{mainthm2} we are content with the stretched bounds given here, mostly because we believe that the optimized version of the martingale argument (described below) is interesting on its own and provides a good alternative to the more precise methodology of Section \ref{secmainthm}. However, we believe that by adapting the same argument used to analyze the near-critical $\mathbb{G}(n,p)$ and $\mathbb{G}(n,d,p)$ models, it should be possible to show that the probability of observing an unusually small maximal components in these other (critical) models is also at most $\exp(-\Omega(A^{3/2}))$. However, such an adaptation requires some work, since the (conditional) distributions of the random variables $\eta_i$ are now more complicated: they are still binomials, but now also the success probabilities are random (and consequently the random walk estimates used to treat the models of Theorem \ref{mainthm} requires some work to be implemented).
Before stating Theorem \ref{mainthm2}, we first need to introduce the critical random graphs considered there.

Given $n,k\in \mathbb{N}$, let $V\coloneqq \{v_1,\dots,v_n\}$ and $W\coloneqq \{w_1,\dots,w_k\}$. The random \textit{bipartite} graph $\mathbb{B}(n,k,p)$ is obtained by performing $p$-bond percolation on the complete bipartite graph with bipartition $(V,W)$ where each node $v_i$ is linked to each $w_j$ (for $i\in [n],j\in [k]$). 
The random intersection graph $\mathbb{G}(n,k,p)$ model is constructed from $\mathbb{B}(n,k,p)$ in the following way. The vertex set is $V$ and each pair of distinct vertices $\{v_i,v_j\}$ is present as an edge if, and only if, $v_i$ and $v_j$ share \textit{at least one} common neighbour in $\mathbb{B}(n,k,p)$. We call $V$ the set of \textit{vertices}, whereas we call $W$ the set of \textit{attributes} (or features or auxiliary vertices). This random graph model was analysed by Behrisch \cite{ber} for the case $k=k(n)=n^{\alpha}$ and $p^2k=c/n$, with $\alpha,c\in (0,1)\cup (1,\infty)$. As shown by Stark \cite{stark}, the vertex degree distribution (i.e. the distribution of the degree of a vertex selected uniformly at random) is highly dependent on the value of $\alpha$. However, as shown by Deijfen and Kets \cite{deijfen_kets_2009}, the clustering of the graph is controllable only when $\alpha=1$. Indeed, when $\alpha\in (0,1)$, the clustering coefficient converges to $1$ as $n\rightarrow \infty$, while it converges to $0$ (as $n\rightarrow \infty$) when $\alpha>1$. On the other hand, in the regime $\alpha=1$, the clustering coefficient converges to $\beta$ (as $n\rightarrow \infty$) and so, by specifying the value of $\beta$, one can control the clustering of the graph. This latter regime where $k=\Theta(n)$ was subsequently investigated by Lageras and Lindholm \cite{laglind}. Specifically, in \cite{laglind} the authors considered the $\mathbb{G}_{n,k,p}$ random graph with $k=\lfloor \beta n\rfloor$ and $p=\gamma/n$, where $\gamma,\beta>0$ are model parameters and proved that the $\mathbb{G}_{n,k,p}$ model undergoes a phase transition as $\beta \gamma^2$ passes $1$. Indeed, setting $\mu= \beta \gamma^2$, they proved that if $\mu<1$ (sub-critical case) then with probability tending to one there is no component in $\mathbb{G}_{n,k,p}$ with more than $O(\log(n))$ vertices, while if $\mu >1$ (super-critical case) then, with probability tending to one, there exists a unique giant component of size $n\delta$ (with $\delta\in (0,1)$ constant) and the size of the second largest component is at most of order $\log(n)$. It was then shown in \cite{de_ambroggio:component_sizes_crit_RGs,de_ambroggio:upper_component_sizes_crit_RGs} that, in the (critical) case $\mu=1$, there is a constant $c=c(\gamma,\beta)>0$ depending on $\gamma,\beta$ such that, for all large enough $A$ and $n$,
\begin{equation}\label{boundgnmp1}
\mathbb{P}\left(|\mathcal C_{\max}(\mathbb{G}_{n,k,p})|>An^{2/3}\right)\leq  cA^{-3/2} \text{ and }\mathbb{P}\left(|\mathcal C_{\max}(\mathbb{G}_{n,k,p})|<n^{2/3}/A\right)\leq  cA^{-1/2}.
\end{equation} 
The second model we consider in this section is the so-called (critical) \textit{quantum random graph}. Here we only provide an informal definition of such a model and refer the interested reader to \cite{dembo_et_al:component_sizes_quantum_RG} and Section \ref{quantum} below for the rigorous definition (see also \cite{de_ambroggio:upper_component_sizes_crit_RGs}). 

We have $n$ circles of length $\beta>0$, punctured with \textit{independent} Poisson point processes of intensity $\lambda>0$. Because of these (Poisson) processes of holes, circles can be written as \textit{finite} \textit{disjoint} unions of connected intervals. The edges of the graph link intervals on distinct circles and they are generated as follows. Given any two (distinct) circles $u\neq v$, we run another \textit{independent} Poisson point process of intensity $1/n$ on a (third) circle of length $\beta>0$. If such a point process jumps at a time which is contained in both an interval of $u$ and an interval of $v$ (recall the decomposition mentioned earlier), then these two intervals are considered directly connected. The resulting random graph, denoted by $\mathbb{G}(n,\beta,\lambda)$, is called \textit{quantum random graph}. As explained in \cite{dembo_et_al:component_sizes_quantum_RG}, the notion of component of a point $x=(u,t)$ (where $u$ is an element of $[n]$ while $t$ is a point of the circle) in this random graph ensemble is well-defined and the \textit{size} of a component is the number of \textit{intervals} it contains.
The critical parameter (curve) for the quantum random graph in the $(\beta,\lambda)$-parameter space was identified by Ioffe and Levit \cite{ioffe_levit} by comparison with a critical branching process whose offspring distribution is the so-called \textit{cut-gamma} law $\Gamma_{\theta}(2,1)$, which is the distribution of $J\coloneqq (J_1+ J_2 ) \wedge \theta$ for independent random variables $J_1,J_2$ with the $\text{Exp}(1)$ law. Setting $F(\lambda \beta)\coloneqq 2(1-e^{-\lambda \beta})-\lambda \beta e^{-\lambda\beta}$, criticality occurs when
\[F(\beta,\lambda)\coloneqq \lambda^{-1}F(\lambda \beta)=1, \]
with $\lambda^{-1}F(\lambda \beta)$ corresponding to the expected length of an interval $I$ in the quantum random graph. (It was shown in \cite{ioffe_levit} that, if $F(\beta,\lambda)>1$, then a giant component of order $\Theta(n)$ emerges, whereas when $F(\beta,\lambda)<1$ all components are typically of order $O(\log(n))$; see \cite{dembo_et_al:component_sizes_quantum_RG} and \cite{ioffe_levit}).
Dembo, Levit and Vadlamani \cite{dembo_et_al:component_sizes_quantum_RG} analysed the (near-)critical behaviour of this model and showed that, when $F(\beta,\lambda)=1$, then $|\mathcal{C}_{\max}(\mathbb{G}(n,\beta,\lambda))|$ is of order $n^{2/3}$. Amongst many other things, they showed that there is a constant $c=c(\beta)>0$ depending on $\beta$ such that, for all sufficiently large $A$ and $n$,
\begin{equation}\label{dembo2}
\mathbb{P}(|\mathcal{C}_{\max}(\mathbb{G}(n,\beta,\lambda))|>An^{2/3})\leq cA^{-3/2} \text{ and }\mathbb{P}(|\mathcal{C}_{\max}(\mathbb{G}(n,\beta,\lambda))|<n^{2/3}/A)\leq cA^{-3/5}.
\end{equation}
By means of the optimised martingale version we can give much better tail bounds (compared to (\ref{boundgnmp1}) and (\ref{dembo2})) for the probability that $|\mathcal{C}_{\max}(\mathbb{G})|<n^{2/3}/A$ when either $\mathbb{G}=\mathbb{G}(n,k,p)$ or $\mathbb{G}=\mathbb{G}(n,\beta,\lambda)$.

\begin{thm}\label{mainthm2}
	Let $p=\gamma/n$ and $k=\lfloor \beta n\rfloor$, with $\gamma,\beta>0$ such that $\beta \gamma^2=1$. There exist constants $A_0=A_0(\gamma,\beta)>0$ and $c=c(\beta,\lambda)>0$ such that, for any $A_0\leq A=O(n^{5/12})$ and for all large enough $n$, we have
	\begin{equation}\label{expboundrig}
	\mathbb{P}(|\mathcal{C}_{\max}(\mathbb{G}(n,k,p))|<n^{2/3}/A)\leq \exp\big(-cA^{3/5}\big).
	\end{equation} 
	Let $F(\beta,\lambda)=1$. There exist constants $A_0=A_0(\lambda,\beta)>0$ and $c=c(\lambda,\beta)>0$ such that, for any $A_0\leq A=O(n^{5/12})$ and for all large enough $n$, we have
	\begin{equation}\label{expboundquantum}
	\mathbb{P}(|\mathcal{C}_{\max}(\mathbb{G}(n,\beta,\lambda))|<n^{2/3}/A)\leq \exp\big(-cA^{3/5}\big).
	\end{equation}
\end{thm}

\begin{obs}
	The requirement $A=O(n^{5/12})$ can be weakened to $A=o(n^{2/3})$, at the expenses of adding a term like $\exp\big(-cn^{1/3}/A^{1/5}\big)$ to both bounds appearing in (\ref{expboundrig}) and (\ref{expboundquantum}), which considerably improve upon (\ref{boundgnmp1}) and (\ref{dembo2}).
\end{obs}

\begin{obs}
	As we pointed out at the start of this section, the same method used to establish Theorem \ref{mainthm2} can also be used to prove similar bounds for the critical random graphs $\mathbb{G}(n,p)$ and $\mathbb{G}(n,d,p)$; however, we have already established stronger estimates in Section \ref{secmainthm} and hence there is no reason to use the optimized martingale argument to obtain bounds of the type $\exp(-cA^{3/5})$ for these two models.
\end{obs}
We continue by providing a detailed, but still informal explanation on how to optimize the martingale argument introduced in \cite{nachmias_peres:CRG_mgs,nachmias:critical_perco_rand_regular} to obtain the stretched exponential upper bounds displayed in Theorem \ref{mainthm2}; we do so in terms of the critical $\mathbb{G}(n,p)$ model, but the same strategy works for the critical random graphs $\mathbb{G}(n,k,p),\mathbb{G}(n,\beta,\lambda)$ and $\mathbb{G}(n,d,p)$.

\subsection{Heuristic derivation of the stretched exponential bounds}\label{mgder}
Here we show how the original argument of \cite{nachmias_peres:CRG_mgs} can be boosted to show that $|\mathcal{C}_{\max}(\mathbb{G}(n,1/n))|\geq n^{2/3}/A$ with probability at least $1-\exp(-\Omega(A^{3/5}))$; we provide the explanation in terms of the $\mathbb{G}(n,1/n)$ model for simplicity of exposition, but the same argument applies to the other models considered in this work.

Recall the exploration process for the $\mathbb{G}(n,p)$ model, with $Y_t$ denoting the number of active vertices at time $t\in \mathbb{N}$ in the exploration process of Section \ref{secmainthm} for the $\mathbb{G}(n,p)$ model. As usual, $t_0\coloneqq 0$ and denote by $t_i$ the first time $t>t_{i-1}$ at which $Y_t=0$, for $i\in \mathbb{N}$ (prior to the end of the procedure). Then 
\begin{equation}\label{FF}
\mathbb{P}(\CC_{\max}(\mathbb{G}(n,p)))\leq \mathbb{P}(t_i-t_{i-1}<n^{2/3}/A \text{ }\forall i).
\end{equation}
Following Nachmias and Peres \cite{nachmias_peres:CRG_mgs}, the basic idea is then to show that it is (very) likely for the process $Y_t$ to go above some large enough level $\ell=\ell(n,A)$ before some time $\mathbb{N}\ni T=T(\ell)\ll n$ and then to remain positive for at least $n^{2/3}/A$ steps. 

Hence we bound from above the probability on the right-hand side of (\ref{FF}) by
\begin{equation}\label{display}
\mathbb{P}(t_i-t_{i-1}<n^{2/3}/A \text{ }\forall i, \exists t<T:Y_t\geq \ell)+\mathbb{P}(Y_t<\ell \text{ }\forall t<T).
\end{equation}
At this stage we introduce the necessary changes to the original argument of \cite{nachmias_peres:CRG_mgs} in order to show that both terms in (\ref{display}) are at most $\lesssim \exp(-cA^{3/5})$. 

Before starting the actual (heuristic) derivation, we make a general observation concerning the choice of $\ell$ and $T$. The first probability appearing in (\ref{display}) is at most the probability that the process $Y_t$, after having reached level $\ell$ (and considered from the the first time it does so), reaches zero in less than $n^{2/3}/A$ steps (otherwise we would find a positive excursion of $Y_t$ lasting for \textit{more} than $n^{2/3}/A$ steps). If we want this probability to be small, the level $\ell$ has to be chosen large; intuitively, thinking of $Y_t$ as a $\mathbb{Z}$-valued, mean-zero random walk (this approximation is justified later), we need $\ell$ to be much bigger than the square root of the time, whence we require that $\ell \gg n^{1/3}/A^{1/2}$. This forces $T$ to be large, otherwise in the second probability in (\ref{display}) we would be asking the number of active vertices $Y_t$ to remain below a large value for a short time interval, a likely event. In particular, we need $T$ to be much bigger than $\ell^2$, i.e. we require $T\gg \ell^2$. But, as we will see below, the increments of the process $Y_t$ have a (negative) drift of order $ - T/n$, whence a very large value of $T$ would cause $Y_t$ to reach zero too quickly, thus making the first probability in (\ref{display}) large. All in all, it becomes evident that a trade-off between $\ell$ and $T$ is needed in order to make \textit{both} terms in (\ref{display}) small at the same time. 

Let us now illustrate the simple argument which we use to obtain stretched exponential bounds for both expressions in (\ref{display}). Let $1\ll h,m\in \mathbb{N}$ (possibly dependent on $A,n$) and $\delta>0$ (independent of $A,n$). To bound the second term in (\ref{display}), i.e. the probability that the number of active vertices stays below $\ell$ for $T$ steps, we set $\ell\coloneqq h\delta$ and split the (discrete) interval $[T]\coloneqq \{1,\dots,T\}$ into $m$ (smaller) disjoint, connected sub-intervals of length $h^2$ and we show that, in each one of these $m$ smaller intervals, the process has a \textit{constant}, \textit{positive} probability (which depends on $\delta$) to go above the level $h\delta$ in $h^2$ steps, provided $\delta<1$ is sufficiently small. This is true because the random process $Y_t$ behaves like a $\mathbb{Z}$-valued random walk with iid increments having mean zero and finite second moment, and such a random walk has a constant positive probability (independent of $h$) to be above $h\delta$ after $h^2$ steps. Denoting by $c=c(\delta)<1$ the positive constant bounding from above the probability that $Y_t$ stays below $h\delta$ for $h^2$ steps and iterating the argument $m$ times, we obtain that $Y_t$ stays below $h\delta$ for $T= mh^2$ consecutive steps with probability at most $c^m=\exp(-m\log(1/c))$. 

To obtain an exponential (upper) bound for the first probability in (\ref{display}) instead, we proceed as follows. Let $\tau$ be the minimum between the first time $t\in \mathbb{N}$ at which the number of active vertices $Y_t$ goes above $h\delta$ and $T=mh^2$. Denote by $\eta_t$ the number of \textit{unseen} vertices which are added to the set of active nodes at step $t$. We start by noticing that the process $Y_{\tau+s}$ (started at time $\tau$) must reach zero in \textit{less} than $n^{2/3}/A$ steps, otherwise we would discover a positive excursion consisting of more than $n^{2/3}/A$ steps. Hence there is a first time $t<n^{2/3}/A$ at which $Y_{\tau+t}=0$. Observe that at time $\tau+i$ we add $\eta_{\tau+i}$ vertices to the set of active nodes and remove one vertex from such a set (as long as the set of active nodes remains non-empty); thus, since the number of active vertices stays positive from time $\tau$ until time $\tau+t$, we can write $Y_{\tau+s}=Y_{\tau}+\sum_{i=1}^{s}(\eta_{\tau+i}-1)$ for $s\leq t$. But we know that $Y_{\tau}\geq h\delta$ (and $Y_{\tau+t}=0$), whence we arrive at $\sum_{i=1}^{t}(\eta_{\tau+i}-1)\leq -h\delta$. Conditional on the history of the exploration process until step $\tau+i-1$, the random variable $\eta_{\tau+i}$ has the $\text{Bin}(n-\tau-i+1-Y_{\tau+i-1},p)$ distribution (with $n-\tau-i+1-Y_{\tau+i-1}$ being the number of \textit{unseen} vertices available at time $\tau+i$). Since the number of active vertices is never `too large', to continue our derivation we approximate
\[\text{Bin}(n-\tau-i-1-Y_{\tau+i-1},p)\approx \text{Bin}(n-\tau-i,p).\]
Then, roughly,
\[\sum_{i=1}^{t}(\eta_{\tau+i}-1)\leq -h\delta \text{ occurs if, and only if, }\sum_{i=1}^{t}(\text{Bin}_i(n-\tau-i,p)-1)\leq -h\delta \text{ happens}.\]
Since $\tau\leq mh^2$ and $t\leq n^{2/3}/A$, the latter event `implies' that
\begin{equation}\label{EEE}
	\sum_{i=1}^{t}(\text{Bin}_i(n-mh^2-n^{2/3}/A,p)-1)\leq -h\delta \text{ occurs}.
\end{equation}
Moreover, using the fact that binomial random variables concentrate around their mean (whence in particular
\[\sum_{i=1}^t\text{Bin}_i(T+n^{2/3}/A,p)\approx (tmh^2)/n+t/(n^{1/3}A),\]
where we used that $p=1/n$ and $t<n^{2/3}/A$), we conclude that the probability of (\ref{EEE}) happening at some $t<n^{2/3}/A$ is (very) roughly the same as the probability of
\begin{equation}\label{FFF}
	\sum_{i=1}^{t}(\text{Bin}_i(n,p)-1)\leq -h\delta+(tmh^2)/n+ t/(n^{1/3}A)
\end{equation}
occurring at some $t<n^{2/3}/A$. Now notice that, since $t<n^{2/3}/A$, we can bound
\[(tmh^2)/n\leq (mh^2)/(n^{1/3}A) \text{ and } t/(n^{1/3}A)\leq n^{1/3}/A^2\]
and hence the probability in (\ref{FFF}) is at most the probability of seeing a time $t<n^{2/3}/A$ at which
\[\sum_{i=1}^{t}(\text{Bin}_i(n,p)-1)\leq -h\delta+(mh^2)/(n^{1/3}A)+ n^{1/3}/A^2,\]
which is small only if both $(mh^2)/(n^{1/3}A)$ and $n^{1/3}/A^2$ are sufficiently smaller than $h\delta$, say smaller than $(h\delta)/3$. In this case then, we would we bounding from above the probability that a mean-zero random walk (with supposedly iid, well-behaved) $\text{Bin}_i(n,p)-1$ increments goes below $\asymp -h$ in less than $n^{2/3}/A$ steps. An application of Doob's submartingale inequality yields that this occurs with probability at most $\lesssim \exp(-(h^2A)/n^{2/3})$. Combining this estimate with our earlier bound of order $e^{-m}$ we see that the sum of the two terms in (\ref{display}) is at most of order
\begin{equation}\label{withm}
\exp(-(h^2A)/n^{2/3})+ \exp(-m).
\end{equation}
The first exponential is small only if $h\gg n^{1/3}/A^{1/2}$ (in which case the earlier constraint $n^{1/3}/(2A^2)\leq (h\delta)/3$ is clearly satisfied provided $A$ is large enough). Recalling the other constraint we imposed earlier, namely that $(mh^2)/(An^{1/3})\leq h/3$, or equivalently $h\leq (An^{1/3})/(3m)$, we immediately see the necessity of having 
\[n^{1/3}/A^{1/2}\ll h \leq (An^{1/3})/(3m),\]
which leads to $m\ll A^{3/2}$, thus excluding right away the optimal bound of order $\exp(-A^{3/2})$. The best choice seems to be something like $m=A,h=n^{1/3}$, in which case the expression in (\ref{withm}) would be $\lesssim \exp(-A)$; however, as we will see later on, a further technical requirement forbids this choice. Indeed, we will also need $(mh^3)/n$ to be at most of (small) constant order, and taking $m,h$ as above would violate the latter requirement. Hence, including this extra constraint, the optimal choice seems to be $h=n^{1/3}/A^{1/5},m=A^{3/5}$, thus leading to the bound $\lesssim\exp(-A^{3/5})$ (under the assumption $A=O(n^{5/12})$).

\subsection{The $\mathbb{G}(n,k,p)$ model}

Recall that we are interested in the case where $k=\lfloor \beta n\rfloor$ and $p=\gamma/n$, for some $\beta\gamma>0$ which are model parameters. Moreover, throughout we assume that $\gamma^2\beta=1$.

As we did in the previous sections, we start by describing an exploration process to reveal the connected components of the random intersection graph $\mathbb{G}_{n,k,p}$.


Fix an ordering of the $n$ vertices in $V$ with $v$ first. Let us denote by $\mathcal{A}_{t}$, $\mathcal{U}_t$ and $\mathcal{E}_t$ the (random) sets of active, unseen and explored \textit{vertices} at the end of step $t\in \mathbb{N}_0$, respectively. Then, for every $t\in \mathbb{N}_0$, we have $V=\mathcal{A}_t\cup \mathcal{U}_t\cup \mathcal{E}_t$ (a disjoint union), so that in particular $\mathcal{U}_t=V\setminus(\mathcal{A}_t\cup \mathcal{E}_t)$ at each step $t$. Moreover, we denote by $\mathcal{D}_t$ the (random) sets of discovered \textit{attributes} by the end of step $t\in \mathbb{N}_0$, so that $W\setminus \mathcal{D}_t$ is the set of fresh attributes at the end of step $t$.\\

\textbf{Algorithm 3}. At time $t=0$, vertex $v$ is declared active whereas all other vertices are declared unseen; moreover, all the attributes are fresh. Therefore $\mathcal{A}_0= \{v\}$, $\mathcal{E}_0=\emptyset$ and $\mathcal{D}_0= \emptyset$. For every $t\in \mathbb{N}$, the algorithm proceeds as follows. 
\begin{itemize}
	\item [(a)] If $|\mathcal{A}_{t-1}|\ge 1$, we let $u_t$ be the first active node (here and in what follows, the term \textit{first} refers to the ordering that we have fixed at the very beginning of the procedure).
	\item [(b)] If $|\mathcal{A}_{t-1}|=0$ and $|\mathcal{U}_{t-1}|\ge 1$, we let $u_t$ be the first unseen vertex.
	\item [(c)] if $|\mathcal{A}_{t-1}|=0=|\mathcal{U}_{t-1}|$ (so that $\mathcal{E}_{t-1}=V$), we halt the procedure.
\end{itemize}
Let us denote by $\mathcal{N}_t$ the (random) set of (fresh) attributes that are linked to $u_t$; formally,
\begin{equation}\label{eqcurlyn}
\mathcal{N}_t\coloneqq \{w\in W\setminus \mathcal{D}_{t-1}:w\sim u_t\}.
\end{equation}
Moreover, we denote by $\mathcal{R}_t$ the (random) set of unseen neighbours of $u_t$ that are linked to \textit{at least one} of the attributes in $\mathcal{N}_t$; formally, we set $\mathcal{R}_t\coloneqq \{u\in \mathcal{U}_{t-1}\setminus\{u_t\}:u\sim w \text{ for some }w\in \mathcal{N}_t\}$
(and note that $\mathcal{U}_{t-1}\setminus \{u_t\} = \mathcal{U}_{t-1}$ if $\mathcal{A}_{t-1}\neq \emptyset$). Then we update 
$\mathcal{U}_t\coloneqq \mathcal{U}_{t-1}\setminus (\mathcal{R}_t\cup \{u_t\})$, $\mathcal{A}_{t}\coloneqq (\mathcal{A}_{t-1}\setminus \{u_t\})\cup \mathcal{R}_t,\mathcal{E}_t\coloneqq \mathcal{E}_{t-1}\cup \{u_t\}$ and $\mathcal{D}_{t}\coloneqq \mathcal{D}_{t-1}\cup \mathcal{N}_t$.\\
\begin{obs}
	Since in the procedure \textbf{Algorithm 3} we explore \textit{one vertex} at each step, we have $\mathcal{A}_t\cup \mathcal{U}_t\neq \emptyset$ for every $t\leq n-1$ and $\mathcal{A}_n\cup\mathcal{U}_n= \emptyset$ (as $\mathcal{E}_n=V)$. Thus the algorithm runs for $n$ steps. 
\end{obs}

Denoting by $\eta_t$ the (random) number of unseen vertices that we add to the set of active nodes at time $t$, since at the end of each step $i$ in which $|\mathcal{A}_{i-1}|\geq 1$ we remove the (active) vertex $u_i$ from $\mathcal{A}_{i-1}$ (after having revealed its unseen neighbours), we have the recursion
\begin{itemize}
	\item $|\mathcal{A}_t|=|\mathcal{A}_{t-1}|+\eta_t-1$ if $|\mathcal{A}_{t-1}|>0$;
	\item $|\mathcal{A}_t|=\eta_t$ if $|\mathcal{A}_{t-1}|=0$.
\end{itemize}
Let us denote by $Y_t$ and $U_t$ the number of active and unseen vertices at the end of step $t\in [n]\cup\{0\}$, respectively, so that 
\[Y_t=|\mathcal{A}_t| \text{ and } U_t=|\mathcal{U}_t|=n-t-Y_t.\]
Denote by $\mathcal{F}_t$ the $\sigma$-algebra generated by all the information collected by the exploration process until the end of step $t$.
Then, conditional on $\mathcal{F}_{t-1}$ \textit{and} $\mathcal{N}_t$, the random variable $\eta_t$ depends on the past until time $t-1$ only through $Y_{t-1}$ and, in particular, setting $N_t\coloneqq |\mathcal{N}_t|$ we have
\begin{equation}\label{lawetagnmp}
\eta_t=_d\Bin(U_{t-1}-\mathbb{1}_{\{Y_{t-1}=0\}},1-(1-p)^{N_t}) \text{ for }t\in [n].
\end{equation} 
We remark that each random variable $N_t$ satisfies $N_t\leq _{sd}\text{Bin}(k,p)$; this is immediate from the description of \textbf{Algorithm 3} given above.

Define $t_0\coloneqq 0$ and $t_i\coloneqq \min\{t\geq t_{i-1}+1:Y_t=0\}$, for $i\geq 1$ (as long as we do not enter Step $(c)$ in \textbf{Algorithm 3}). Then, denoting by $\mathcal{C}_i$ the $i$-th component revealed during the exploration process, we have $|\CC_i|=t_i-t_{i-1}$ for all $i$ and so, setting $T\coloneqq \lceil n^{2/3}/A\rceil$, we can write
\begin{equation}\label{excgnpint}
\mathbb{P}(|\mathcal{C}_{\max}(\mathbb{G}_{n,k,p})|<n^{2/3}/A)
=\mathbb{P}(t_i-t_{i-1}<n^{2/3}/A \text{ }\forall i)\le \mathbb{P}(t_i-t_{i-1}<T \text{ }\forall i).
\end{equation}
Following the strategy described in Section \ref{mgder}, we let $1\ll h,m\in \mathbb{N}$ (possibly dependent on $A,n$) and denote by $\delta\in (0,1)$ a constant that we specify later. Then we bound
\begin{align}\label{befproposint}
\nonumber\mathbb{P}(t_i-t_{i-1}<T \text{ }\forall i)&\leq \mathbb{P}(t_i-t_{i-1}<T \text{ }\forall i, \exists t\in [mh^2]:Y_t\geq h\delta)\\
&+\mathbb{P}(Y_t<h\delta \text{ }\forall t\in [mh^2])
\end{align}
and control these two terms separately. We start by bounding the first term on the right-hand side of (\ref{befproposint}) with of Proposition \ref{Prop1Gnpint} below and subsequently, with Proposition \ref{Prop2Gnpint}, we control the second probability on the right-hand side of (\ref{befproposint}).

\begin{prop}\label{Prop1Gnpint}
	Suppose that $(mh^2)/n=O(1)$ and $T^2/(hn)\leq \delta/2$. Then, for all large enough $n$, we have
	\begin{equation*}
	\mathbb{P}(t_i-t_{i-1}<T \text{ }\forall i, \exists t\in [mh^2]:Y_t\geq h\delta)	\leq  e^{-ch}\vee e^{-(ch^2)/T},
	\end{equation*}
	for some constant $c=c(\gamma,\beta,\delta)$ which depends on $\gamma,\beta$ and $\delta$.
\end{prop}
\begin{proof}
Let us denote by $\tau'$ the first time $t\geq 1$ at which $Y_t\geq h\delta$ and set $\tau\coloneqq \tau'\wedge mh^2$. Note that, on the event $\{\exists t\in [mh^2]:Y_t\geq h\delta\}$, we have $Y_{\tau}\geq h\delta$; moreover, if $t_i-t_{i-1}<T$ for every $i$, then necessarily there exists a (first) time $t<T$ such that $Y_{\tau+s}>0$ for all $s\leq t-1$ and $Y_{\tau+t}=0$ (otherwise we would find an excursion lasting for more than $T$ steps, which is impossible on the event $\{t_i-t_{i-1}<T \text{ }\forall i\}$). Thus we can express $Y_{\tau+t}$ as $Y_{\tau+t}=Y_{\tau}+\sum_{i=1}^{t}(\eta_{\tau+i}-1)$. Hence, since on the event $\{Y_{\tau}\geq h\delta\}$ we have $-Y_{\tau}\leq -h\delta$, we can bound 
\begin{align}\label{longprobgen}
\nonumber\mathbb{P}(t_i-t_{i-1}<T \text{ }\forall i, &\exists t\in [mh^2]:Y_t\geq h\delta)\\
\nonumber&\leq \mathbb{P}(\exists t<T:Y_{\tau+s}>0 \text{ }\forall s\leq t-1,Y_{\tau+t}=0,Y_{\tau}\geq h\delta)\\
&\leq \mathbb{P}(\exists t<T:Y_{\tau+s}>0 \text{ }\forall s\leq t-1,\sum_{i=1}^{t}(\eta_{\tau+i}-1)\leq -h\delta).
\end{align}
In order to control (from above) the number of active vertices which we observe in the time window $[\tau,\tau+T]\cap \mathbb{N}$ (the need for a uniform upper bound on $Y_t$ will become clear in a moment), we define, for $t<T$, the events
\begin{equation*}
\mathcal{H}_1(t)\coloneqq \{Y_{\tau+s}>0 \text{ }\forall s\leq t-1,\sum_{i=1}^{t}(\eta_{\tau+i}-1)\leq -h\delta\} \text{ and  }\mathcal{H}_2(t)\coloneqq \{Y_{\tau+z}<h \text{ }\forall z\leq t-1\}.
\end{equation*}
Then the probability on the right-hand side of (\ref{longprobgen}) is at most
\begin{equation}\label{treegnpgen}  \mathbb{P}(\bigcup_{t<T}\mathcal{H}_1(t)\cap \mathcal{H}_2(t))+\mathbb{P}\Big(\bigcup_{t<T}\{Y_{\tau+s}>0 \text{ }\forall s\leq t-1\}\cap \mathcal{H}^c_2(t)\Big).
\end{equation}
With the next lemma we show that the second probability on the right-hand side of (\ref{treegnpgen}) is small.
\begin{lem}\label{lem1betarig}
	Suppose that $(mh^2)/n=O(1)$. Then, for all large enough $n$, we have
	\begin{equation*}
	\mathbb{P}\Big(\bigcup_{t<T}\{Y_{\tau+s}>0 \text{ }\forall s\leq t-1\}\cap \mathcal{H}^c_2(t)\Big)\leq e^{-ch}\vee e^{-(ch^2)/T},
	\end{equation*} 
	for some constant $c=c(\gamma,\beta)>0$ which depends on $\gamma,\beta$.
\end{lem}
\begin{proof}
	We can bound from above the probability which appears in the statement of the lemma by 
	\begin{equation}\label{1gnpbetarig}
	\mathbb{P}\Big(\exists t<T:\sum_{i=1}^{t}(X_{\tau+i}-E_{\tau+i})\geq h-1-\sum_{i=1}^{t}E_{\tau+i}\Big),
	\end{equation}
	where we set $X_{\tau+i}\coloneqq \eta_{\tau+i}-1$ and $E_{\tau+i}\coloneqq \mathbb{E}[X_{\tau+i}|\mathcal{F}_{\tau+i-1}]$.
	Now recall from (\ref{lawetagnmp}) that, conditional on $\mathcal{F}_{\tau+i-1}$ \textit{and} $\tau$, the random variable $\eta_{\tau+i}$ is distributed as
	\begin{equation}\label{lawetabetarig}
	\text{Bin}(n-\tau-(i-1)-Y_{\tau+i-1}-\mathbb{1}_{\{Y_{\tau+i-1}=0\}},1-(1-p)^{N_{\tau+i}}).
	\end{equation}
	Therefore, using the classical bound $(1-x)^{\ell}\geq 1-\ell p$ (which is valid for all $x >-1,\ell \in \mathbb{N}$), we see that 
	\begin{equation}\label{boundsuccprob}
	1-(1-p)^{N_{\tau+i}}\leq N_{\tau+i}p
	\end{equation}
	and so, as $N_t\leq_{sd}\text{Bin}(k,p)$ for each $t$ and $nkp^2\leq \gamma^2\beta =1$, we obtain that
	\begin{equation}\label{eithetarig}
	E_{\tau+i}\leq np\mathbb{E}[N_{\tau+i}]-1\leq nkp^2-1\leq 0.
	\end{equation}
	Thus the probability on the right-hand side of (\ref{1gnpbetarig}) is, for any given $r\in (0,1)$, at most
	\begin{equation}\label{befsubrig}
	\mathbb{P}(\max_{t<T}e^{r\sum_{i=1}^{t}(X_{\tau+i}-E_{\tau+i})}\geq e^{r(h-1)}).
	\end{equation}
	In a moment we will need a lower bound on $|\mathcal{D}_{\tau+i-1}|$, the number of discovered attributes by the end of time $\tau+i-1$. To this end, we define 
	\begin{equation}\label{primeprime}
	X'_{\tau+i}\coloneqq \mathbb{1}_{\{|\mathcal{D}_{\tau+i-1}|\leq \omega\}}X_{\tau+i} \text{ and }E'_{\tau+i}\coloneqq \mathbb{1}_{\{|\mathcal{D}_{\tau+i-1}|\leq \omega\}}E_{\tau+i}
	\end{equation}
	where $\omega\geq 1$ has to be specified. Now observe that, as $\mathcal{D}_t\subset \mathcal{D}_{t+1}$ at all times $t$ (since more and more attributes become discovered as the exploration proceeds), we obtain
	\begin{equation*}
	\mathbb{P}(\exists t< mh^2+T:|\mathcal{D}_t|>\omega)\leq \mathbb{P}(|\mathcal{D}_{mh^2+T}|>\omega)\leq e^{-r\omega}\mathbb{E}[e^{r\text{Bin}(k,p)}]^{mh^2+T}.
	\end{equation*}
	Taking $\omega\coloneqq 2pk(mh^2+T)$ and $0<r<1$, it is immediate to show that the expression on the right-hand side of the last inequality is at most $\exp\{-rpk(mh^2+T)(1-r)\}$ whence, setting e.g. $r=1/2$, we arrive at 
	\begin{equation}\label{rubb1}
	\mathbb{P}(\exists t< mh^2+T:|\mathcal{D}_t|>2pk(mh^2+T))\leq e^{-\frac{pk(mh^2+T)}{4}}\leq e^{-c'(mh^2+T)}
	\end{equation}
	for some constant $c'=c'(\gamma,\beta)>0$ depending on $\gamma,\beta$. Thus, going back to (\ref{befsubrig}) and setting
	\begin{equation}\label{matcalw}
	\mathcal{W}\coloneqq \{|\mathcal{D}_t|\leq 2pk(mh^2+T) \text{ }\forall t< mh^2+T\},
	\end{equation}
	noticing that (for each $i\in [T]$) we have $X'_{\tau+i}=X_{\tau+i}$ and $E'_{\tau+i}=E_{\tau+i}$ on $\mathcal{W}$, we can bound (thanks to (\ref{rubb1}))
	\begin{align*}
	\mathbb{P}(\max_{t<T}e^{r\sum_{i=1}^{t}(X_{\tau+i}-E_{\tau+i})}\geq e^{r(h-1)})&\leq \mathbb{P}(\{\max_{t<T}e^{r\sum_{i=1}^{t}(X_{\tau+i}-E'_{\tau+i})}\geq e^{r(h-1)}\}\cap \mathcal{W})+\mathbb{P}(\mathcal{W}^c)\\
	&\leq \mathbb{P}(\max_{t<T}e^{r\sum_{i=1}^{t}(X'_{\tau+i}-E'_{\tau+i})}\geq e^{r(h-1)})+e^{-c'(mh^2+T)}.
	\end{align*}
	Since the process defined by $\exp\{r\sum_{i=1}^{t}(X'_{\tau+i}-E'_{\tau+i})\}$ is a positive submartingale (relative to $\mathcal{F}_{\tau+t}$), we can use Doob's inequality to conclude that
	\begin{equation}\label{beftowergnpthetarig}
	\mathbb{P}(\max_{t<T}e^{r\sum_{i=1}^{t}(X'_{\tau+i}-E'_{\tau+i})}\geq e^{r(h-1)})\leq e^{-r(h-1)}\mathbb{E}[e^{r\sum_{i=1}^{T-1}(X'_{\tau+i}-E'_{\tau+i})}].
	\end{equation}
	Using (\ref{lawetabetarig}) together with the chain of inequalities $1-x\leq e^{-x}\leq 1-x+x^2/2$ (which are valid for all $x\geq 0$) we see that
	\[1-(1-p)^{N_{\tau+i}}\geq pN_{\tau+i}-\frac{1}{2}(pN_{\tau+i})^2.\]
	Moreover, since on the event $\{|\mathcal{D}_{\tau+i-1}|\leq 2pk(mh^2+T)\}$ we have $\text{Bin}(k-2pk(mh^2+T),p)\leq _{sd}N_{\tau+i}$, we obtain (on such event)
	\begin{equation*}
	p\mathbb{E}[N_{\tau+i}|\mathcal{F}_{\tau+i-1}]\geq p^2(k-2pk(mh^2+T)
	\end{equation*}
	and
	\begin{equation*}
	p^2\mathbb{E}[N^2_{\tau+i}|\mathcal{F}_{\tau+i-1}] \leq p^2\mathbb{E}[\text{Bin}^2(k,p)]\leq p^2(kp+(kp)^2) =O_{\gamma,\beta}(1/n^2),
	\end{equation*}
	so that
	\begin{equation*}
	p\mathbb{E}[N_{\tau+i}|\mathcal{F}_{\tau+i-1}]-\frac{1}{2}p^2\mathbb{E}[N^2_{\tau+i}|\mathcal{F}_{\tau+i-1}]\geq p^2k-O(p^3k(mh^2+T)).
	\end{equation*}
	Thus, recalling that $\gamma^2\beta=1$ (and since $p^3k\leq (\gamma^2\beta)/n^2$), we can conclude that, on $\{|\mathcal{D}_{\tau+i-1}|\leq 2pk(mh^2+T)\}$,
	\begin{align}\label{clessidragnpthetarig}
	\nonumber\mathbb{E}[\eta_{\tau+i}|\mathcal{F}_{\tau+i-1}]&\geq (n-\tau-(i-1)-Y_{\tau+i-1}-\mathbb{1}_{\{Y_{\tau+i-1}=0\}})(p^2k-O(p^3k(mh^2+T)))\\ 
	&\geq 1-\frac{\tau+i-1+Y_{\tau+i-1}+\mathbb{1}_{\{Y_{\tau+i-1}=0\}}}{n}-O_{\gamma,\beta}((mh^2+T)/n^2).
	\end{align}
	Whence, recalling the definition of $E'_{\tau+i}$,  we obtain
	\[E'_{\tau+i}\geq -\mathbb{1}_{\{|\mathcal{D}_{\tau+i-1}|\leq 2pk(mh^2+T)\}}\left( \frac{\tau+i-1+Y_{\tau+i-1}+\mathbb{1}_{\{Y_{\tau+i-1}=0\}}}{n}+O_{\gamma,\beta}((mh^2+T)/n^2)\right).\]
	A repeated application of the inequality $e^x\leq 1+x+x^2$ (which is valid for all $x\in [0,1]$) together the fact that $e^{x}\geq 1+x$ for all $x\in \mathbb{R}$ yields
	\begin{equation}\label{mgfthetarig}
	\mathbb{E}[e^{rX'_{\tau+i}}|\mathcal{F}_{\tau+i-1}]
	\leq \exp\left(-r\mathbb{1}_{\{|\mathcal{D}_{\tau+i-1}|\leq 2pk(mh^2+T)\}}\frac{\tau+i-1+Y_{\tau+i-1}+\mathbb{1}_{\{Y_{\tau+i-1}\}}}{n}+O_{\gamma,\beta}(r^2)\right).
	\end{equation}
	Thus we arrive at 
	\begin{equation*}
	\mathbb{E}[e^{r(X'_{\tau+i}-E'_{\tau+i})}|\mathcal{F}_{\tau+i-1}]=e^{-rE'_{\tau+i}}\mathbb{E}[e^{rX'_{\tau+i}}|\mathcal{F}_{\tau+i-1}]\leq e^{O_{\gamma,\beta}(r^2)+O_{\gamma,\beta}(r(mh^2+T)/n^2)}.
	\end{equation*}
	It follows that
	\begin{equation*}
	\mathbb{E}[e^{r\sum_{i=1}^{T-1}(X'_{\tau+i}-E'_{\tau+i})}]
	\leq e^{O_{\gamma,\beta}(r^2)+O_{\gamma,\beta}(r(mh^2+T)/n^2)}\mathbb{E}[e^{r\sum_{i=1}^{T-2}(X'_{\tau+i}-E'_{\tau+i})}]
	\end{equation*}
	and iterating we arrive at
	\[\mathbb{E}[e^{r\sum_{i=1}^{T-1}(X'_{\tau+i}-E'_{\tau+i})}]\leq e^{O_{\gamma,\beta}(r^2T)+O_{\gamma,\beta}(T(mh^2+T)/n^2)}.\]
	Note that, as $T/n\ll 1$ and $(mh^2)/n=O(1)$ by assumption, we obtain
	\[O_{\gamma,\beta}(rT(mh^2+T)/n^2)=O_{\gamma,\beta}(r(mh^2+T)/n)=O_{\gamma,\beta}(r).\]
	Hence, when $n$ is large enough, we can bound from above the expression on the right-hand side of (\ref{beftowergnpthetarig}) by
	\[c\exp(-rh+O_{\gamma,\beta}(r^2T)),\]
	for some constant $c=c(\gamma,\beta)>0$.
	Now if $T/h\leq C$ for some constant $C>0$, we can find a small enough $r=r(C,\gamma, \beta)>0$ (which does not depend on $n$) such that $O_{\gamma,\beta}((rT)/h)\leq 1/2$ and hence we obtain
	\[c\exp\left\{-rh+O_{\gamma,\beta}(r^2T)\right\}\leq ce^{-(rh)/2}.\]
	On the other hand, if $T/h\gg 1$, then we can find a large enough constant $C=C(\gamma,\beta)$ such that, taking $r=h/(CT)\ll 1$, 
	\[c\exp\left(-rh+O_{\gamma,\beta}(r^2T)\right)\leq ce^{-h^2/(C'T)}\]
	for some other constant $C'>0$ (depending on $\gamma,\beta$). All in all (making use of (\ref{rubb1})), we conclude that there is a large enough constant $c=c(\gamma,\beta)>0$ which depends on $\gamma,\beta$ such that, for all large enough $n$, 
	\[\mathbb{P}(\exists t<T:\sum_{i=1}^{t}(X_{\tau+i}-E_{\tau+i})\geq h-1)\leq e^{-ch}\vee e^{-(ch^2)/T}+e^{-c'(mh^2+T)}.\]
	Since $mh^2\gg h^2/T$, the last exponential term is of smaller order and the desired conclusion follows.	
\end{proof}
Going back to (\ref{treegnpgen}) we see that, in order to complete the proof of Proposition \ref{Prop1Gnpint}, we still need to bound (from above) the probability
\begin{equation}\label{ksdbetarig}
\mathbb{P}(\bigcup_{t<T}\mathcal{H}_1(t)\cap \mathcal{H}_2(t))\leq \mathbb{P}(\bigcup_{t<T}\{\sum_{i=1}^t(\eta_{\tau+i}-1)\leq -h\delta\}\cap \mathcal{H}_2(t)),
\end{equation}
which we do next.
\begin{lem}
	Suppose that $(mh^2)/n=O(1)$ and $T^2/(hn)\leq \delta/2$. Then, for all large enough $n$, we have
	\[\mathbb{P}(\bigcup_{t<T}\{\sum_{i=1}^t(\eta_{\tau+i}-1)\leq -h\delta\}\cap \mathcal{H}_2(t))\leq e^{-ch}\vee e^{-(ch^2)/T},\]
	for some constant $c=c(\gamma,\beta,\delta)>0$ which depends on $\gamma,\beta$ and $\delta$.
\end{lem}
\begin{proof}
	On the event $\mathcal{H}_2(t)$ we can write
	\begin{equation*}
	\eta_{\tau+i}-1=\mathbb{1}_{\{Y_{\tau+i-1}< h\}}(\eta_{\tau+i}-1)\eqqcolon X_{\tau+i}
	\end{equation*}
	and hence, given any $r\in (0,1)$, the probability which appears in the statement of the lemma is at most
	\begin{equation*}
	\mathbb{P}(\bigcup_{t<T}\{\sum_{i=1}^t X_{\tau+i}\leq -h\delta\})=\mathbb{P}(\max_{t<T}e^{-r\sum_{i=1}^t X_{\tau+i}}\geq e^{rh\delta}).
	\end{equation*}
	As before, we need a lower bound on the number of discovered attributes at the end of step $\tau+i-1$ to control the (conditional) first moment of the $X_{\tau+i}$ given $\mathcal{F}_{\tau+i-1}$. Defining $\mathcal{W}$ as in (\ref{matcalw}) and $X'_{\tau+i}$ as in (\ref{primeprime}), we bound from above the expression on the right-hand side of the last inequality by
	\begin{equation*}
	\mathbb{P}(\max_{t<T}e^{-r\sum_{i=1}^t X'_{\tau+i}}\geq e^{rh\delta})+\mathbb{P}(\mathcal{W}^c).
	\end{equation*}
	It follows from (\ref{eithetarig}) that (as $r>0$) the process $e^{-r\sum_{i=1}^t X'_{\tau+i}}$ is a submartingale (relative to $\mathcal{F}_{\tau+t}$), whence we can use Doob's inequality to conclude that
	\begin{equation}\label{rhsrig}
	\mathbb{P}(\max_{t<T}e^{-r\sum_{i=1}^t X'_{\tau+i}}\geq e^{rh\delta})\leq e^{-rh\delta}\mathbb{E}\left[e^{-r\sum_{i=1}^{T-1}X'_{\tau+i}}\right].
	\end{equation}
	Since $\tau\leq mh^2$ (by definition) and recalling (\ref{clessidragnpthetarig}), by similar computations as those that led to (\ref{mgfthetarig}) we obtain, setting $\mathcal{C}_i\coloneqq \{Y_{\tau+i-1}<h,|\mathcal{D}_{\tau+i-1}|<2pk(mh^2+T)\}$,
	\begin{align*}
	\mathbb{E}\left[e^{-rX'_{\tau+i}}|\mathcal{F}_{\tau+i-1}\right]&=e^{r\mathbb{1}_{\mathcal{C}_i}} \mathbb{E}\left[e^{-r\mathbb{1}_{\mathcal{C}_i}\eta_{\tau_i}}|\mathcal{F}_{\tau+i-1}\right]\\
	&\leq  \exp\left(r\frac{mh^2+i+h}{n}+rO_{\gamma,\beta}(r(mh^2+T)/n^2)+O_{\gamma,\beta}(r^2)\right).
	\end{align*}
	Whence we arrive at
	\begin{equation*}
	\mathbb{E}\left[e^{-r\sum_{i=1}^{T-1}X'_{\tau+i}}\right]
	\leq \exp\left(r\frac{mh^2+T-1+h}{n}+O_{\gamma\beta}\big(\frac{r(mh^2+T)}{n^2}\big)+O_{\gamma,\beta}(r^2)\right)\mathbb{E}\left[e^{-r\sum_{i=1}^{T-2}X_{\tau+i}}\right]
	\end{equation*}
	and iterating we obtain
	\begin{equation*}
	\mathbb{E}\left[e^{-r\sum_{i=1}^{T-1}X'_{\tau+i}}\right]
	\leq\exp\left(r\frac{mh^2}{n}+r\frac{T^2}{n}+\frac{hT}{n}+O_{\gamma,\beta}(rT(mh^2+T)/n^2)+O_{\gamma,\beta}(r^2T)\right).
	\end{equation*}
	Since $T/n\ll 1, (mh^2)/n=O(1)$ and $T^2/(hn)\leq \delta/2$ we see that, for all large enough $n$, the expression on the right-hand side of (\ref{rhsrig}) is at most
	\[c\exp\left\{-(rh\delta)/2+O_{\beta}(r^2T)\right\},\]
	for some constant $c=c(\gamma,\beta)>0$.
	Now if $T/h\leq C$ for some constant $C>0$, we can find a small enough $r=r(C,\gamma,\beta,\delta)>0$ (which does not depend on $n$) such that $O_{\gamma,\beta}((rT)/h)\leq 1/4$ and hence we obtain 
	\[c\exp\left\{-(rh\delta)/2+O_{\gamma,\beta}(r^2T)\right\}\leq ce^{-(rh\delta)/4}. \]
	On the other hand, if $T/h\gg 1$, then we can find a large enough constant $C=C(\gamma,\beta)$ such that, taking $r=(h\delta)/(CT)\ll 1$, 
	\[c\exp\left\{-(rh\delta)/2+O_{\gamma,\beta}(r^2T)\right\}\leq c\exp\left\{-c'(h\delta)^2/T\right\}\] 
	for some other constant $c'=c'(\gamma,\beta)>0$ (depending on $\gamma,\beta$). All in all, we conclude that there is a large enough constant $c=c(\gamma,\beta,\delta)>0$ which depends on $\lambda,\beta$ and $\delta$ such that, for all large enough $n$,
	\begin{equation*}\label{4gnprig}
	e^{-r(h-1)}\mathbb{E}\left[e^{-r\sum_{i=1}^{T-1}X_{\tau+i}}\right]\leq e^{-ch}\vee e^{-(ch^2)/T}+\mathbb{P}(\mathcal{W}^c);
	\end{equation*}
	using (\ref{rubb1}) the desired conclusion follows.
\end{proof}
The proof of the proposition is completed combining (\ref{longprobgen}), (\ref{treegnpgen}) and the last two lemmas.
\end{proof}
Next we control the second probability on the right-hand side of (\ref{befproposint}).
\begin{prop}\label{Prop2Gnpint}
	Suppose that $(mh^3)/n=O(1)$. Then there is $\delta_0=\delta_0(\gamma,\beta)>0$ such that, if $0<\delta\leq \delta_0$, then for all large enough $n$ we have 
	\[\mathbb{P}(Y_t<h\delta \text{ }\forall t\in [mh^2])\leq \exp(-mc_0)\]
	for some constant $c_0=c_0(\delta)>0$.   
\end{prop}
\begin{proof}
	Define $I_i\coloneqq ((i-1)h^2,ih^2]\cap \mathbb{N}$ for $i\in [m]$, so that $[mh^2]=\cup_{i=1}^{m}I_i$ (a disjoint union), and let $\tau^*$ be the first time $t\in \mathbb{N}$ at which $|\mathcal{D}_t|>2pkmh^2$. 
	Noticing that $\tau^*>mh^2$ implies $\tau^*>ih^2$ for every $i\in [m]$, we can bound from above the probability that $Y_{t}< h\delta$ for all $t\in [mh^2]$ by
	\begin{equation}\label{longexp}
	\mathbb{E}\left[\mathbb{1}_{\cap_{i\in [m-1]}\{Y_{t}< h\delta \text{ }\forall t\in I_i, \tau^*>ih^2\}}\mathbb{P}(Y_{t}< h\delta \text{ }\forall t\in I_{m},\tau^*>mh^2|\mathcal{F}_{(m-1)h^2})\right]+\mathbb{P}(\tau^*\leq mh^2).
	\end{equation} 
	The last probability in (\ref{longexp}) can be bounded as in the previous proposition; in particular, it is am most $\exp(-cmh^2)$ for some constant $c=c(\gamma,\beta)>0$.	We claim that the (conditional) probability which appears within the last expectation is at most $4\delta^2$, that is
	\begin{equation}\label{MM}
	\mathbb{P}(Y_{t}<  h\delta \text{ }\forall t\in I_m,\tau^*>mh^2|\mathcal{F}_{(m-1)h^2})\leq 4\delta^2;
	\end{equation}
	we establish (\ref{MM}) along the lines of \cite{nachmias_peres:CRG_mgs}. To this end, define $M_m(t)\coloneqq Y_{(m-1)h^2+t}$ and $\mathcal{G}_m(t)\coloneqq \mathcal{F}_{(m-1)h^2+t}$ for $t\in \{0\}\cup[h^2]$; then the probability in (\ref{MM}) equals
	\[\mathbb{P}(M_m(t)<  h\delta \text{ }\forall t\in [h^2],\tau^*>mh^2|\mathcal{G}_m(0)).\]
	Define $\tau'_m$ to be the first time $t\geq 1$ such that $M_m(t)\geq h\delta$ and set $\tau_m\coloneqq \tau'_m \wedge h^2$. Note that, if $\tau^*>mh^2$, then $|\mathcal{D}_{(m-1)h^2+t}|\leq 2pkmh^2$ for $t\in [h^2]$. Thus, denoting by $\tau^*_m$ the first time $t\in [h^2]$ at which $|\mathcal{D}_{(m-1)h^2+t}|> 2pkmh^2$, we see that $\tau^*>mh^2$ implies $\tau^*_m>h^2$. Moreover, if $M_m(t)<  h\delta$ for all $t\in [h^2]$ and $\tau^*>mh^2$, then $\tau_m\wedge \tau^*_m=\tau_m=h^2$. Therefore, using (the conditional version of) Markov's inequality, we obtain
	\begin{align}\label{markovgen}
	\nonumber\mathbb{P}(M_m(t)<  h\delta \text{ }\forall t\in [h^2],\tau^*>mh^2|\mathcal{G}_{m}(0))&\leq \mathbb{P}(M_m(t)<  h\delta \text{ }\forall t\in [h^2],\tau^*_m>h^2|\mathcal{G}_{m}(0))	\\
	&\nonumber\leq \mathbb{P}(\tau_m\wedge \tau^*_m=h^2|\mathcal{G}_m(0))\\
	&\leq \frac{\mathbb{E}[\tau_m\wedge \tau^*_m|\mathcal{G}_{m}(0)]}{h^2}.
	\end{align}
	We want to use Theorem \ref{OST} to bound from above the (conditional) expected value of $\tau_m\wedge\tau^*$ given $\mathcal{G}_{m}(0)$. In particular, if we manage to show that
	\begin{equation}\label{boundmomentstopp}
	\mathbb{E}[\tau_m\wedge \tau^*_m|\mathcal{G}_{m}(0)]\leq 4(h\delta)^2,
	\end{equation}
	then the ratio in (\ref{markovgen}) is at most $4\delta^2$ and we are done. Hence we can focus on showing (\ref{boundmomentstopp}). To this end, we first give bounds on the first and second moment of the $\eta_i$. 
	
	Let $t\leq \tau_m\wedge \tau^*_m$. Recalling (\ref{lawetagnmp}) and using the tower property together with the bound $(1-x)^{\ell}\geq 1-\ell x$ (which is valid for every $x>-1\ell\in \mathbb{N}$) we immediately obtain
	\begin{equation*}\label{cl2rig}
	\mathbb{E}[\eta^2_{(m-1)h^2+t}|\mathcal{F}_{(m-1)h^2+t-1}]
	\leq 2+\gamma;
	\end{equation*}
	similarly, $\mathbb{E}[\eta_{(m-1)h^2+t}|\mathcal{F}_{(m-1)h^2+t-1}]\leq 1$. Next, observe that, if $t\leq \tau_m\wedge \tau^*_m$, since $(m-1)h^2+t-1\leq mh^2$
	\begin{equation}\label{almostfinito}
	\mathbb{E}(\eta_{(m-1)h^2+t}|N_{(m-1)h^2+t},\mathcal{F}_{(m-1)h^2+t-1})
	\geq (n-mh^2-h\delta)(pN_{(m-1)h^2+t}-(pN_{(m-1)h^2+t})^2/2),
	\end{equation}
	where the last inequality follows from the usual bound $e^{-x}\leq 1-x+x^2/2$ (valid for all $x\geq 0$). Since each $N_i$ is stochastically dominated by the $\text{Bin}(k,p)$ distribution, we obtain (recalling the definition of $p,k$)
	\begin{equation}\label{ooo}
	p^2\mathbb{E}[N^2_{(m-1)h^2+t}|\mathcal{F}_{(m-1)h^2+t-1}]\leq p^2(kp+(kp)^2)=O_{\gamma,\beta}(1/n^2).
	\end{equation} 
	Moreover, if $t\leq \tau_m\wedge \tau^*_m$ we have
	\begin{equation}\label{qqq}
	p\mathbb{E}[N_{(m-1)h^2+t}|\mathcal{F}_{(m-1)h^2+t}]=p^2(k-|\mathcal{D}_{(m-1)h^2+t-1}|)\geq p^2(k-2pkmh^2).
	\end{equation}
	Therefore taking conditional expectation $\mathbb{E}[\cdot |\mathcal{F}_{(m-1)h^2+t-1}]$ on both sides of (\ref{almostfinito}) we obtain, using (\ref{ooo}) and (\ref{qqq}), 
	\begin{equation}\label{firsrig}
	\mathbb{E}[\eta_{(m-1)h^2+t}|\mathcal{F}_{(m-1)h^2+t-1}]
	\geq 1-\frac{cmh^2}{n}
	\end{equation}
	for all large enough $n$, where $c=c(\gamma,\beta)>0$ is constant. Furthermore (as $t\leq \tau_m\wedge \tau^*_m$)
	\begin{align}\label{afterjenrig}
	\mathbb{E}[\eta^2_t|\mathcal{F}_{t-1}]&\geq 2-\frac{cmh^2}{n};
	\end{align}
	the same bounds in (\ref{firsrig}) and (\ref{afterjenrig}) hold when $Y_{(m-1)h^2+t-1}=0$. 
	Using above estimates in (\ref{firsrig}) and (\ref{afterjenrig}) we see that, if $0\leq t\leq \tau_m\wedge \tau^*_m$ and $Y_{(m-1)h^2+t-1}\geq 1$, then
	\begin{equation*}
	\mathbb{E}[M^2_{m}(t)|\mathcal{G}_{m}(t-1)]
	\geq M^2_{m}(t-1)+2(\eta_{(m-1)h^2+t}-1)+(\eta_{(m-1)h^2+t}-1)^2\geq M^2_{m}(t-1)+1-\frac{c\delta m h^3}{n}.
	\end{equation*}
	The same bound is also satisfied when $Y_{(m-1)h^2+t-1}=0$. By taking a small enough $\delta$ we see that the process defined by 
	\[M^2_{m}(t\wedge \tau_m\wedge \tau^*_m)-(t\wedge \tau_m\wedge \tau^*_m)/2, \text{ }t\in \{0\}\cup [h^2] \]
	is a submartingale. Applying Theorem \ref{OST} with the (bounded) stopping times $\sigma_1=0$ and $\sigma_2=\tau_m\wedge \tau^*_m$ we obtain
	\begin{equation*}
	\mathbb{E}[M^2_{m}(\tau_m\wedge \tau^*_m)-(\tau_m\wedge \tau^*_m)/2|\mathcal{G}_{m}(0)]
	\geq M^2_{m}(0),
	\end{equation*}
	from which it follows that
	\begin{equation*}\label{needupper}
	\mathbb{E}[\tau_m\wedge \tau^*_m|\mathcal{G}_{m}(0)]\leq 2\mathbb{E}[M^2_{m}(\tau_m\wedge \tau^*_m)|\mathcal{G}_m(0)].
	\end{equation*}
	There remains to bound the expected value on the right-hand side of the last inequality. Using once again the estimates on the moments of the $\eta_i$ given earlier and the definition of $\mathcal{G}_m(0)$ we obtain 
	\begin{equation*}
	\mathbb{E}[M^2_{m}(\tau_m\wedge \tau^*_m)|\mathcal{G}_{m}(0)]\leq 2(h\delta))^2
	\end{equation*}
	for all large enough $n$. It follows that $\mathbb{E}[\tau_m\wedge \tau^*_m|\mathcal{F}_{(m-1)h^2}]\leq 4(h\delta)^2$, from which we obtain (\ref{boundmomentstopp}). Therefore the probability on the left-hand side of (\ref{markovgen}) is at most $4\delta^2$, establishing (\ref{MM}). Substituting the bound of (\ref{MM}) into (\ref{longexp}) and iterating $m$ times we obtain the desired inequality.	
\end{proof}
Let $m\coloneqq \lceil A^{3/5}\rceil$ and $h\coloneqq \lceil n^{1/3}/A^{1/5}\rceil$. Combining (\ref{excgnpint}), (\ref{befproposint}) together with Propositions \ref{Prop1Gnpint} and \ref{Prop2Gnpint} (and recalling that, by assumption, $A=O(n^{5/12})$) we see that $\mathbb{P}(|\mathcal{C}_{\max}(\mathbb{G}_{n,k,p})|<n^{2/3}/A)\leq \exp\big(-cA^{3/5}\big)$, as desired.

\subsection{The $\mathbb{G}(n,\beta,\lambda)$ model}\label{quantum}

In this last section we consider the (critical) \textit{quantum} version of the Erd\H{o}s-R\'enyi random graph, which we defined non-rigorously in the introductory section.

We refer the reader to \cite{dembo_et_al:component_sizes_quantum_RG} and references therein for explanations on such terminology, moving on instead to the precise description of the model, as it is given in \cite{dembo_et_al:component_sizes_quantum_RG} (the description given here is taken verbatim from \cite{de_ambroggio:upper_component_sizes_crit_RGs}).

Let $\mathbb{S}_{\beta}$ denote the circle of length $\beta>0$ and set $\mathbb{G}_{n,\beta}\coloneqq [n]\times \mathbb{S}_{\beta}$, so that to each element $v\in [n]$ we associate the copy $\mathbb{S}^v_{\beta}=\{v\}\times \mathbb{S}_{\beta}$ of $\mathbb{S}_{\beta}$. In this way, a point in $\mathbb{G}_{n,\beta}$ has \textit{two} coordinates: its site $v\in [n]$ and the location on the circle $t\in \mathbb{S}_{\beta}$. Then we create, within each $\mathbb{S}^v_{\beta}$, finitely many holes, according to independent Poisson point processes $(\mathcal{P}_v:v\in [n])$ of intensity $\lambda>0$. As a result, each punctured circle $\mathbb{S}^v_{\beta}\setminus \mathcal{P}_v$ can be expressed as the (disjoint) union of $k_v\in \mathbb{N}$ connected intervals, namely
\begin{equation}\label{decomp}
\mathbb{S}^v_{\beta}\setminus \mathcal{P}_v=\bigcup_{j=1}^{k_v}I^v_j,
\end{equation}
where $|\mathcal{P}_v|$ equals $k_v$ (unless $|\mathcal{P}_v|=0$, in which case $k_v=1$ and the circle $\mathbb{S}^v_{\beta}$ remains intact). 

We add edges in the following (random) manner. To each unordered pair $\{u,v\}$ of (distinct) vertices $u, v\in [n]$, we associate a further circle $\mathbb{S}^{u,v}_{\beta}$ of length $\beta$ and a Poisson point process $\mathcal{L}_{u,v}=\mathcal{L}_{v,u}$ on $\mathbb{S}^{u,v}_{\beta}$ with intensity $1/n$. The processes $\mathcal{L}_{u,v}$ are assumed to be independent for different $(u,v)$ and also independent of $(\mathcal{P}_z:z\in [n])$. 

Then, two intervals $I^u_j,I^v_{\ell}$ ($u\neq v$) of the decomposition (\ref{decomp}) are considered to be directly connected if there exists some $t\in \mathcal{L}_{u,v}$ such that $(u,t)\in I^u_j$ and $(v,t)\in I^v_{\ell}$. In simple words, two intervals $I^u_j$ and $I^v_{\ell}$ (with $u\neq v$) of the decomposition (\ref{decomp}) are linked by an edge if they both include a time $t$ at which the Poisson point process $\mathcal{L}_{u,v}$ (on $\mathbb{S}^{u,v}_{\beta}$) jumps.

The resulting random graph, which is obtained after connecting intervals in the way we have just described, is denoted by $\mathbb{G}_{n,\beta,\lambda}$.

As remarked in \cite{dembo_et_al:component_sizes_quantum_RG}, setting $\mathcal{P}\coloneqq \cup_{u\in [n]}\mathcal{P}_u$ (a \textit{finite} collection of points), the decomposition $\mathbb{G}(n,\beta,\lambda)=\CC_1\cup\dots\cup\CC_N$ into maximal components is well-defined and, moreover, each point $x=(v,t)\in \mathbb{S}^v_{\beta}$ is almost surely not in $\mathcal{P}$. Hence, the notion of component $\CC(x)$ of $x$ is also well-defined and throughout the size of a component is the number of \textit{intervals} it contains.
The following equivalent description of the $\mathbb{G}_{n,\beta,\lambda}$ model will be convenient for our purposes \cite{dembo_et_al:component_sizes_quantum_RG}. Let $\lambda>0$ and assign to each vertex $v\in [n]$ a circle of length $\theta\coloneqq \lambda \beta>0$ with an independent rate one Poisson process $\mathcal{P}_v$ of holes on it. The links between each pair of punctured circles are created by means of iid Poisson processes of intensity $(\lambda n)^{-1}$ (which are also independent of the rate $1$ Poisson processes of holes). The resulting random graph, denoted by $\mathbb{G}_{n,\theta}$, has the same law as $\mathbb{G}_{n,\beta,\lambda}$ (see \cite{dembo_et_al:component_sizes_quantum_RG}).

As we did for the binomial model and the random graph obtained by bond percolation on a random $d$-regular graph, in order to establish the upper bound stated in Theorem \ref{mainthm} for the $\mathbb{G}_{n,\theta}=_d\mathbb{G}_{n,\beta,\lambda}$ model, the idea is to use an algorithm to sequentially construct an instance of this random graph, interval by interval.

In the following description, which is taken from \cite{dembo_et_al:component_sizes_quantum_RG} the vertices (which we recall are circles) have first been labelled with numbers $1,\dots,n$ and, at the end of each time step $t\in \mathbb{N}$ of the algorithm, exactly one interval becomes \textit{explored}.\\

\textbf{Algorithm 4}.  At time $t=0$, we fix the vertex $w_0=1$ and choose a point $s_0$ uniformly at random on $\mathbb{S}^{w_0}_{\theta}$. The point $(w_0,s_0)$ is declared \textit{active}, whereas the remaining space is declared \textit{neutral}. Hence, denoting by $\mathcal{A}_t$ the set of active points at the end of step $t\in \mathbb{N}_0$, we have $|\mathcal{A}_0|=1$. For every $t\in \mathbb{N}$, the algorithm works as follows. 
\begin{itemize}
	\item [(a)] If $|\mathcal{A}_{t-1}|\geq 1$, we choose an active point $(w_t,s_t)$ whose vertex has the smallest index among all active points. In case of a tie (which may occur since each $\mathbb{S}^v_{\theta}$ could contain several active points), we choose the active point which chronologically appeared earlier than the others on the same vertex. 
	\item [(b)] If $|\mathcal{A}_{t-1}|=0$ and there exists at least one neutral circle, we choose $w_t$ to be the neutral vertex with the smallest index and select $s_t$ uniformly at random on $\mathbb{S}^{w_t}_{\theta}$. Then we declare the point $(w_t,s_t)$ active.
	\item [(c)] If $|\mathcal{A}_{t-1}|=0$ and there is no neutral circle left, we choose $w_t$ to be the vertex of smallest index among the vertices having some neutral part and select $s_t$ uniformly at random on the neutral part of $\mathbb{S}^{w_t}_{\theta}$. Then we declare the point $(w_t,s_t)$ active.
	\item [(d)] If $|\mathcal{A}_{t-1}|=0$ and there is no neutral part available on any circle, then we terminate the procedure.
\end{itemize}
The algorithm proceeds as follows. Using iid $\text{Exp}(1)$ random variables $J^t_-,J^t_+$ and writing $s^1_t,s^2_t$ for the points on the neutral space of $\mathbb{S}^{w_t}_{\theta}$ around $s_t$, we extract out of the \textit{maximal} neutral interval $\{w_t\}\times (s^1_t,s^2_t)$ around the (active) point $(w_t,s_t)$ the sub-interval $I_t\coloneqq \{w_t\}\times \tilde{I}_t$, where
\[\tilde{I}_t\coloneqq (s^1_t \vee (s_t-J^t_-),s^2_t\wedge (s_t+J^t_+)).\]
If $\mathbb{S}^{w_t}_{\theta}$ is neutral (apart from active points), we take $s^1_t=-\infty=-s^2_t$ (so that $\tilde{I}_t= (s_t-J^t_-,s_t+J^t_+)$) and $\tilde{I}_t=\mathbb{S}_{\theta}$ (the whole circle of length $\theta$) if $J^t_-+J^t_+\geq \theta$. We then remove from the list of active points $\mathcal{A}_{t-1}$ \textit{all} those points which got encompassed by the interval $I_t$, including the active point currently under investigation $(w_t,s_t)$. The links in the graph connected to all such points, other than  $(w_t,s_t)$, are considered to be \textit{surplus edges}. 

The connections of $I_t$ are then constructed in the following manner. 
\begin{itemize}
	\item [(i)] For $i\neq w_t$, we regard $\tilde{I}_t$ as a subset of $\mathbb{S}^{w_t,i}_{\theta}$ and sample the process of links $\mathcal{L}_{w_t,i}$ for times $r$ restricted to $\tilde{I}_t$. That is, we run the process $(\mathcal{L}_{w_t,i}(r))_{r\in \tilde{I}_t}$ and, denoting by $j^i_1,\dots,j^i_{L_i}\in \tilde{I}_t$ the jumps of this process (if any), then we create a link between each pair of points $(w_t,j^i_{\ell}),(i,j^i_{\ell})$, $\ell\in [L_i]$.
	\item [(ii)] Subsequently, we erase all links between $I_t$ and points on already explored intervals, and record each $(i,j^i_{\ell})$ ($\ell\in [L_i], i\neq w_t$) on the neutral space as an active point, labelled with the time (order) of its registration (see part (a) of the algorithm, where such ordering is used).
\end{itemize}
After examining all the connections from $I_t$, we declare such interval \textit{explored} and increase $t$ by one. (We notice that the procedure halts after finitely many steps, because the number of intervals is finite.)\\

Denote by $\zeta_t$ the \textit{new} links that are formed during step $t$ in \textbf{Algorithm 4}. Moreover, we let $\text{Spl}_t$ denote the number of surplus edges found by the end of the first $t$ steps in the procedure. Define 
\begin{equation}\label{etaqrg}
\eta_t\coloneqq \zeta_t-(\text{Spl}_t-\text{Spl}_{t-1})
\end{equation}
and note that, since at those times $t$ where $|\mathcal{A}_{t-1}|\geq 1$ we add to the list of active points the $\zeta_t$ new links found during step $t$ but we remove those (active) points which are encompassed by $I_t$ (including $(w_t,s_t)$), we have the recursion
\begin{itemize}
	\item $|\mathcal{A}_{t}|=|\mathcal{A}_{t-1}|+\eta_t-1$, if $|\mathcal{A}_{t-1}|\geq 1$;
	\item $|\mathcal{A}_{t}|=\eta_t$, if $|\mathcal{A}_{t-1}|=0$.
\end{itemize}
To bound (from above) the probability that $\mathcal{C}_{\max}(\mathbb{G}(n,\beta,\lambda))$ contains less than $n^{2/3}/A$ nodes, we follow the strategy adopted in \cite{dembo_et_al:component_sizes_quantum_RG}. Specifically, we consider a more \textit{restrictive} exploration process, which leads to shorter (positive) excursions compared to those of $|\mathcal{A}_t|$. The advantage of such a new procedure is that it yields a simpler (conditional) distribution for the random variable $\eta_i$. 

\paragraph{Reduced exploration process.} This procedure, after creating the first active point on each node $u\in [n]$, it \textit{voids} all space on that vertex apart from the relevant interval around that point, thus sequentially producing components with no more intervals than does the original exploration process. Specifically, at each step $t\in \mathbb{N}$ of this procedure, either $I_t$ is originated from a point that was active at the end of step $t-1$ or, if $|\mathcal{A}_{t-1}|=0$, it is formed from a point $(w_t,s_t)$ with $s_t$ selected uniformly at random on the completely neutral circle $\mathbb{S}^{w_t}_{\theta}$, if such a circle exists; otherwise, we halt the procedure. Moreover, this new (restrictive) exploration process keeps \textit{at most} one link from $I_t$ to any (as of yet) never visited (in particular, neutral) circle $\mathbb{S}^i_{\theta}$, erasing all other connections which are being formed during step $t$ by the original exploration process, namely \textbf{Algorithm 4}.

\begin{obs}
	Since in the above procedure we void all space on a node after creating the first active point on it, we see that the algorithm runs for $n$ steps.
\end{obs}

It is important to notice that such restrictive exploration process \textit{ does not} have surplus links 
and, moreover, its number of active points $|\mathcal{A}^*_t|$ also satisfies $|\mathcal{A}^*_0|=1$ and a recursion of the type:
\begin{itemize}
	\item $|\mathcal{A}^*_{t}|=|\mathcal{A}^*_{t-1}|+\eta^*_t-1$, if $|\mathcal{A}^*_{t-1}|\geq 1$;
	\item $|\mathcal{A}^*_{t}|=\eta^*_t$, if $|\mathcal{A}^*_{t-1}|=0$,
\end{itemize}
with $\eta^*_t$ representing the number of new active points created at time $t$. 

Let us denote by $Y_t$ and $U_t$ the number of active points and neutral vertices at the end of step $t\in [n]\cup\{0\}$ in the reduced exploration process, respectively, so that
\[Y_t=|\mathcal{A}^*_t| \text{ and }U_t=n-t-Y_t.\]
Denote by $\mathcal{F}_t$ the $\sigma$-algebra generated by the information collected by the exploration process until the end of step $t$. Then, conditional on $\mathcal{F}_{t-1}$ \textit{and} $J_t$, the random variable $\eta^*_t$ depends on the past until time $t-1$ only through $Y_{t-1}$ and
\begin{equation}\label{lawetagntheta}
\eta^*_t=_d\text{Bin}(U_{t-1}-\mathbb{1}_{\{Y_{t-1}=0\}}, 1-e^{-J_t/(\lambda n)}),
\end{equation}
where we recall that the $J_i$ are iid $\Gamma_{\theta}(2,1)$-distributed random variables. 

As before, the component sizes are given by $t^*_i-t^*_{i-1}$, where we set $t^*_0\coloneqq 0$ and let $t^*_i\coloneqq \min\{t\geq t^*_{i-1}+1:Y_t=0\}$ for $i\geq 1$ (as long as there is at least one neutral circle).

Using the reduced exploration process we can bound
\begin{equation}\label{afterres}
\mathbb{P}(t_i-t_{i-1}<n^{2/3}/A \text{ }\forall i)\leq \mathbb{P}(t^*_i-t^*_{i-1}<n^{2/3}/A \text{ }\forall i).
\end{equation}
Then, arguing as we did for the random graphs $\mathbb{G}(n,p)$ and $\mathbb{G}_{n,d,p}$, we take $1\ll h=h(n)\in \mathbb{N}$, $m\in \mathbb{N}$ (possibly dependent on $n$), and denote by $\delta\in (0,1)$ a constant that we specify later. Then, setting $T\coloneqq \lceil n^{2/3}/A\rceil$, we bound from above the probability on the right-hand side of (\ref{afterres}) by 
\begin{equation}
\mathbb{P}(t^*_i-t^*_{i-1}<T \text{ }\forall i, \exists t\in [mh^2]:Y_t\geq h\delta)+\mathbb{P}(Y_t<h\delta \text{ }\forall t\in [mh^2]).
\end{equation}
The two terms which appear in the last display are bounded (from above) by means of the following two propositions. 
\begin{prop}\label{Prop1Gntheta}
	Suppose that $(mh^2)/n=O(1)$ and $T^2/(hn)\leq \delta/2$. Then, for all large enough $n$, we have
	\begin{equation*}
	\mathbb{P}(t^*_i-t^*_{i-1}<T \text{ }\forall i, \exists t\in [mh^2]:Y_t\geq h\delta)\leq  e^{-ch}\vee e^{-(ch^2)/T},
	\end{equation*}
	for some constant $c=c(\lambda,\beta,\delta)>0$ which depends on $\lambda,\beta$ and $\delta$.
\end{prop}
\begin{proof}
	The proof follows the exact same steps carried out in the proof of Proposition \ref{Prop1Gnpint} and therefore we omit the details; we just compute the relevant expectations needed to adapt the argument.
	
	Recall from (\ref{lawetagntheta}) that, conditional on $\mathcal{F}_{\tau+i-1}$, the random variable $\eta^*_{\tau+i}$ is distributed as
	\begin{equation}\label{lawetabeta}
	\text{Bin}(n-\tau-(i-1)-Y_{\tau+i-1}-\mathbb{1}_{\{Y_{\tau+i-1}=0\}},1-e^{-J_{\tau+i}/(\lambda n)}).
	\end{equation}
	Therefore, using the classical bounds $e^x\geq 1+x$ (which is valid for all $x \in \mathbb{R}$) we see that
	\[1-e^{-J_{\tau+i}/(\lambda n)}\leq \frac{J_{\tau+i}}{\lambda n}\]
	and so we obtain  
	\begin{equation}\label{eitheta}
	E_{\tau+i}\leq  n\mathbb{E}[J_{\tau+i}]/(\lambda n)-1\leq 0.
	\end{equation}
	Using (\ref{lawetabeta}) together with the bound $e^{-x}\leq 1-x+x^2/2$ (which is valid for all $x\geq 0$) we obtain
	\begin{equation}\label{clessidragnptheta}
	\mathbb{E}[\eta^*_{\tau+i}|\mathcal{F}_{\tau+i-1}]
	\geq 1-\frac{\tau+i-1+Y_{\tau+i-1}+\mathbb{1}_{\{Y_{\tau+i-1}=0\}}}{n}-O_{\lambda,\beta}(n^{-1}),
	\end{equation}
	whence
	\[E_{\tau+i}\geq - \frac{\tau+i-1+Y_{\tau+i-1}+\mathbb{1}_{\{Y_{\tau+i-1}=0\}}}{n}-O_{\lambda,\beta}(n^{-1}).\]
	Moreover, a repeated application of the inequality $e^x\leq 1+x+ x^2$ (which is valid for all $x\in [0,1]$) together with the fact that $e^{x}\geq 1+x$ for all $x\in \mathbb{R}$ yields
	\begin{equation}\label{mgftheta}
	\mathbb{E}[e^{rX_{\tau+i}}|\mathcal{F}_{\tau+i-1}]
	\leq \exp\left(-r\frac{\tau+i-1+Y_{\tau+i-1}+\mathbb{1}_{\{Y_{\tau+i-1}\}}}{n}+O_{\lambda,\beta}(r^2)\right).
	\end{equation}
	Thus we obtain 
	\begin{equation*}
	\mathbb{E}[e^{r(X_{\tau+i}-E_{\tau+i})}|\mathcal{F}_{\tau+i-1}]=e^{-rE_{\tau+i}}\mathbb{E}[e^{rX_{\tau+i}}|\mathcal{F}_{\tau+i-1}]\leq e^{O_{\lambda,\beta}(r^2)+O_{\lambda,\beta}(1/n)}.
	\end{equation*}
	Since $\tau\leq mh^2$ (by definition) and recalling (\ref{clessidragnptheta}), by similar computations as the ones that led to (\ref{mgftheta}) we obtain 
	\begin{equation*}
	\mathbb{E}\left[e^{-rX_{\tau+i}}|\mathcal{F}_{\tau+i-1}\right]
	\leq  \exp\left(r\frac{mh^2+i+h}{n}+O_{\lambda,\beta}(r/n)+O_{\lambda,\beta}(r^2)\right).
	\end{equation*}
	The result follows arguing as in Proposition \ref{Prop1Gnpint}.
\end{proof}
\begin{prop}\label{Prop2Gntheta}
	Suppose that $(mh^3)/n=O(1)$. Then there is $\delta_0=\delta_0(\lambda,\beta)>0$ such that, if $0<\delta\leq \delta_0$, then for all large enough $n$ we have
	\[\mathbb{P}(Y_t<h\delta \text{ }\forall t\in [mh^2])\leq \exp(-mc_0)\]
	for some constant $c_0=c_0(\delta)>0$.
\end{prop}
\begin{proof}
	The proof follows the same steps carried out in the proof of Proposition \ref{Prop2Gnpint} and hence we omit most of the details.
	
	First of all note that, recalling (\ref{lawetagntheta}) and using the tower property together with the bound $e^{x}\geq 1+x$ (which is valid for every $x$) we immediately obtain
	\begin{align*}\label{cl2}
	\mathbb{E}[(\eta^*)^2_{t}|\mathcal{F}_{t-1}]=\mathbb{E}_{\mathcal{F}_{t-1}}[\mathbb{E}((\eta^*)^2_{t}|\mathcal{F}_{t-1},J_t)]
	&\leq \mathbb{E}_{\mathcal{F}_{t-1}}[\mathbb{E}(\text{Bin}^2(n,1-e^{-J_t/(\lambda n)})|J_t)]\\
	&\leq \mathbb{E}_{\mathcal{F}_{t-1}}[n(1-e^{-J_t/(\lambda n)})]+\mathbb{E}_{\mathcal{F}_{t-1}}[(n(1-e^{-J_t/(\lambda n)}))^2]\\
	&\leq C
	\end{align*}
	where $C=C(\lambda,\beta)\coloneqq 1+\lambda^{-2}\mathbb{E}[J^2_1]$ depends on $\lambda,\beta$; similarly, $\mathbb{E}[\eta^*_{t}|\mathcal{F}_{t-1}]\leq 1$ (at all times $t$). Next, suppose that $1\leq Y_{t-1}<h\delta$. Then we can bound, using (\ref{clessidragnptheta}),
	\begin{equation}\label{STT}
	\mathbb{E}[\eta^*_{t}|Y_{t-1}]\geq 1-\frac{2mh^2}{n}-O_{\lambda,\beta}(1/n),
	\end{equation}
	for all large enough $n$. Moreover, using the inequalities $e^x\geq 1+x$ and $e^{-x}\leq 1-x+x^2/2$ (which are valid for all $x\in \mathbb{R}$ and $x\geq 0$, respectively) we see that (as $1\leq Y_{t-1}<h\delta$ and $t\leq mh^2$)
	\begin{align}\label{afterjen}
	\nonumber\mathbb{E}[(\eta^*_t)^2|\mathcal{F}_{t-1}]&=\mathbb{E}_{\mathcal{F}_{t-1}}[\mathbb{E}[\text{Bin}^2(n-(t-1)-Y_{t-1},1-e^{-J_t/(\lambda n)})|J_t]]\\
	&\geq (n-mh^2-h\delta)\left(\frac{1}{ n}-O_{\lambda,\beta}(1/n^2)\right)+(n-mh^2-h\delta)^2\mathbb{E}[(1-e^{-J_t/(\lambda n)})^2].
	\end{align}
	Now since
	\begin{equation*}
	\mathbb{E}[(1-e^{-J_t/(\lambda n)})^2]\geq \mathbb{E}[1-e^{-J_t/(\lambda n)}]^2\geq n^{-2}\left(1-\frac{\mathbb{E}[J^2_t]}{2\lambda^2n}\right)^2\geq n^{-2}(1-O_{\lambda,\beta}(1/n)),
	\end{equation*}
	we see that the last term on the right-hand side of (\ref{afterjen}) is at least 
	\[ 1-\frac{mh^2+h\delta}{n}-O_{\lambda,\beta}(1/n),\]
	provided $n$ is sufficiently large. Since the second-last term on the right-hand side of (\ref{afterjen}) satisfies
	\begin{equation*}
	(n-mh^2-h\delta)\left(\frac{1}{n}-O_{\lambda,\beta}(1/n)\right)\geq 1-\frac{mh^2+h\delta}{n}-O_{\lambda,\beta}(1/n)
	\end{equation*}
	and $h\delta \leq mh^2$, we conclude that
	\begin{equation}\label{secgntheta}
	\mathbb{E}[(\eta^*_t)^2|Y_{t-1}]\geq 2-4\frac{mh^2}{n}-O_{\lambda,\beta}(1/n)\geq 2-\frac{cmh^2}{n}
	\end{equation}
	for some constant $c=c(\lambda,\beta)>0$; the same bounds in (\ref{STT}) and (\ref{secgntheta}) hold when $Y_{t-1}=0$. The desired result then follows arguing as in Proposition \ref{Prop2Gnpint}
\end{proof}
The required upper bound on the probability that $|\mathcal{C}_{\max}(\mathbb{G}(n,\beta,\lambda))|<n^{2/3}/A$ can be obtained combining Propositions \ref{Prop1Gntheta} and \ref{Prop2Gntheta}, using the same values for $h,m$ employed when analysing the $\mathbb{G}(n,k,p)$ model.

\section*{Acknowledgements}
The author thanks Matthew Roberts for a useful discussion on this problem when the former was a PhD student at the University of Bath.

\bibliographystyle{plain}
\def\cprime{$'$}

\end{document}